\documentclass[11pt]{amsart}

\usepackage{amsmath,amsthm,verbatim,amssymb,amsfonts,amscd, graphicx}
\usepackage{graphics}
\usepackage{mathtools}
\usepackage{tikz}
\usepackage{xifthen} 
\usepackage{commath} 
\usepackage{booktabs} 
\usepackage{xcolor} 
\usepackage{xfrac} 
\usepackage{array} 
\usepackage{multicol} 
\usetikzlibrary{positioning}
\usepackage{enumitem}
\usepackage{pgfplots}
\setlist[enumerate]{label=\rm{(\arabic*)}}
\setlist[enumerate,2]{label=\rm({\it\roman*})}
\setlist[itemize]{label=\raisebox{0.25ex}{\tiny$\bullet$}}

\usepackage[backref, colorlinks, linktocpage, citecolor = blue, linkcolor = blue]{hyperref} 

\theoremstyle{plain}

\newtheorem{theorem}{Theorem}
\newtheorem{lemma}{Lemma}[section]
\newtheorem{corollary}[lemma]{Corollary}
\newtheorem{proposition}[lemma]{Proposition}
\newtheorem{question}[lemma]{Question}
\theoremstyle{definition}
\newtheorem{definition}[lemma]{Definition}

\theoremstyle{remark}
\newtheorem{remark}[lemma]{Remark}
\newtheorem{observation}[lemma]{Observation}
\newtheorem{example}[lemma]{Example}

\newcommand{\PPP}{\mathbb{P}}
\newcommand{\RRR}{\mathbb{R}}

\newcommand{\NNN}{\mathbb{N}}
\newcommand{\AAA}{\mathbb{A}}
\newcommand{\CCC}{\mathbb{C}}
\newcommand{\FFF}{\mathbb{F}}

\newcommand{\divv}{\mathrm{div}}
\newcommand{\sd}{~|~}

\usepackage{scalerel}

\DeclarePairedDelimiter{\ceil}{\lceil}{\rceil}
\DeclarePairedDelimiter{\floor}{\lfloor}{\rfloor}

\DeclareMathOperator{\Ima}{Im}
\DeclareMathOperator{\Aut}{Aut}
\DeclareMathOperator{\GL}{GL}

\begin{document}
\subjclass[2010]{Primary 14H20 14H50; Secondary 14J26 14E05}
\keywords{plane curves, simple singularities, blow-up, Hirzebruch surface}

\title{Plane curves of fixed bidegree and their $A_k$-singularities}
\author[Julia Schneider]{Julia Schneider}
\address{Julia Schneider, Universit\"{a}t Basel, Departement Mathematik und Informatik, Spiegelgasse $1$, CH-$4051$ Basel, Switzerland}
\email{julia.noemi.schneider@unibas.ch}
\maketitle

\begin{abstract}
  We provide a tool how one can view a polynomial on the affine plane of bidegree $(a,b)$ -- by which we mean that its Newton polygon lies in the triangle spanned by $(a,0)$, $(0,b)$ and the origin -- as a curve in a Hirzebruch surface having nice geometric properties. As an application, we study maximal $A_k$-singularities of curves of bidegree $(3,b)$ and find the answer for $b\leq 12$.
\end{abstract}

\tableofcontents

\section{Introduction}\label{section:Introduction}

We study algebraic curves (not necessarily reduced) on the affine plane $\mathbb{A}^2(\mathbb{C})$ that have a singularity of type $A_k$, which means that there is an analytical local isomorphism such that the curve is given by $y^2-x^{k+1}=0$ in a neighbourhood of the singular point (c.f. Definition~\ref{definition:SingularitiesOfTypeAk}). We ask:
\begin{question}\label{question:MaximalAkForDegreed}
  For $d\geq 1$, what is the maximal $k$ such that there exists a curve of degree $d$ that has an $A_k$-singularity?
\end{question}
We denote this by $N(d)$ and can give answers for small $d$:
\begin{center}
  \begin{tabular}{ >{$}r<{$} || >{$}c<{$} | >{$}c<{$} | >{$}c<{$}| >{$}c<{$}| >{$}c<{$}| >{$}c<{$}| >{$}c<{$} }
    d & 1 & 2 & 3 & 4 & 5 & 6 & 7...\\  \hline
    N(d) & 0 & 1 & 3 & 7 & 12 & 19 & ? \\
  \end{tabular},
\end{center}
where an explicit equation for $d=5$ can be found in~\cite{wall_1996}, and the result for $d=6$ is by Yang, who gave a classification of all simple singularities of sextic curves in \cite{yang_1996}. (Note that the answers of $N(2)$, $N(3)$ and $N(4)$ differ if we only consider irreducible curves.)
The difficulty of the question increases rapidly for larger values of $d$, so the asymptotic behaviour is studied and bounds for \[\alpha=\limsup\frac{2\,N(d)}{d^2}\] are wanted, where we multiplied by $2$ to obtain nicer numbers, as it is often done in the literature.
Gusein-Zade and Nekhoroshev \cite{gusein-zade_2000} found in 2000 that $1.5\geq\alpha\geq\frac{15}{14}\simeq 1.07142$ and in the same year, Cassou-Nogu\`es and Luengo \cite{cassou-luengo_2000} refined the lower bound to $8-4\sqrt3\simeq1.07179$.
A decade passed until Orevkov \cite{orevkov_2012} improved it even further to $\frac{7}{6}=1.1\overline 6$ in 2012.

Question~\ref{question:MaximalAkForDegreed} can also be approached through fixing a bidegree instead of the degree.
We say that a polynomial $F$ (or equivalently, the curve in $\mathbb{A}^2(\mathbb{C})$ defined by its zero set) has \textit{bidegree $(a,b)$} if its Newton polygon lies in the triangle spanned by $(a,0)$, $(0,0)$ and $(0,b)$.
In particular, a polynomial is of bidegree $(d,d)$ if and only if it is of degree at most $d$.
So we generalize Question~\ref{question:MaximalAkForDegreed}:
\begin{question}\label{question:MaximalAkForBidegreeab}
  For $(a,b)\in \mathbb{N}^2$, what is the maximal $k$ such that there is a curve in $\mathbb{A}^2(\mathbb{C})$ of bidegree $(a,b)$ with an $A_k$-singularity?
\end{question}
Similar to above, we denote this by $N(a,b)$.
For instance, one finds $N(1,b)=0$ for all $b$, and fixing $a=2$ yields $N(2,b)=b-1$ (c.f. Example~\ref{example:N1b} respectively Lemma~\ref{lemma:BoundsForThe2SectionCase}).
We have studied the case where $a=3$ and found the following values of $N(3,b)$:
\begin{theorem}\label{theorem:TheTheoremTable}
  For small $b$, $N(3,b)$ is given by the following table:
  \begin{center}
    \begin{tabular}{ >{$}r<{$} || >{$}c<{$} | >{$}c<{$} | >{$}c<{$}| >{$}c<{$}| >{$}c<{$}| >{$}c<{$}| >{$}c<{$} | >{$}c<{$} | >{$}c<{$} | >{$}c<{$} | >{$}c<{$} }
      b &  3 & 4 & 5 & 6 & 7 & 8 & 9 & 10 & 11 & 12\\ \hline
      N(3,b) & 3 & 5 & 7 & 8 & 10 & 12 & 13 & 15 & 17 & 18 \\
    \end{tabular}.
  \end{center}
  Moreover, for $b\geq 4$ there are irreducible polynomials that achieve the maximal singularities.
\end{theorem}
Studying polynomials of bidegree $(a,b)$ is interesting on its own, however it could also help to determine the asymptotical behaviour of $N(d)$, thanks to the following result.
\begin{proposition}[Orevkov {\cite{orevkov_2012}}]\label{proposition:Orevkov}
  If $N(a,b)+1\geq k$, then $\alpha\geq \frac{2k}{ab}$.
\end{proposition}
And in fact, it \textit{does} help: Luengo found $N(4,6)\geq13$, Orevkov applied this proposition and got $\alpha\geq\frac{7}{6}$. Initially, we hoped to improve this bound, but the best we get with our results is $N(3,11)=17$ yielding $\alpha\geq\frac{12}{11}\simeq 1.09$.

In fact, using $N(3,b)$ it is not possible to obtain a better lower bound than Orevkov's $\alpha\geq \frac{7}{6}$:
A result in knot theory by Feller~\cite{feller_2016} about the existence of algebraic cobordisms between the torus knots $T_{2,k+1}$ and $T_{3,b}$ gives an upper bound for $N(3,b)$ if $b$ is no multiple of $3$ (namely $\frac{5b-4}{3}$, c.f. Lemma~\ref{lemma:KnotUpperBound}), which implies that $\frac{2(N(3,b)+1)}{3b}<\frac{7}{6}$ for all $b>13$ (c.f. Lemma~\ref{lemma:NotAsGoodAsOrevkov}).
Theorem~\ref{theorem:TheTheoremTable} provides the result for $b\leq12$.\\

In Section~\ref{section:PolynomialVsDivisor}, we provide an algebro-geometric tool how to translate a polynomial $F$ on the affine plane of bidegree $(a,b)$ into a curve $C$ on a Hirzebruch surface $\FFF_m$, where $m$ is the integer such that $b=am-r$ for some $0\leq r<a$.
We will call $C$ the \textit{$(a,am)$-divisor of $F$} (c.f. Definition~\ref{definition:a-am-divisor}).
Lemma~\ref{lemma:PolynomialVsDivisor} gives, in particular, a geometric description of $C$.

As an application of the discribed correspondence, we will extensively study bidegree $(3,b)$ in Sections~\ref{section:ToBe},~\ref{section:OrNotToBe} and~\ref{section:KnotTheory}  and prove Theorem~\ref{theorem:TheTheoremTable} in the end. Since Theorem~\ref{theorem:TheTheoremTable} is a question about the maximal $k$, its proof consists of two parts: existence (Section~\ref{section:ToBe}) and non-existence (Section~\ref{section:OrNotToBe}).

Both sections start with providing a ``recipe'' how one can translate the curve $C$ with a large $A_k$-singularity into a curve on $\FFF_{1}$ (or $\FFF_0$) that is (almost) smooth (c.f. Section~\ref{subsection:OrNotToBe--TheRecipe} and Remark~\ref{remark:3SectionGetsEventuallySmooth}), and vice versa (c.f. Section~\ref{section:ToBe--Recipe}).
This is achieved with a chain of elementary links centered at the singularity (c.f. Definition~\ref{definition:SingularChainOfNLinks}), respectively the inverse of this birational map (c.f. Definition~\ref{definition:transversalChainOfnLinks}).

In Section~\ref{section:ToBe--Ingredients} we present ``ingredients'' in $\PPP^2$ that we can blow-up to $\FFF_1$ and then use the recipe to ``cook'' large singularities, giving a lower bound for $N(3,b)$.

In Section~\ref{subsection:OrNotToBe--NonExistence}, an upper bound for $N(3,b)$ is given by showing that the ``ingredients'' that are required by the recipe do not exist if $k$ is too large with respect to $b$.
However, the best upper bound is already given in the result by Feller mentioned above, to which we give a short introduction in Section~\ref{section:KnotTheory}.
This is the reason why we only present the non-existence of configurations in the case where $b$ is a multiple of $3$, since the other computations do not add any value to this paper.

Theorem~\ref{theorem:TheTheoremTable} stops at $b=12$, because the computations are done case-by-case and get more and more tedious.
It would be interesting to have a family of curves of bidegree $(3,b)$ with increasing $b$ that have maximal $A_k$-singularity.

Moreover, in Remark~\ref{remark:SpecialWeierstrassPoints} we observe a connection to Weierstrass points on $\PPP^1\times\PPP^1$, recently introduced in \cite{maugesten-moe_2018}.

\bigskip

I thank Peter Feller for introducing knot theory and its connections to algebraic geometry to me, Mattias Hemmig for helpful discussions, and my PhD advisor J\'er\'emy Blanc for guidance.

\section{Preliminaries}\label{section:Preliminaries}

In Section~\ref{subsection:HirzebruchSurfaces} we recall what a Hirzebruch surface is and fix our notation.
Then, we introduce singularities of type $A_k$ in Section~\ref{subsection:SingularitiesOfTypeAk} and observe what happens when blowing up such a singularity.
We continue to provide some easy bounds in Section~\ref{subsection:BabyBounds}.
To conclude the preliminaries, we introduce in Section~\ref{subsection:Link} the notions ``$p$-link'' and ``cofiberedness'' on a Hirzebruch surface and explain why these are of interest in our setting.

\subsection{Hirzebruch surfaces}\label{subsection:HirzebruchSurfaces}
Let $m\geq0$ be an integer.
The $m$-th \textit{Hirzebruch surface} $\FFF_m$ is defined to be the quotient of $\left(\AAA^2\setminus\{(0,0)\}\right)^2$ modulo the following equivalence relation on it: The two points $\left((x_0,x_1),(y_0,y_1)\right)$ and $\left((x_0',x_1'),(y_0',y_1')\right)$ are equivalent if there are $\lambda,\mu\in \CCC^*$ such that \[\left((x_0,x_1),(y_0,y_1)\right)=\left((\mu x_0',\lambda^{-m}\mu x_1'),(\lambda y_0',\lambda y_1')\right).\]
We denote the equivalence class of $\left((x_0,x_1),(y_0,y_1)\right)$ by $[x_0:x_1;y_0:y_1]$.

We will always see $\FFF_m$ as a $\PPP^1$-bundle over $\PPP^1$, via $[x_0:x_1;y_0:y_1]\mapsto[y_0:y_1]$.
The fibers are then the curves of the form $\alpha y_0+\beta y_1=0$ for $[\alpha:\beta]\in\PPP^1$.

The section given by $x_1=0$ is denoted by $S_-$ and has self-intersection $-m$.
On the other hand, we denote by $S_+$ the section given by $x_0=0$, which has self-intersection $m$.

We can visualize this surface with the following figure, where the number in the bracket denotes the self-intersection:
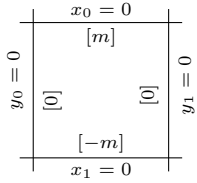
\begin{center}
\begin{tikzpicture}[scale=0.9]
  \def\abw{0.2}
  \def\fontsz{\tiny}
   \draw (-3-\abw+0,0) to (-1+\abw+0,0);
   \draw (-3+0,0-\abw) to (-3+0,2+\abw);
   \node (a) at (-2+0,0){};
    \node () [below=-0.2cm of a] {\fontsz $x_1=0$};
    \node () [above=-0.2cm of a] {\fontsz $[-m]$};
   \node (b) at (-3+0,1){};
    \node () [above=0cm of b, rotate=90] {\fontsz $y_0=0$};
    \node () [below=0cm of b, rotate=90] {\fontsz $[0]$};

  \draw (-3-\abw+0,2) to (-1+\abw+0,2);
  \draw (-1+0,0-\abw) to (-1+0,2+\abw);
  \node (c) at (-2+0,2){};
      \node () [above=-0.2cm of c] {\fontsz $x_0=0$};
      \node () [below=-0.2cm of c] {\fontsz $[m]$};
  \node (d) at (-1+0,1){};
     \node () [below=-0.2cm of d,rotate=90] {\fontsz $y_1=0$};
     \node () [above=-0.2cm of d,rotate=90] {\fontsz $[0]$};

\end{tikzpicture}
\end{center}

Moreover, $\FFF_m\setminus\left(\{x_i=0\}\cup\{y_j=0\}\right)\simeq\AAA^2$ for $i,j=0,1$.
Hence we can embedd $\AAA^2$ into $\FFF_m$ for example with $\iota_m:\AAA^2\hookrightarrow\FFF_m, (x,y)\mapsto[x:1;y:1]$, as the following picture illustrates:
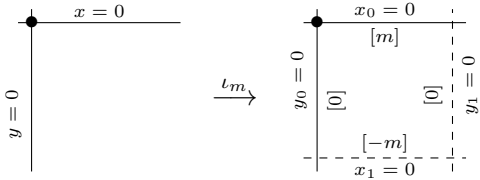
\begin{center}
\begin{tikzpicture}[scale=0.9]
  \def\abw{0.2}
  \def\fontsz{\tiny}
   \draw (-3-\abw+0,2) to (-1+\abw+0,2);
   \draw (-3+0,0-\abw) to (-3+0,2+\abw);
   \node (a) at (-2+0,2){};
    \node () [above=-0.2cm of a] {\fontsz $x=0$};
   \node (b) at (-3+0,0.5){};
    \node () [above=0cm of b, rotate=90] {\fontsz $y=0$};
    \draw (0,1) node{$\stackrel{\iota_m}\longrightarrow$};
    \node () at (-3,2){$\bullet$};
  \begin{scope}[xshift=120]
       \draw[dashed] (-3-\abw+0,0) to (-1+\abw+0,0);
       \draw[] (-3+0,0-\abw) to (-3+0,2+\abw);
       \node (a) at (-2+0,0){};
        \node () [below=-0.2cm of a] {\fontsz $x_1=0$};
        \node () [above=-0.2cm of a] {\fontsz $[-m]$};
       \node (b) at (-3+0,1){};
        \node () [above=0cm of b, rotate=90] {\fontsz $y_0=0$};
        \node () [below=0cm of b, rotate=90] {\fontsz $[0]$};

      \draw (-3-\abw+0,2) to (-1+\abw+0,2);
      \draw[dashed] (-1+0,0-\abw) to (-1+0,2+\abw);
      \node (c) at (-2+0,2){};
          \node () [above=-0.2cm of c] {\fontsz $x_0=0$};
          \node () [below=-0.2cm of c] {\fontsz $[m]$};
      \node (d) at (-1+0,1){};
         \node () [below=-0.2cm of d,rotate=90] {\fontsz $y_1=0$};
         \node () [above=-0.2cm of d,rotate=90] {\fontsz $[0]$};
         \node () at (-3,2){$\bullet$};
   \end{scope}
\end{tikzpicture}
\end{center}

Recall that for each divisor $D$ on $\FFF_m$ there are integers $a,b$ with $D\sim aS_-+bf$.
If $D$ is effective, then $a$ and $b$ are at least $0$.
Moreover, if $D$ is irreducible and $D\neq S_-$, we have $0\leq D\cdot S_-=-am+b$ and hence $b\geq am$.

\begin{definition}\label{definition:aSection}
  Let $a\geq 1$ be an integer and let $C\subset\FFF_m$ be an effective divisor not containing any fibers.
  We call $C$ an $a$-\textit{section} if $C\cdot f=a$ for fibers $f\in\FFF_m$.
\end{definition}
Note that a $1$-section is a smooth, irreducible curve isomorphic to $\PPP^1$. Therefore, it will simply be called a \textit{section}.

\bigskip
\subsection{Singularities of type $A_k$}\label{subsection:SingularitiesOfTypeAk}

\begin{definition}\label{definition:SingularitiesOfTypeAk}
  Let $C$ be a curve on a smooth surface.
  A point $s\in C$ is called \textit{singularity of type} $A_k$ for some integer $k\geq1$ if there are local analytic coordinates in which $C$ around $s$ is given by the equation $y^2-x^{k+1}=0$.
  We sometimes abuse the notation and say that a smooth point has a ``singularity'' of type $A_0$.
\end{definition}

For small $k$, the real part of an $A_k$-singularity looks locally like the following:
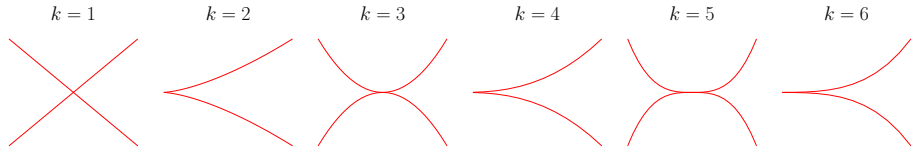
\begin{center}
\begin{tikzpicture}[scale=0.3]
\begin{axis}[title=\Huge{$k=1$}, axis lines=none]
\addplot[color=red]{x};
\addplot[color=red]{-x};
\end{axis}
\end{tikzpicture}%
\begin{tikzpicture}[scale=0.3]
\begin{axis}[title=\Huge{$k=2$},axis lines=none, samples=200]
\addplot[color=red]{sqrt(x^3)};
\addplot[color=red]{-sqrt(x^3)};
\end{axis}
\end{tikzpicture}%
\begin{tikzpicture}[scale=0.3]
\begin{axis}[title=\Huge{$k=3$},axis lines=none]
\addplot[color=red]{x^2};
\addplot[color=red]{-x^2};
\end{axis}
\end{tikzpicture}%
\begin{tikzpicture}[scale=0.3]
\begin{axis}[title=\Huge{$k=4$},axis lines=none]
\addplot[color=red]{sqrt(x^5)};
\addplot[color=red]{-sqrt(x^5)};
\end{axis}
\end{tikzpicture}%
\begin{tikzpicture}[scale=0.3]
\begin{axis}[title=\Huge{$k=5$},axis lines=none]
\addplot[color=red]{x^3};
\addplot[color=red]{-x^3};
\end{axis}
\end{tikzpicture}%
\begin{tikzpicture}[scale=0.3]
\begin{axis}[title=\Huge{$k=6$},axis lines=none]
\addplot[color=red]{sqrt(x^7)};
\addplot[color=red]{-sqrt(x^7)};
\end{axis}
\end{tikzpicture}%
\end{center}

\begin{remark}\label{remark:OddEvenSingularities}
  If $k$ is odd, then the singularity is reducible and we call it a \textit{node}.
  But there are irreducible curves with such a singularity.

  If $k$ is even, then we call it a \textit{cusp} and the singularity is irreducible. Thus, it cannot arise as the intersection of two curves.
\end{remark}

\begin{example}\label{example:ReducibleCubicPolynomialWithA3}
  Consider the polynomial $F=y(y-x^2)\in\CCC[x,y]$ of degree $3$ and the map%
  \begin{align*}
    \varphi:\CCC^2&\to\CCC^2\\
    (x,y)&\mapsto\left(i\sqrt{2}\,x,y-x^2\right).
  \end{align*}%
  This map sends $F$ onto $(y-x^2)(y+x^2)=y^2-x^4$.
  So $V(F)$ is sent onto $y^2-x^4=0$, which corresponds to an $A_3$-singularity.
  The map $\varphi$ is holomorphic and it has a holomorphic inverse, given by $(u,v)\mapsto\left(\frac{-i}{\sqrt{2}}\,u,v-\frac{1}{2}u^2\right)$.
  Therefore, there is a local analytic isomorphism that sends $V(F)$ onto $y^2-x^4=0$ and so $F$ (respectively $V(F)$) has an $A_3$-singularity at the origin.

  The following picture illustrates the (real part of the) zero set of $F$: \begin{center}
  \begin{tikzpicture}[scale=0.3,baseline=-2]
     \begin{axis}[title=\Huge{$y(y-x^2)=0$},axis lines=none]
       \addplot[color=red]{x^2};
       \addplot[color=red]{0};
     \end{axis}
   \end{tikzpicture}%
  \end{center}
\end{example}

The following result by Wall in \cite{wall_2004} shows that singularities of type $A_k$ arise naturally.

\begin{lemma}[Theorem 2.2.7 in \cite{wall_2004}]\label{lemma:Multiplicity2ImpliesAk}
  Let $C$ be a curve with a point of multiplicity $2$ that is reduced at the point.
  Then that point is a singularity of type $A_k$ for some $k\geq 1$.
\end{lemma}

We describe the notion of singularities of type $A_k$ using blow-ups.

\begin{lemma}\label{lemma:OneBlowUpOrDownOfAk}
  Let $\pi:Y\to X$ be the blow-up centered at $s\in X$ with exceptional divisor $E\subset Y$.
  Let $C\subset X$ be a curve reduced at $s$ and let $\tilde C\subset Y$ be its strict transform.
  \begin{enumerate}
    \item\label{item:OneBlowUpOrDownOfAk--equivalence} The following are equivalent:
      \begin{enumerate}
        \item\label{item:OneBlowUpOrDownOfAk--multiplicity2} $m_s(C)=2$
        \item\label{item:OneBlowUpOrDownOfAk--intersectionEC} $\tilde C\cdot E=2$
        \item\label{item:OneBlowUpOrDownOfAk--Ak} $C$ has an $A_k$-singularity at $s$ for some $k\geq 1$.
      \end{enumerate}
    \item\label{item:OneBlowUpOrDownOfAk--IToIII} If \ref{item:OneBlowUpOrDownOfAk--equivalence} holds, then the following statements hold:
      \begin{enumerate}[label=(\Roman*)]
        \item\label{item:OneBlowUpOrDownOfAk--I} $\tilde C\cap E$ contains two distinct points if and only if $k=1$.
        \item\label{item:OneBlowUpOrDownOfAk--II} $\tilde C \cap E=\{s'\}$ where $s'\in \tilde C$ is smooth if and only if $k=2$.
        \item\label{item:OneBlowUpOrDownOfAk--III} $\tilde C\cap E=\{s'\}$ where $s'\in\tilde C$ is a singular point of type $A_{k-2}$ if and only if $k\geq 3$.
      \end{enumerate}
  \end{enumerate}
  Moreover, in case \ref{item:OneBlowUpOrDownOfAk--II} the exceptional divisor $E$ and $\tilde C$ are tangent at $s'$.
\end{lemma}

\begin{proof}
  The statements \ref{item:OneBlowUpOrDownOfAk--multiplicity2} and \ref{item:OneBlowUpOrDownOfAk--intersectionEC} are equivalent, and ``\ref{item:OneBlowUpOrDownOfAk--multiplicity2}$\implies$\ref{item:OneBlowUpOrDownOfAk--Ak}'' is exactly Lemma~\ref{lemma:Multiplicity2ImpliesAk}.
  As an $A_k$-singularity is locally given by $y^2-x^{k+1}=0$, which is of multiplicity $2$ in the origin, statement \ref{item:OneBlowUpOrDownOfAk--Ak} implies \ref{item:OneBlowUpOrDownOfAk--multiplicity2}.
  So \ref{item:OneBlowUpOrDownOfAk--equivalence} holds.

  To establish \ref{item:OneBlowUpOrDownOfAk--IToIII}, it is enough to consider $C$ in a neighbourhood of $(0,0)\in\AAA^2$, where it is given by the equation $y^2-x^{k+1}=0$.
  Locally, the blow-up is given by $\pi:\AAA^2\to\AAA^2, (x,y)\mapsto (x,xy)$, so the exceptional divisor $E$ is defined by $x=0$.
  The preimage $\pi^{-1}(C)$ is given by $x^2(y^2-x^{k-1})=0$, hence the strict transform $\tilde C$ is given by $y^2-x^{k-1}=0$, which corresponds to an $A_{k-2}$-singularity if $k\geq 2$.
  If $k\geq 3$, then $k-2\geq1$, so it is a singular point and we have \ref{item:OneBlowUpOrDownOfAk--III}.
  If $k=2$, then $k-2=0$, so it is a smooth point and we have \ref{item:OneBlowUpOrDownOfAk--II}.
  If $k=1$, we have that $\tilde C$ is given by $y^2-1=0$, so the exceptional divisor intersects the exceptional divisor $E$ at two points, namely at $(0,1)$ and $(0,-1)$, and we have \ref{item:OneBlowUpOrDownOfAk--I}.

  Since the three cases  \ref{item:OneBlowUpOrDownOfAk--I}, \ref{item:OneBlowUpOrDownOfAk--II} and  \ref{item:OneBlowUpOrDownOfAk--III} cannot occur simultaneously, we have proved the ``if and only if''-statements in \ref{item:OneBlowUpOrDownOfAk--IToIII}.

  To conclude the proof, note that in case \ref{item:OneBlowUpOrDownOfAk--II} there is only one point on $\tilde C\cap E$ but by \ref{item:OneBlowUpOrDownOfAk--intersectionEC} we have $\tilde C\cdot E=2$.
  Thus, $I_{s'}(\tilde C,E)=2$, so $E$ and $\tilde C$ are tangent at $s'$.
  This achieves the proof.
\end{proof}

\begin{corollary}\label{corollary:SeveralBlowUpOfAk}
  Let $C$ be a curve on a smooth surface $X$ with an $A_k$-singularity at some point $s$.
  Then there exists a sequence $\pi:Y\to X$ of $\ceil{\frac{k}{2}}$ blow-ups such that the strict transform $\tilde C$ is smooth at the intersection with $\pi^{-1}(s)$.
\end{corollary}

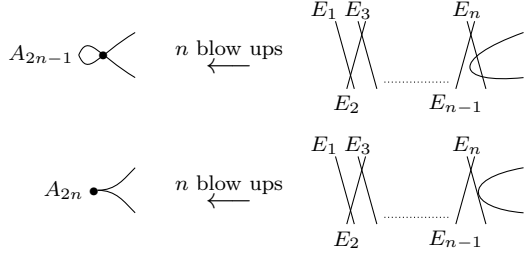
\begin{figure}[h]
  \begin{center}
    \begin{tikzpicture}[xscale=0.5,yscale=0.3,baseline=(b.base)]
      \node (b) at (0.15,-0.05){\tiny{$\bullet$}};
      \draw (-0.5,0) ..controls (-0.25,-1) and (0,0).. (1,1);
      \draw (-0.5,0) ..controls (-0.25,1) and (0,0).. (1,-1);
      \node (nothing) at (-2.5,0){};
      \draw (-1.5,0) node{\scriptsize{$A_{2n-1}$}};
      \draw (3.5,0) node{$\stackrel{n\text{ blow ups}}\longleftarrow$};
      \begin{scope}[xshift=180]
        \draw (0,1.5)--(0.5,-1.5);
        \draw (0.3,-1.5)--(0.8,1.5);
        \draw (0.6,1.5)--(1.1,-1.5);
        \draw[densely dotted] (1.3,-1.2) --(3,-1.2);
        \draw (3.2,-1.5)--(3.7,1.5);
        \draw (3.5,1.5)--(4,-1.5);
        \draw (5,1) ..controls (3,-0.2) and (3.2,-1.3).. (5,-1);
        \draw (-0.3,2) node{\scriptsize{$E_1$}};
        \draw (0.3,-2.2) node{\scriptsize{$E_2$}};
        \draw (0.6,2) node{\scriptsize{$E_3$}};
        \draw (3.2,-2.2) node{\scriptsize{$E_{n-1}$}};
        \draw (3.5,2) node{\scriptsize{$E_n$}};
      \end{scope}
    \end{tikzpicture}
    \begin{tikzpicture}[xscale=0.5,yscale=0.3,baseline=(b.base)]
      \node (b) at (-0.1,-0.05){\tiny{$\bullet$}};
      \draw (0,0) ..controls (0.5,0) and (0.7,0.5).. (1,1);
      \draw (0,0) ..controls (0.5,0) and (0.7,-0.5).. (1,-1);
      \node (nothing) at (-2.5,0){};
      \draw (-0.9,0) node{\scriptsize{$A_{2n}$}};
      \draw (3.5,0) node{$\stackrel{n\text{ blow ups}}\longleftarrow$};
      \begin{scope}[xshift=180]
        \draw (0,1.5)--(0.5,-1.5);
        \draw (0.3,-1.5)--(0.8,1.5);
        \draw (0.6,1.5)--(1.1,-1.5);
        \draw[densely dotted] (1.3,-1.2) --(3,-1.2);
        \draw (3.2,-1.5)--(3.7,1.5);
        \draw (3.5,1.5)--(4,-1.5);
        \draw (5,1) ..controls (3.3,0.5) and (3.5,-0.8).. (5,-1);
        \draw (-0.3,2) node{\scriptsize{$E_1$}};
        \draw (0.3,-2.2) node{\scriptsize{$E_2$}};
        \draw (0.6,2) node{\scriptsize{$E_3$}};
        \draw (3.2,-2.2) node{\scriptsize{$E_{n-1}$}};
        \draw (3.5,2) node{\scriptsize{$E_n$}};
      \end{scope}
    \end{tikzpicture}
  \end{center}
  \caption{Illustration of Corollary~\ref{corollary:SeveralBlowUpOfAk}. Above: $k$ odd, below: $k$ even. For $i=1,\ldots,n$, the exceptional divisor of the $i$-th blow-up is denoted by $E_i$.}
  \label{figure:SeveralBlowUpOfAk}
\end{figure}

\begin{proof}
  If $k=1$ or $k=2$ we are done with applying Lemma~\ref{lemma:OneBlowUpOrDownOfAk} once.
  If $k\geq 3$ let $n=\ceil{\frac{k}{2}}\geq 2$.
  By applying Lemma~\ref{lemma:OneBlowUpOrDownOfAk} $n$ times, we get a sequence of $n$ blow-ups as described in this lemma. Figure~\ref{figure:SeveralBlowUpOfAk} depicts the situation.
\end{proof}

\subsection{Baby bounds}\label{subsection:BabyBounds}

As a warm-up, we give bounds for $N(1,b)$, $N(2,b)$, and $N(3,3)$ in this section and remark that an irreducible curve of genus $g$ has at most an $A_{2g}$-singularity (c.f. Lemma~\ref{lemma:GenusUpperBoundForIrreducibleDivisors}).

\begin{example}\label{example:N1b}
  Let us prove that $N(1,b)=0$ for all integers $b$.
  Let $F$ be a (reduced) polynomial of bidegree $(1,b)$, so $F=\lambda x+ G(y)$, where $G\in\CCC[y]$ is a polynomial in one variable.
  By applying a translation, we may assume that $F$ has an $A_k$-singularity in $(0,0)$.
  If $\lambda\neq 0$, we can parametrize the curve given by the zero set of $F$ by $x=-\lambda^{-1}G(y)$, so it is a smooth curve.
  If $\lambda=0$ then $F=G(y)$ is a polynomial in one variable, and reduced by hypothesis.
  So $F$ has no multiple factors, and hence no singular points.
  Therefore, $F$ is again smooth.
\end{example}

\begin{example}\label{example:a2LowerBound}
  The polynomial $F=x^2-y^{2m-1}$ is of bidegree $(2,2m-1)$ and has an $A_{2m-2}$-singularity.
  Hence $N(2,2m)\geq N(2,2m-1)\geq 2m-2$.

  Note that for bidegree $(2,2)$ the bound is not sharp, since $F=xy$ has an $A_1$-singularity.
  In fact, we will see in Example \ref{example:2SectionWithEvenbLowerBound} that it is not sharp for all bidegree $(2,2m)$.
\end{example}

\begin{lemma}\label{lemma:GenusUpperBoundForIrreducibleDivisors}
  Let $C$ be an irreducible divisor on a smooth surface with a singularity of type $A_k$.
  Then, $k\leq2g(C)$, where $g(C)$ denotes the arithmetic genus of $C$.
\end{lemma}

\begin{proof}
  By Lemma~\ref{lemma:OneBlowUpOrDownOfAk} and Corollary~\ref{corollary:SeveralBlowUpOfAk}, there are $n=\ceil{\frac{k}{2}}$ infinitely near points with multiplicity $2$.
  This yields \[g(C)\geq\frac{1}{2}\sum m_p(C)(m_p(C)-1)\geq n,\] where the sum runs over all singular points of $C$, including infinitely near ones.
  Hence, $k\leq 2n\leq 2g(C)$.
\end{proof}

\begin{lemma}\label{lemma:UpperBoundDegree3}
  Let $F$ be a polynomial of bidegree $(3,3)$ (that is of degree at most $3$) with an $A_k$-singularity.
  Then, $k\leq3$.
  Moreover, if $F$ is irreducible, then $k\leq 2$.
\end{lemma}

\begin{proof}
  If the degree of $F$ is one or two, then we already know that $k\leq 2$.
  So we assume that its degree is $3$.

  If $F$ is irreducible, we can homogenize it to an irreducible polynomial $F'$ of degree $3$ in $\CCC[x,y,z]_3$.
  Hence the curve $C=V(F')\subset\PPP^2$ has arithmetic genus $1$.
  Lemma~\ref{lemma:GenusUpperBoundForIrreducibleDivisors} gives $k\leq 2g(C)=2$.

  So let us assume that $F$ is reducible.
  Then the polynomial $F$ is either the product of $3$ linear terms, which can give at most an $A_1$-singularity, or the product of a linear and a quadratic term.
  Let $L\subset\AAA^2$ be the zero set of the linear term and let $Q\subset\AAA^2$ be the zero set of the quadratic term.
  There are two possibilities: \begin{enumerate}
    \item $L$ and $Q$ intersect at two points, and then the intersection is transversal.
    This gives an $A_1$-singularity.
    \item $L$ and $Q$ intersect at one point, and $L$ is a tangent to $Q$.
    This gives an $A_3$-singularity, as in Example \ref{example:ReducibleCubicPolynomialWithA3}.
  \end{enumerate}
\end{proof}

\begin{corollary}\label{corollary:N33}
  $N(3,3)=3$.
\end{corollary}

\begin{proof}
  The upper bound comes from Lemma~\ref{lemma:UpperBoundDegree3} and the existence of such a singularity comes from Example \ref{example:ReducibleCubicPolynomialWithA3}.
\end{proof}

\subsection{Links and cofiberedness}\label{subsection:Link}
Recall Corollary~\ref{corollary:SeveralBlowUpOfAk} and Figure~\ref{figure:SeveralBlowUpOfAk}.
Instead of blowing up the singular point $n$ times, we will do one blow-up at a time in the following way.

\begin{definition}\label{definition:pLink}
  Let $m$ be an integer, let $p\in\FFF_m$ be a point and let $f$ be the fiber containing it.
  A birational map $\pi:\FFF_m\dashrightarrow\FFF_{m\pm1}$ that is the blow-up centered at $p$ followed by the contraction of the strict transform of $f$ to a point $s\in\FFF_{m\pm1}$ will be called \textit{$p$-link from $\FFF_m$} with \textit{inverse point} $s$.
\end{definition}

\begin{remark}
  Note that a $p$-link $\pi$ with inverse point $s$ is a birational map $\FFF_m\dashrightarrow\FFF_{m+1}$ if $p\in S_-$, and it is a birational map $\FFF_m\dashrightarrow\FFF_{m-1}$ if $p\notin S_-$.
  It is uniquely determined by $p$ up to composition with an automorphism of $\FFF_{m\pm1}$.
  Moreover, its inverse $\pi^{-1}:\FFF_{m\pm1}\dashrightarrow \FFF_m$ is a $s$-link of $\FFF_{m\pm1}$ with inverse point $p$, which justifies the denotation of ``inverse point''.
\end{remark}


\begin{definition}\label{definition:Cofibered}
  Let $C\subset\FFF_m$ be a divisor and let $p$ and $p'$ be two distinct points in $\FFF_m$.
  We say that $p'$ is a \textit{cofibered point of $p$ with respect to $C$} (or $p$ and $p'$ are \textit{$C$-cofibered}) if $p$ and $p'$ lie on $C$ and on the same fiber.
\end{definition}

The following lemma shows that when studying $3$-sections with an $A_k$-singularity, cofiberedness is a natural property.

\begin{lemma}\label{lemma:3-SectionSituation}
  Let $C$ be a $3$-section in $\FFF_m$ that has an $A_k$-singularity at a point $s\in C$ for some $k\geq 3$.
  Then, $s$ has a cofibered point $p\in C$.
\end{lemma}

\begin{proof}
  Since we have $C\cdot f=3$ and $I_s(C,f)\geq 2$ (because an $A_k$-singulartiy has multiplicity $2$), there is either one more point $p\in C\cap f$ with $I_p(C,f)=1$, which means that $s$ and $p$ are $C$-cofibered, or $I_s(C,f)=3$.
  As $C$ does not contain any fiber by assumption, the latter case is not possible :
  It means that there is a point $s'\in \tilde C\cap \tilde f$ in the exceptional divisor of the blow-up centered at $s$ with $I_{s'}(\tilde C,\tilde f)=1$.
  If $k\geq 3$, this is not possible because $s'\in\tilde C$ is a singular point.
\end{proof}

Figure~\ref{figure:3-sectionSituation} depicts the situation of the above lemma.
The reader is urged to keep these pictures in mind when thinking about $3$-sections with a large $A_k$-singularity.

\begin{figure}[h]
  \begin{center}
    \begin{tikzpicture}[xscale=0.5,yscale=0.3,baseline=(b.base)]
      \draw (0,-1) -- (0, 7);
      \node (b) at (0,0){};
      \draw (-2,6) ..controls (-2,5.5) and (-1,5).. (0,5);
      \draw (0,5) ..controls (0.5,5) and (1.5,6).. (2,5);
      \draw (0,5) ..controls (0 ,5.3) and (2,4).. (2,5);
      \draw (0,5) ..controls (-2,4) and (-2,0).. (2,0);
      \draw (0.3,2.5) node{\scriptsize{$f$}};
      \draw (0.3,-0.2) node{\scriptsize{$p$}};
      \draw (0.3,5.6) node{\scriptsize{$s$}};
      \draw (-0.5,5.7) node{\scriptsize{$A_{k}$}};
      \begin{scope}[xshift=180]
        \draw (0,-1) -- (0, 7);
        \node (b) at (0,0){};
        \draw (-2,6) ..controls (-2,5.5) and (-1,5).. (0,5);
        \draw (0,5) ..controls (-2,4) and (-2,0).. (2,0);
        \draw (0.3,2.5) node{\scriptsize{$f$}};
        \draw (0.3,-0.2) node{\scriptsize{$p$}};
        \draw (0.3,5) node{\scriptsize{$s$}};
        \draw (-0.5,5.7) node{\scriptsize{$A_{k}$}};
      \end{scope}
    \end{tikzpicture}%
  \end{center}
  \caption{A $3$-section with a large $A_k$-singularity (left: $k$ odd, right: $k$ even).}
  \label{figure:3-sectionSituation}
\end{figure}
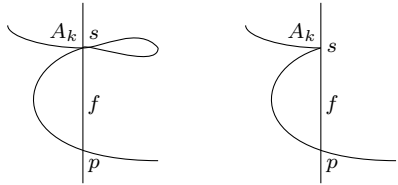

%

\section{Polynomial in $\AAA^2$ vs. Divisor in $\FFF_m$}\label{section:PolynomialVsDivisor}

In this section we study polynomials $F$ in $\AAA^2$ of bidegree $(a,am-r)$ for some $a,m\geq1$ and $0\leq r<a$ and divisors $C\sim aS_+$ in $\FFF_m$.
We obtain a correspondence between such polynomials and divisors in Lemma~\ref{lemma:PolynomialVsDivisor}, which is the main statement of this section.
As an application, we find an upper bound for $A_k$:
If the corresponding divisor (we say: ``$(a,am)$-divisor'', see Definition~\ref{definition:a-am-divisor}) is irreducible, we compute the genus and get Lemma~\ref{lemma:GenusUpperBoundForPolynomials}.
If the divisor is reducible, the bound is stated in Lemma~\ref{lemma:ReducibleDivisorsNotInteresting} in the case where $a=3$.

\bigskip

\begin{lemma}\label{lemma:TriangleToPolynomial}
  Let $a,m,r$ be integers such that $m\geq 1$ and $0\leq r<a$. A polynomial $F\in\CCC[x,y]$ is of bidegree $(a,am-r)$ if and only if it is of the form \[F=\sum_{i=0}^a x^i \sum_{j=0}^{N(i)}a_{ij}y^j,\] where $N(i)=m(a-i)-r+\floor{\frac{ir}{a}}$ for all $i=0,\ldots,a$ and $a_{ij}\in\CCC$.
\end{lemma}

\begin{proof}
  Observe that a pair $(i,j)\in\NNN^2$ lies in the triangle spanned by $(a,0), (0,0)$ and $(0,am-r)$ if and only if $0\leq i\leq a$ and $aj+(am-r)i\leq a(am-r)$. The latter inequality can be reformulated into \[j\leq\frac{(a-i)(am-r)}{a}=(1-\frac{i}{a})(am-r)=m(a-i)-r+\frac{ir}{a}.\] As $j$ is an integer this is equivalent to $j\leq N(i)$ and the lemma follows.
\end{proof}

Recall from Section \ref{subsection:HirzebruchSurfaces} the embedding \begin{align*}
  \iota_m:\AAA^2&\hookrightarrow\FFF_m,\\
   (x,y)&\mapsto[x:1;y:1].
\end{align*}

\begin{definition}\label{definition:a-am-divisor}
  Let $F\in\CCC[x,y]$ be of bidegree $(a,am)$.
  An effective divisor $C\subset\FFF_m$ with $C\sim aS_+$ such that $C\mid_{\iota_m(\AAA^2)}$ corresponds to the zero set of $F$ in $\AAA^2$ will be called a $(a,am)$\textit{-divisor of} $F$.
\end{definition}

The following lemma shows that it exists uniquely, hence it will  be called \textit{the} $(a,am)$-divisor of $F$.

Before stating the lemma, we give a short overview of it:
Parts~\ref{item:PolynomialVsDivisor--PD} and~\ref{item:PolynomialVsDivisor--DP} show the correspondence of a polynomial $F$ of bidegree $(a,am)$ and an $(a,am)$-divisor $C$.
Then, the equivalence of~\ref{item:PolynomialVsDivisor--PolyType} and~\ref{item:PolynomialVsDivisor--Divisibility} translates the meaning of having bidegree $(a,am-r)$ into a condition on the equation of the zero set of $C$. This condition is then stated in a geometric manner in~\ref{item:PolynomialVsDivisor--A} and~\ref{item:PolynomialVsDivisor--B} for $r=1$ respectively $r=2$.

\begin{lemma}\label{lemma:PolynomialVsDivisor}
  Let $a\geq1$ and $m\geq1$ be two integers. \begin{enumerate}
    \item\label{item:PolynomialVsDivisor--PD} Let $F$ be a polynomial of bidegree $(a,am)$.
    Then, there is a unique divisor $C\subset\FFF_m$ which is an $(a,am)$-divisor of $F$.
    \item\label{item:PolynomialVsDivisor--DP} Let $C\subset\FFF_m$ be a divisor with $C\sim aS_+$. Then, there exists a polynomial $F$ (unique up to multiplication with a constant) of bidegree $(a,am)$ such that $C$ is its $(a,am)$-divisor.
    Moreover, if $C$ is irreducible, then so is $F$.
  \end{enumerate}
  If \ref{item:PolynomialVsDivisor--PD} and / or \ref{item:PolynomialVsDivisor--DP} hold, let $G=\sum_{i=0}^ax_0^i\,x_1^{a-i}\,G_{m(a-i)}(y_0,y_1)$ be a polynomial on $\FFF_m$ whose zero set is $C$, where the $G_{m(a-i)}$ are homogenenous of degree $m(a-i)$, and let $r$ be an integer with $0\leq r<a$.
  The following are equivalent: \begin{enumerate}[label=(\roman*)]
    \item\label{item:PolynomialVsDivisor--PolyType} $F$ is a polynomial of bidegree $(a,am-r)$,
    \item\label{item:PolynomialVsDivisor--Divisibility} $y_1^{r-\floor{\frac{ir}{a}}}$ divides $G_{m(a-i)}$ for all $i=0,\ldots,a$.
  \end{enumerate}
  Moreover, for small $r$ we have the following statements: \begin{enumerate}[label=(\Alph*)]
    \item\label{item:PolynomialVsDivisor--A} For $r=1$, \ref{item:PolynomialVsDivisor--PolyType} holds if and only if $I_p(y_1,C)=a$, or $y_1=0$ is a component of $C$,
    \item\label{item:PolynomialVsDivisor--B} for $r=2$ and $a= 3$, \ref{item:PolynomialVsDivisor--PolyType} holds if and only if \begin{enumerate}[label=(\alph*)]
      \item\label{item:PolynomialVsDivisor--a} $m_p(C)=3$, or
      \item\label{item:PolynomialVsDivisor--b}  $m_p(C)=2$ and $C$ has only one tangent direction at $p$, namely the one given by $y_1=0$,
    \end{enumerate}
  \end{enumerate}
  where $p=[0:1;1:0]$.
\end{lemma}

\begin{proof}
  We show \ref{item:PolynomialVsDivisor--PD}.
  Thanks to Lemma~\ref{lemma:TriangleToPolynomial} we can write $F=\sum_{i=0}^a x^i \sum_{j=0}^{m(a-i)}a_{ij}y^j$.
  We homogenize it to a polynomial $G$ of degree $(a,0)$ on $\FFF_m$ with $\iota_m$ and obtain \[G:=\sum_{i=0}^ax_0^i\, x_1^{a-i}\underbrace{\sum_{j=0}^{m(a-i)}a_{ij}\, y_0^j \, y_1^{m(a-i)-j}}_{=:G_{m(a-i)}}.\]
  As $x_0^a$ is also of degree $(a,0)$, we have $\frac{G}{x_0^a}\in k(\FFF_m)$ and so \begin{align*}
    \divv\left(\frac{G}{x_0^a}\right)&=\divv(G)-a\, \divv(x_0)\\
    &\sim \divv(G)-aS_+,
  \end{align*} and finally we set $C$ to be the effective divisor $C:=\divv(G)\sim aS_+$.

  Observe that $G(x,1,y,1)=F(x,y)$ and so $\iota_m$ is an isomorphism between the zero set of $F$ in $\AAA^2$ and $C\mid_{\iota_m(\AAA^2)}$, so $C$ is a $(a,am)$-divisor of $F$.

  To show the uniqueness of $C$, we assume that there is another effective divisor $C'\subset \FFF_m$ with $C'\sim aS_+$ and $C'\mid_{\iota_m(\AAA^2)}=C\mid_{\iota_m(\AAA^2)}$.
  So we have \[(C-C')\mid_{\iota_m(\AAA^2)}=0\] and because $\FFF_m\setminus\iota_m(\AAA^2)=S_-\cup f$ holds there are some $\alpha,\beta\geq0$ such that $C-C'=\alpha S_-+\beta f\sim0$, since $C\sim aS_+\sim C'$.
  Hence $\alpha=\beta=0$ and so $C=C'$, and \ref{item:PolynomialVsDivisor--PD} is proved.

  Let us prove \ref{item:PolynomialVsDivisor--DP}.
  As $C\sim aS_+$ there is a $g\in k(\FFF_m)$ with $\divv(g)=C-a\,S_+=C-a\,\divv(x_0)$.
  Hence there is a polynomial $G$ on $\FFF_m$ of degree $(a,0)$ with $g=\frac{G}{x_0^a}$.
  Hence $C=\divv(G)$ and $G$ is of the form $G=\sum_{i=0}^ax_0^i\,x_1^{a-i}\,G_{m(a-i)}(y_0,y_1)$.
  Let \[F(x,y):=G(x,1,y,1)=\sum_{i=0}^a x^i\,G_{m(a-i)}(y,1).\]
  We remark that if $C$ is irreducible, then $G$ is and hence also $F$ is irreducible.
  The zero set of $F$ corresponds to $C|_{\iota_m(\AAA^2)}$ with $\divv(F)=C|_{\iota_m(\AAA^2)}$, so $F$ is unique up to multiplication with a constant.
  As $G_{m(a-i)}(y,1)$ is of degree at most $m(a-i)$, the monomials $x^i\,y^j$ appearing in $F$ satisfy $j\leq m(a-i)$.
  By Lemma~\ref{lemma:TriangleToPolynomial}, the polynomial $F$ is of bidegree $(a,am)$ and $C$ is hence its $(a,am)$-divisor.
  This concludes the proof of~\ref{item:PolynomialVsDivisor--DP}.

  Let us show the equivalence of~\ref{item:PolynomialVsDivisor--PolyType} and~\ref{item:PolynomialVsDivisor--Divisibility}.
  Note that in~\ref{item:PolynomialVsDivisor--PD} and in~\ref{item:PolynomialVsDivisor--DP} we have \[F(x,y)=G(x,1,y,1)=\sum_{i=0}^ax^iG_{m(a-i)}(y,1).\]
  So if~\ref{item:PolynomialVsDivisor--PolyType} holds, that is if $F$ is of bidegree $(a,am-r)$, then by Lemma~\ref{lemma:TriangleToPolynomial}, the degree of $G_{m(a-i)}(y,1)$ is at most $N(i)$, where $N(i)=m(a-i)-r+\floor{\frac{ir}{a}}$ for all $i=0,\ldots,a$.
  As $G_{m(a-i)}$ is of degree $m(a-i)$, this implies that \[y_1^{m(a-i)-N(i)}=y_1^{r-\floor{\frac{ir}{a}}}\] needs to divide $G_{m(a-i)}(y_0,y_1)$, which is~\ref{item:PolynomialVsDivisor--Divisibility}.

  For the converse direction, we assume~\ref{item:PolynomialVsDivisor--Divisibility} and find $G_{m(a-i)}=0$ or $G_{m(a-i)}=y_1^{r-\floor{\frac{ir}{a}}}P_{N(i)}(y_0,y_1)$, where the $P_{N(i)}$ are homogeneous polynomials of degree $N(i)$.
  Hence, $F=\sum_{i=0}^ax^iP_{N(i)}(y,1)$, which is a polynomial of bidegree $(a,am-r)$ by Lemma~\ref{lemma:TriangleToPolynomial}.
  This is~\ref{item:PolynomialVsDivisor--PolyType}.

  It remains to prove the statements \ref{item:PolynomialVsDivisor--A} and \ref{item:PolynomialVsDivisor--B}. Let us start with \ref{item:PolynomialVsDivisor--A}.
  Note that for $r=1$ we have that $r-\floor{\frac{ir}{a}}$ is zero for $i=a$, and else it is~$1$.
  Hence, \ref{item:PolynomialVsDivisor--Divisibility} translates to $y_1\mid G_{m(a-i)}$ for $i=0,\ldots,a-1$.

  Assuming \ref{item:PolynomialVsDivisor--Divisibility}, $G$ can be written as \[y_1\sum_{i=0}^{a-1}x_0^ix_1^{a-i}H_{m(a-i)}(y_0,y_1)+G_0x_0^a,\] where $G_{m(a-i)}=y_1\,H_{m(a-i)}$.
  Therefore, if $G_0=0$, then $y_1$ divides $G$, and so $y_1=0$ is a component of $C$.
  If $G_0\neq 0$, then \[I_{[0:1;1:0]}(y_1,G)=I_{(0,0)}(y,G(x,1,1,y))=I_{(0,0)}(y,x^a)=a.\]
  Hence, we have shown that \ref{item:PolynomialVsDivisor--Divisibility} implies that $y_1=0$ is a component of $C$, or $I_p(y_1,C)=a$, which is the first part of \ref{item:PolynomialVsDivisor--A}.

  For the other direction of \ref{item:PolynomialVsDivisor--A}, note that if $y_1=0$ is a component of $C$, then we have directly that $G_0=0$ and that $y_1$ divides $G_{m(a-i)}$ for $i=0,\ldots, a-1$, implying \ref{item:PolynomialVsDivisor--Divisibility}.

  It remains to show that $I_p(y_1,C)=a$ implies \ref{item:PolynomialVsDivisor--Divisibility}, too.
  As $a=I_{(0,0)}(G(x,1,1,y),y)$, we find that $x^a\mid G(x,1,1,0)=\sum_{i=0}^ax^iG_{m(a-i)}(1,0)$.
  So we have $G_{m(a-i)}(1,0)=0$ for $i=0,\ldots,a-1$ and so $y_1\mid G_{m(a-i)}$, which is \ref{item:PolynomialVsDivisor--Divisibility}.
  Hence, \ref{item:PolynomialVsDivisor--A} is proved.

  Now, let us show \ref{item:PolynomialVsDivisor--B}.
  In one direction, we will show the more general statement ``\ref{item:PolynomialVsDivisor--Divisibility}$\implies$ $m_p(C)\geq3$ or \ref{item:PolynomialVsDivisor--b}'' for any $a\geq3$.
  For $a=3$, we have $m_p(C)\leq C\cdot f=3S_+\cdot f=3$, where $f$ is the fiber going through $p$, and thus we have \ref{item:PolynomialVsDivisor--a} or \ref{item:PolynomialVsDivisor--b}.

  Note that for $r=2$ and any $a\geq 3$ we have $\left(2-\floor{\frac{2i}{a}}\right)_{i=0}^2=(2,2,1)$.

  Assuming \ref{item:PolynomialVsDivisor--Divisibility} yields $y_1^2\mid G_{ma}$, $y_1^2\mid G_{m(a-1)}$ and $y_1\mid G_{m(a-2)}$.
  On the affine chart $\{[x:1;1:y]\mid (x,y)\in\AAA^2\}$ containing $p=[0:1;1:0]$ we can write $G$ as \[x^aG_0+x^{a-1}G_m(1,y)+\cdots+x^2G_{m(a-2)}(1,y)+xG_{m(a-1)}(1,y)+G_{ma}(1,y),\] and so it has no terms of degree $0$ and $1$.
  A term of degree 2 can only come from $G_{ma}(1,y)=\lambda y^2+\text{terms of higher degree}$. If $\lambda\neq0$ we have $m_p(G)=2$ and the only tangent direction of $G$ at $p$ comes from $y=0$.
  If $\lambda=0$ we have $m_p(G)\geq 3$.
  So we are either in case~\ref{item:PolynomialVsDivisor--a} or~\ref{item:PolynomialVsDivisor--b}.

  Let us now prove the converse direction.
  Assuming $a=3$, we consider \[G(x,1,1,y)=G_{3m}(1,y)+xG_{2m}(1,y)+ x^2G_m(1,y)+x^3G_0\] in both cases \ref{item:PolynomialVsDivisor--a} and \ref{item:PolynomialVsDivisor--b}:
  \begin{enumerate}
      \item[(a)]  If $m_p(C) = 3$ no terms of degree less than $3$ may appear in $G(x,1,1,y)$.
      So we have $y_1^3\mid G_{3m}$, $y_1^2\mid G_{2m}$ and $y_1\mid G_m$.
      This is even stronger than \ref{item:PolynomialVsDivisor--Divisibility}.
      \item[(b)] If $m_p(C)=2$ there may be no terms of degree less than $2$ in $G(x,1,1,y)$, and the only term of degree 2 is $y^2$ because $y_1=0$ is the only tangent direction of $C$ at $p$. So $y_1^2\mid G_{3m}$, $y_1^2\mid G_{2m}$ and $y_1\mid G_m$, implying \ref{item:PolynomialVsDivisor--Divisibility}.
    \end{enumerate}
\end{proof}

\begin{observation}
  Let $a,m\geq 1$ be two integers and let $F$ be a polynomial of bidegree $(a,am)$ with a singularity of type $A_k$ at $(0,0)$ for some integer $k\geq 1$.
  Then the $(a,am)$-divisor $C$ of $F$ has an $A_k$-singularity at $s=([0:1;0:1])$ (where $s$ stands for ``singular'').
  Figure~\ref{figure:PolynomialVsDivisor} depicts $C$ for $a=3$.

  Since $C\sim aS_+$, the divisor $C$ is an $a$-section if and only if $C$ does not contain any fibers.
\end{observation}

\begin{figure}[h]
\begin{center}

  \begin{tikzpicture}[xscale=0.5,yscale=0.3,baseline=(b.base)]
    \draw (0,-1) -- (0, 7);
    \draw (-3,0) -- (2, 0);
  \node (b) at (0,7){};
    \draw (-2,6) ..controls (-2,5.5) and (-1,5).. (0,5);
    \draw (0,5) ..controls (0.5,5) and (1.5,6).. (2,5);
    \draw (0,5) ..controls (0 ,5.3) and (2,4).. (2,5);
    \draw (0,5) ..controls (-2,4) and (-2,0).. (2,2);
    \draw (-2.2,0.5) node{\scriptsize{$S_-$}};
    \draw (0.3,0.5) node{\scriptsize{$f$}};

    \draw (-0.5,5.7) node{\scriptsize{$A_{k}$}};

  \end{tikzpicture}
  \begin{tikzpicture}[xscale=0.5,yscale=0.3,baseline=(b.base)]
    \draw (0,-1) -- (0, 7);
      \draw (-2,-1) -- (-2, 7);
    \draw (-3,0) -- (2, 0);
  \node (b) at (0,7){};
    \draw (-2,6) ..controls (-2,6.8) and (-1.7,6.9).. (-1.5,7);
    \draw (-2,6) ..controls (-2,5.5) and (-1,5).. (0,5);
    \draw (0,5) ..controls (0.5,5) and (1.5,6).. (2,5);
    \draw (0,5) ..controls (0 ,5.3) and (2,4).. (2,5);
    \draw (0,5) ..controls (-2,4) and (-2,0).. (2,2);
    \draw (-2.5,0.5) node{\scriptsize{$S_-$}};
    \draw (0.3,0.5) node{\scriptsize{$f$}};
      \draw (-2.4,4) node{\scriptsize{$f'$}};

    \draw (-0.5,5.7) node{\scriptsize{$A_{k}$}};

  \end{tikzpicture}
  \begin{tikzpicture}[xscale=0.5,yscale=0.3,baseline=(b.base)]
    \draw (0,-1) -- (0, 7);
    \draw (-4,0) -- (2, 0);
    \draw (-2.5,-1) -- (-2.5, 7);
  \node (b) at (0,7){};
    \draw (-2.5,6) ..controls (-2.8,5.5) and (-3.2,5).. (-3.5,5);
    \draw (-2.5,6) ..controls (-2,5.5) and (-1,5).. (0,5);
    \draw (0,5) ..controls (0.5,5) and (1.5,6).. (2,5);
    \draw (0,5) ..controls (0 ,5.3) and (2,4).. (2,5);
    \draw (0,5) ..controls (-2,4) and (-2,0).. (2,2);
    \draw (-2.9,0.5) node{\scriptsize{$S_-$}};
    \draw (0.3,0.5) node{\scriptsize{$f$}};
    \draw (-2.9,4) node{\scriptsize{$f'$}};
    \draw (-0.5,5.7) node{\scriptsize{$A_{k}$}};

  \end{tikzpicture}
\end{center}
\caption{Illustration of Lemma~\ref{lemma:PolynomialVsDivisor} with an $A_k$-singularity in the case $a=3$, $k$ odd, and $r=0,1,2$ (left to right).}
\label{figure:PolynomialVsDivisor}
\end{figure}
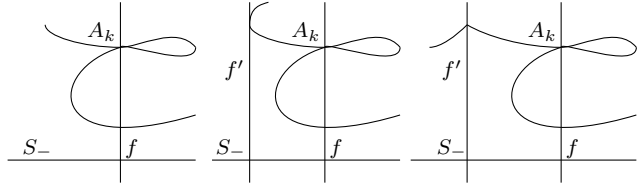

\begin{lemma}\label{lemma:TwoPointsCanBeChosenAsWished}
  Let $m\geq1$ be an integer and let $s$ and $t$ be two points on $\FFF_m$ that do not lie on the same fiber and that do not lie on $S_-$.
  Then, there exists an automorphism $\alpha\in\Aut(\FFF_m)$ such that $\alpha(s)=[0:1;0:1]$ and $\alpha(t)=[0:1;1:0]$.
\end{lemma}

\begin{proof}
  Applying an automorphism of the form $[x_0:x_1;y_0:y_1]\to[x_0:x_1;ay_0+by_1:cy_0+dy_1]$ with $\left(\begin{smallmatrix}
    a & b\\
    c & d\\
  \end{smallmatrix}\right)\in\GL_2(\CCC)$ we can assume that the fiber of $s$ is $y_0=0$, and the fiber of $t$ is $y_1=0$.
  As both points do not lie on $S_-$, we get $s=[a':1;0:1]$ and $t=[b':1;1:0]$ for some $a',b'\in\CCC$.
  By applying the coordinate change $x_0\mapsto x_0-x_1(a'\,y_1^m+b'\,y_0^m)$ we obtain the result.
\end{proof}

\begin{example}\label{example:2SectionWithEvenbLowerBound}
  Consider $C=C_1+C_2\subset\FFF_m$, where $C_1$ is given by the zero set of $F=x_0-x_1(y_0^m+y_1^m)$ and $C_2$ by $G=x_0-x_1y_1^m$.
  Note that $C_1,C_2\sim S_-+mf$ are both sections in $\FFF_m$, so $C\sim 2S_+$ is a $2$-section.
  Let us see that $C$ has an $A_{2m-1}$-singularity at $s=([1:1;0:1])$.
  Then we can apply a change of coordinates that sends $s$ onto $([0:1;0:1])$, namely $x_0\mapsto x_0-x_1y_1^m$.
  The existence of such a divisor implies with Lemma~\ref{lemma:PolynomialVsDivisor},\ref{item:PolynomialVsDivisor--DP} the existence of a polynomial $F$ of bidegree $(2,2m)$ with a singularity of type $A_{2m-1}$ at $(0,0)$, and so $N(2,2m)\geq 2m-1$.

  By inserting the parametrisation of $C_2$ into $F$ we find \[F(x_1y_1^m,x_1,y_0,y_1)=x_1y_0^m,\] hence $C_1$ and $C_2$ intersect only at $([1:1;0:1])$ with $I_s(C_1,C_2)=C_1\cdot C_2=m$.
  Therefore, after $m$ blow-ups $C_1$ and $C_2$ separate and $C$ gets smooth.
  As in Corollary \ref{corollary:SeveralBlowUpOfAk}, it follows that $C$ has an $A_k$-singularity where $k=2m$ or $k=2m-1$.
  Recall that the $A_k$-singularity has to be odd (as in Remark \ref{remark:OddEvenSingularities}) since it is the intersection of two curves, hence $k=2m-1$ as claimed.
\end{example}

\begin{lemma}\label{lemma:GenusUpperBoundForPolynomials}
  Let $a,m,r$ be integers with $a,m\geq 1$ and $0\leq r<a$.
  Let $F$ be a polynomial of bidegree $(a,am-r)$ with a singularity of  type $A_{k}$ such that its $(a,am)$-divisor is irreducible.
  Then $k\leq (a-1)(am-2)$.
  Moreover, if $a\geq 3$ and $r=2$, then $k\leq (a-1)(am-2)-2$.
\end{lemma}

\begin{proof}
  Let $C$ be the $(a,am)$-divisor of $F$.
  Hence we have $C\sim aS_+\subset\FFF_m$ irreducible and we can compute its arithmetic genus \begin{align*}
    g(C) &= \frac{1}{2}C\cdot (C+K_{\FFF_m})+1\\
    & = \frac{1}{2}aS_+\cdot\left((a-2)S_-+\left((a-1)m-2\right)f\right)+1\\
    & = \frac{1}{2}a\left((a-1)m-2\right)+1.
  \end{align*}
  Lemma~\ref{lemma:GenusUpperBoundForIrreducibleDivisors} yields \[k\leq 2g(C)=a\left((a-1)m-2\right)+2=(am-2)(a-1)\] and the first part of the lemma is proved.

  If $a\geq 3$ and $r=2$, by Lemma~\ref{lemma:PolynomialVsDivisor},\ref{item:PolynomialVsDivisor--PD} there is another singular point on $C$.
  So we have $k\leq 2g(C)-2$, which finishes the proof.
\end{proof}

\begin{lemma}\label{lemma:ReducibleDivisorWithAkSingularity}
  Let $C=C_1+\ldots+C_l$ be an effective divisor on a smooth surface with an $A_k$-singularity at a point $p\in C$, where $k\geq1$ and all $C_i$ are irreducible for $i=1,\ldots,l$.
  Then up to exchanging the order of the $C_i$'s, one of the following holds: \begin{enumerate}
    \item\label{item:ReducibleDivisorWithAkSingularity--1} $C_1$ has a singularity of type $A_k$ at $p$ and $p$ does not lie on any of the other $C_i$'s.
    \item\label{item:ReducibleDivisorWithAkSingularity--2} $p\in C_1\,\cap\, C_2$ is a smooth point of $C_1$ and $C_2$, $C_1\neq C_2$, that does not lie on any of the other $C_i$'s. Moreover, $k=2n-1$ where $n=I_p(C_1,C_2)$.
  \end{enumerate}
\end{lemma}

\begin{proof}
  If $p$ only lies in one of the $C_i$, then \ref{item:ReducibleDivisorWithAkSingularity--1} holds.
  So let us assume that $p\in C_1\cap C_2$ is the $A_k$-singularity, which implies that $C_1$ and $C_2$ are distinct because of the reducedness of $C$ at the $A_k$-singularity.
  Hence, $2=m_p(C)=\sum_{i=1}^l m_p(C_i)$ and so $p$ is a smooth point in $C_1$ and $C_2$, and  $p$ does not lie on any of the other $C_i$'s.
  In this case, the singularity needs to  be reducible, so $k$ has to be odd.
  Hence there is an integer $n$ with $k=2n-1$, that is $n=\ceil{\frac{k}{2}}$, where $n=I_p(C_1,C_2)$ is the number of points that have to be blown up until $C_1$ and $C_2$ do not intersect anymore, as in Corollary~\ref{corollary:SeveralBlowUpOfAk}.
\end{proof}

We have all ingredients to find the value of $N(2,b)$ for any $b\geq2$.

\begin{lemma}\label{lemma:BoundsForThe2SectionCase}
  For each integer $b\geq 1$ we have $N(2,b)=b-1$.
\end{lemma}

\begin{proof}
  The lower bound has been studied in Example \ref{example:a2LowerBound} for $b$ odd and in Example \ref{example:2SectionWithEvenbLowerBound} for $b$ even.
  It remains to show that they are also an upper bound.

  Let $F$ be a polynomial of bidegree $(2,2m-r)$ with an $A_k$-singularity, where $r\in\{0,1\}$.
  Let $C\subset\FFF_m$ be its $(2,2m)$-divisor.
  If $C$ is irreducible, then by Lemma~\ref{lemma:GenusUpperBoundForPolynomials} we obtain $k\leq2m-2$.

  So let us assume that $C$ is reducible, hence we can write $C=C_1+\cdots+C_l$, where all $C_i$ are irreducible for $i=1,\ldots,l$.
  Recall that $C\sim 2\,S_+=2\,(S_-+m\,f)$.
  We can apply Lemma~\ref{lemma:ReducibleDivisorWithAkSingularity}.

  In case~\ref{item:ReducibleDivisorWithAkSingularity--1}, we can write $C_1\sim aS_-+bf$ with $0\leq a\leq2$ and $0\leq b\leq 2m$.
  If $a=0$ (respectively $a=1$), then $C_1$ is a fiber (respectively a section, since $C_1$ is irreducible and contains thusly no fibers) and therefore smooth.
  So let us assume that $a=2$.
  Then, $0\leq C_1\cdot S_-=-2m+b$ and hence $b\geq2m$ and so $b=2m$ and $C=C_1$ is irreducible, a contradiction.

  In case~\ref{item:ReducibleDivisorWithAkSingularity--2} let us write $C_1\sim aS_-+bf$ and $C_2\sim cS_-+df$ with $a+c\leq 2$ and $b+d\leq2m$.
  We may assume that $a\geq c$, and $a=1$ (since $a=2$ implies that $C=C_1$ as before, and $a=0$ implies that $C_1$ and $C_2$ are both fibers, hence they do not meet). \begin{itemize}
    \item If $c=0$, then $n=I_p(C_1,C_2)\leq C_1\cdot C_2=(S_-+bf)\cdot d\,f=1$ and so $k=2n-1\leq1$ with Lemma \ref{lemma:ReducibleDivisorWithAkSingularity}.
    \item If $c=1$, then $n=I_p(C_1,C_2)\leq C_1\cdot C_2=b+d-m\leq m$ and so $k\leq2m-1$ with Lemma \ref{lemma:ReducibleDivisorWithAkSingularity}.

    It remains to show that equality in the latter equation cannot happen for $r=1$.
    We have $k=2m-1$ only if $b=d=m$ (which means that $C_1,C_2\sim S_+$ and hence $C=C_1+C_2$) and if $n=C_1\cdot C_2$ (that is if $C_1$ and $C_2$ intersect at $p$ only).
    So for any point $q$ distinct from $p$ we have $I_q(C,f)=I_q(C_1,f)+I_q(C_2,f)\leq1$ since $C_1$ and $C_2$ are both sections.
    Hence, case~\ref{item:PolynomialVsDivisor--A} of Lemma~\ref{lemma:PolynomialVsDivisor},\ref{item:PolynomialVsDivisor--PD} cannot apply and so $F$ is not of bidegree $(2,2m-1)$.
  \end{itemize}
  This finishes the proof.
\end{proof}

From now on we delve into the study of $N(3,b)$. A first result shows that $A_k$-singularities of reducible polynomials are not interesting enough.

\begin{lemma}\label{lemma:ReducibleDivisorsNotInteresting}
  Let $m\geq 2$ be an integer and let $F\in\CCC[x,y]$ be a polynomial of bidegree $(3,3m-r)$, where $r\in\{0,1,2\}$.
  Then, its $(3,3m)$-divisor $C\subset\FFF_m$ is either irreducible or has at most a singularity of type $A_{4m-1-r}$.
\end{lemma}
\begin{proof}
  By Lemma~\ref{lemma:PolynomialVsDivisor},\ref{item:PolynomialVsDivisor--PD}, we have $C\sim 3S_+$.
  Assume $C$ is reducible and write $C=C_1+\cdots+C_l$, where all $C_i$ for $i=1,\ldots,l$ are irreducible divisors.
  Assume that $C$ has a singularity of type $A_k$ at a point $p\in C$.
  Hence, either \ref{item:ReducibleDivisorWithAkSingularity--1} or \ref{item:ReducibleDivisorWithAkSingularity--2} of Lemma~\ref{lemma:ReducibleDivisorWithAkSingularity} holds.

  Let us first look at \ref{item:ReducibleDivisorWithAkSingularity--1}.
  Since $C_1$ is effective, we have $C_1\sim aS_-+bf$ with $0\leq a\leq 3$ and $0\leq b\leq 3m$.
  If $a=0$ (respectively $a=1$), $C_1$ is a fiber (respectively a section) and therefore smooth.

  If $a>1$, then $0\leq C_1\cdot S_-=-am+b$ and hence $b\geq am$.
  So $a=3$ is impossible, because that would give $b=3m$ and hence $C=C_1$ would be irreducible.
  The only remaining possibility is $a=2$ and therefore $C_1\sim 2S_-+bf$ with $2m\leq b\leq 3m$.
  Its arithmetic genus is \begin{align*}
    g(C_1)&=\frac{C_1\cdot(C_1 + K_{\FFF_m})}{2}+1\\
    &=\frac{C_1\cdot(b-m-2)f}{2}+1\\
    &=b-m-1\\
    &\leq 2m-1.
  \end{align*}
  By Lemma~\ref{lemma:GenusUpperBoundForIrreducibleDivisors}, $2g(C_1)\leq4m-2$ is an upper bound for $k$.
  It remains to show that we cannot have an $A_{4m-2}$-singularity if $r=2$.
  We find an $A_{4m-2}$-singularity at $p$ only if $g(C_1)=2m-1$ (which corresponds to $b=3m$) and $p$ is the only singular point of $C_1$ (as in Lemma~\ref{lemma:GenusUpperBoundForIrreducibleDivisors}).
  Hence, we get $C_1\sim 2S_-+3mf$ and so $C=C_1+S_-$.
  We want to see that this cannot occur for $r=2$ by finding a contradiction to \ref{item:PolynomialVsDivisor--B} of Lemma~\ref{lemma:PolynomialVsDivisor},\ref{item:PolynomialVsDivisor--PD}.
  Since $S_-$ is smooth, and $C_1$ contains no singular point except $p$, the divisor $C$ cannot have a point with multiplicity $\geq3$.
  A point with multiplicity $2$ besides $p$ is possible, but then one of its tangent directions is given by $S_-$, and this does not have the same tangent direction as a fiber.
  This contradicts \ref{item:PolynomialVsDivisor--B} of Lemma~\ref{lemma:PolynomialVsDivisor},\ref{item:PolynomialVsDivisor--PD}, and so $C$ cannot have a singularity of type $A_{4m-2}$ if $r=2$.

  We move on to \ref{item:ReducibleDivisorWithAkSingularity--2}: We will show that $C_1\cdot C_2\leq 2m$ and that equality $n=I_p(C_1,C_2)=2m$ cannot hold if $r=1$ (and thus if $r=2$).
  This implies with Lemma~\ref{lemma:ReducibleDivisorWithAkSingularity} that $k=2n-1\leq 4m-1$, and that $k=2n-1\leq 4m-2$ if $r=1$.
  However, since in case~\ref{item:ReducibleDivisorWithAkSingularity--2} only odd singularities are possible, we even have $k\leq 4m-3$ if $r=1$ (and thus if $r=2$).
  This achieves the proof.
  So we assume~\ref{item:ReducibleDivisorWithAkSingularity--2} and will prove the claims above.
  We write \begin{align*}
    C_1&\sim aS_-+bf\\
    C_2&\sim cS_-+df
  \end{align*} with $a+c\leq 3$ and $b+d\leq 3m$.
  We can assume $a\geq c$ and $a\leq2$ (because $a=3$ implies that $C=C_1$ is irreducible as before).
  If $c=0$, then $C_2$ is a fiber and since we assumed $m\geq 2$ we have $C_1\cdot C_2=a\leq 2<2m$.
  So we can assume $c=1$ (since $c\geq2$ is not possible because $a\geq c$ and $a+c\leq3$). Hence, $1\leq a\leq 2$.

  \begin{itemize}
    \item If $a=2$, then $0\leq C_1\cdot S_-=b-2m$ and so $b\geq 2m$.
    \begin{itemize}
      \item First, note that if $C_2=S_-$, then $C_1\cdot C_2=(2S_-+bf)\cdot S_-=-2m+b\leq m<2m$. (Note that the inequality is strict.)
      \item In the other case we have  $C_2\neq S_-$. Then, we have $0\leq      C_2\cdot S_-=-m+d$ and hence $b\geq 2m$ and $d\geq m$, which implies with $b+d\leq 3m$ that $b=2m$ and $d=m$.
      Hence $C_1\sim2S_+$, $C_2\sim S_+$ and so we have $C_1\cdot C_2=2m$ and $C=C_1+C_2$.
    \end{itemize}
    So only in the latter case the equality $C_1\cdot C_2=2m$ can occur. Assuming $n=I_p(C_1,C_2)=2m$, the point $p$ is the only point in the intersection of $C_1$ and $C_2$, and so for all points $q$ distinct from $p$ we have $I_q(C,f)=I_q(C_1,f)+I_q(C_2,f)\leq 2$.
    Therefore, case~\ref{item:PolynomialVsDivisor--A} in Lemma~\ref{lemma:PolynomialVsDivisor},\ref{item:PolynomialVsDivisor--PD} does not occur (since $C$ does not contain any fiber), and hence $r=0$.
    \item If $a=1$, then $C_1\cdot C_2=-m+d+b\leq 2m$.
    Assume in a first step that equality $C_1\cdot C_2=2m$ holds (later on we assume the stronger equality $n=2m$), which means that $b+d=3m$.
    Hence, $C=C_1+C_2+S_-$.
    If $b=0$ (or analogously, $d=0$), then $C_1=S_-$ and $C_2\sim S_-+3mf$.
    But now $C$ is not reduced at $p$, a contradiction to $C$ having an $A_k$-singularity at $p$.
    Hence, we have $b\neq 0$ and $d\neq0$. Hence, $C$ does not contain any fiber.

    If we assume $n=2m$, is it possible to have $r=1$?
    To achieve $n=C_1\cdot C_2$,the only point in the intersection of $C_1$ and $C_2$ is $p$, and so any point $q$ distinct from $p$ satisfies $I_q(C,f)=I_q(S_-,f)+I_q(C_1,f)+I_q(C_2,f)\leq 2$, since $C_i\cdot f=1$ for $i=1,2$.
    We conclude that case~\ref{item:PolynomialVsDivisor--A} in Lemma~\ref{lemma:PolynomialVsDivisor},\ref{item:PolynomialVsDivisor--PD} cannot occur.
  \end{itemize}
  We have showed that whenever $n=2m$ occurs, case~\ref{item:PolynomialVsDivisor--A} in Lemma~\ref{lemma:PolynomialVsDivisor},\ref{item:PolynomialVsDivisor--PD} is not satisfied. Therefore, $n=2m$ does not happen if $r=1$.
\end{proof}

\begin{corollary}\label{corollary:ReducibleDivisorsNotInteresting}
  Let $m\geq 2$ be an integer and let $F\in\CCC[x,y]$ be a polynomial of bidegree $(3,3m-r)$, where $r\in\{0,1,2\}$. Then, $F$ is irreducible, or has at most a singularity of type $A_{4m-1-r}$.
\end{corollary}

\begin{proof}
  Let $C$ be the $(3,3m)$-divisor of $F$.
  If $C$ is irreducible, then so is $F$, by~\ref{item:PolynomialVsDivisor--DP} of Lemma~\ref{lemma:PolynomialVsDivisor}.
  If $C$ is reducible, then $C$ and thus also $F$ has at most a singularity of type $A_{4m-1-r}$ by Lemma~\ref{lemma:ReducibleDivisorsNotInteresting}.
\end{proof}

\begin{remark}\label{remark:UpperBoundsReducibleIrreducible}
  Let $m\geq2$ and $r\in\{0,1,2\}$ be two integers and let $F$ be a polynomial of bidegree $(3,3m-r)$ with an $A_k$-singularity.
  If its $(3,3m)$-divisor is irreducible, we obtain an upper bound from Lemma~\ref{lemma:GenusUpperBoundForPolynomials}, if it is reducible we obtain one from Lemma~\ref{lemma:ReducibleDivisorsNotInteresting}:\begin{center}
    \begin{tabular}{r|| >{$}c<{$} | >{$}c<{$} | >{$}c<{$}| >{$}c<{$} |}
      & N(3,3m) & N(3,3m-1) & N(3,3m-2) & \text{asymptotically}\\ \hline
      $m\geq 2$, reducible & 4m-1 & 4m-2 & 4m-3 & \sim 4m \\ \hline
      $m\geq 1$, irreducible & 6m-4 & 6m-4 & 6m-6 & \sim 6m \\ \hline
    \end{tabular}
  \end{center}

  This gives us the following upper bounds (UB) in the cases $b=3m-r=3,\ldots,12$:
  \begin{center}
    \begin{tabular}{ r || >{$}c<{$} | >{$}c<{$} | >{$}c<{$}| >{$}c<{$}| >{$}c<{$}| >{$}c<{$}| >{$}c<{$} | >{$}c<{$} | >{$}c<{$} |>{$}c<{$} | }
      $b$ &  3 & 4 & 5 & 6 & 7 & 8 & 9 & 10 & 11 & 12\\ \hline\hline
      UB irreducible & 2 & 6 & 8 & 8 & 12 & 14 & 14 & 18 & 20 & 20\\ \hline
      UB reducible & 3 & 5 & 6 & 7 & 9 & 10 & 11 & 13 & 14 & 15\\ \hline
    \end{tabular}
  \end{center}
\end{remark}

\section{To Be ...}\label{section:ToBe}

In this section the goal is to give a lower bound for $N(3,b)$ where $b\leq 12$, namely the existence of a polynomial of bidegree $(3,b)$ with a certain singularity of type $A_k$ is shown.

In what follows we will not give the specific equation of a polynomial, but rather prove that a polynomial with certain properties exists.

In Section~\ref{section:ToBe--Recipe} we introduce our method: It is a ``recipe'' that ``cooks up'' polynomials with large singularities.
However, a recipe alone is not enough -- only ingredients make it useful.
We introduce these in Section~\ref{section:ToBe--Ingredients} and then use our recipe to prepare polynomials with large singularities.

\subsection{The recipe}\label{section:ToBe--Recipe}

We start by introducing some definitions that simplify the statements that follow.
Now is a good time to go back to take a look at Figures~\ref{figure:3-sectionSituation} and~\ref{figure:PolynomialVsDivisor}.

\begin{definition}\label{definition:TransversalPoint}
  Let $m\geq0$ be an integer and let $C\subset\FFF_m$ be an effective divisor and $p\in\FFF_m$ a point.
  We say that $p$ is a \textit{transversal point of $C$} (or that $C$ is \textit{transversal at $p$}) if $C$ intersects the fiber $f$ containing $p$ transversally, that is $I_p(C,f)=1$.
\end{definition}

We are interested in a configuration of curves with a certain behaviour on a fiber, described in the following definition.

\begin{definition}\label{definition:aConfiguration}
  Let $m\geq0$ be an integer, let $p$ and $s$ be two points in $\FFF_m$ and let $C$ and $S$ be two divisors on $\FFF_m$. We say that the configuration $(C,S,s,p)_m$ is an \textit{$a$-configuration} if the following hold: \begin{itemize}
    \item $C$ is an $a$-section,
    \item $S$ is a section,
    \item $p$ is a transversal point of $C$,
    \item $s\in C$,
    \item $p$ and $s$ are $C$-cofibered.
  \end{itemize}
\end{definition}

Sometimes, we will be interested only in a part of the configuration. In this case, we denote by $\bullet$ the parts we do not know about.
For instance, if we say that $(C,\bullet,\bullet,p)_m$ is an $a$-configuration, we just mean that $C$ is an $a$-section and that it contains a transversal point $p$. The existence of the rest of the configuration is \textbf{not} required.

\begin{definition}\label{definition:DisjointTangent}
  We say that an $a$-configuration $(C,S,\bullet,\bullet)_m$ is \textit{disjoint} if the intersection of $C$ and $S$ is empty.
  We say that the $a$-configuration $(C,S,\bullet,p)_m$ is \textit{tangent} if $C\cap S=\{p\}$, and we say that the $a$-configuration $(C,S,\bullet,\bullet)_m$ is \textit{tangent} if there exists a point $p$ such that $(C,S,\bullet,p)_m$ is tangent.
\end{definition}

For example, the situation in Figure~\ref{figure:PolynomialVsDivisor} depicts a disjoint $3$-configuration.
We will focus on $3$-configurations that are tangent or disjoint.

\begin{definition}\label{definition:TypeOfA3Configuration}
  Let $k\geq -1$ be an integer.
  Let $\mathcal{C}=(C,\bullet,s,p)_m$ be a $3$-configuration and let $f$ be the fiber meeting $s$ and $p$.
  We say that $\mathcal C$ is \textit{of type $-1$, $0$ or $k\geq 1$} in the following cases:
 \begin{enumerate}[label=(\Roman*)]
    \item\label{I}  If $C\cap f=\{p,s,t\}$ for a point $t$ distinct from $p$ and $s$, $\mathcal{C}$ is of type $-1$,
    \item\label{II} if $C\cap f=\{p,s\}$, and $f$ and $C$ are tangent at $s$, then $\mathcal{C}$ is of type $0$,
    \item\label{III} if $C\cap f=\{p,s\}$, and $s$ is an $A_k$-singularity of $C$ for some $k\geq 1$, then $\mathcal{C}$ is of type $k\geq 1$.
  \end{enumerate}
  We say that $(C,\bullet,s,\bullet)_m$ (respectively $(C,\bullet,\bullet,p)_m$) is of type $k$, if there is a point $p$ (respectively a point $s$) such that $(C,\bullet,s,p)_m$ is a $3$-configuration of type $k$.
\end{definition}
Note that $s\in C$ is a smooth point if $\mathcal{C}$ is of type $\leq0$, and $s\in C$ is an $A_k$-singularity if $\mathcal{C}$ is of type $k\geq 1$.

For example, we can rephrase Lemma~\ref{lemma:3-SectionSituation} with our new notions: Let $C$ be a $3$-section on $\FFF_m$ that has an $A_k$-singularity at a point $s\in C$ for some $k\geq3$. Then, $(C,\bullet,s,\bullet)_m$ is a $3$-configuration (and it is of type $k\geq 3$).

We now show that any $3$-configuration $(C,\bullet,\bullet,p)_m$ is of type $k$ for some $k\geq -1$.

\begin{lemma}\label{lemma:SituationOfA3SectionOnTheFiber}
  Let $m\geq0$ be an integer.
  Any $3$-configuration $(C,\bullet,\bullet,p)_m$ is of type $k$ for some $k\geq -1$.
\end{lemma}

\begin{proof}
  We show that there exists a point $s$ such that $(C,\bullet,s,p)_m$ is of type $k$ for some $k\geq -1$.
  As $p$ is a transversal point of $C$ we have $I_p(C,f)=1$, where $f$ is the fiber meeting $p$. This yields \[3=C\cdot f=\sum_{q\in\FFF_m} I_q(C,f)=1+\sum_{q\in\FFF_{m}\setminus\{p\}}I_q(C,f),\] which means that there are either another two distinct points lying on both curves and we are in case~\ref{I}, or there is only one point $s\neq p$ that lies on both $C$ and $f$ and has $I_s(C,f)=2$.
  Now, there are two possibilities: The point $s\in C$ is smooth and we are in case~\ref{II}, or $s$ is singular with $m_s(C)=2$ and we are in case~\ref{III},
  because then $s$ is an $A_k$-singularity of $C$ for some $k\geq 1$ by Lemma~\ref{lemma:OneBlowUpOrDownOfAk}, which we can apply since $C$ does not contain any fibers.
\end{proof}

On a $3$-configuration $(C,S,s,p)_m$ we will perform two kinds of links: a $p$-link or an $s$-link.
In this section we focus on $p$-links, in Section~\ref{section:OrNotToBe} we will study $s$-links.

\begin{lemma}\label{lemma:aSectionIsaSection}
  Let $a\geq1$ and $m\geq0$ be two integers. Let $p\in\FFF_m$ be any point, and let $D$ be a divisor on $\FFF_m$.
  Consider a $p$-link $\pi:\FFF_m\dashrightarrow\FFF_{m'}$.
  Let $D':=\pi_*(D)$. Then, $D$ is an $a$-section if and only if $D'$ is an $a$-section.
  If this holds, then $D$ is irreducible if and only if $D'$ is irreducible.
\end{lemma}
\begin{proof}
  Let $s'\in\FFF_{m'}$ be the inverse point of $\pi$.
    Since $\pi$ sends a general fiber of $\FFF_m$ (that is one that does not contain $p$) onto a general fiber of $\FFF_{m'}$ (that is one that does not contain $s'$), the push-forward $D'$ is an $a$-section if and only if $D$ is an $a$-section.
    If this holds, neither $D$ nor $D'$ contains a fibre, so $D$ is irreducible if and only if $D'$ is irreducible.
\end{proof}

\begin{lemma}\label{lemma:PropertiesaSectionAfterpLink}
  Let $a\geq1$ and $m\geq0$ be two integers. Let $(D,\bullet,\bullet,p)_m$ be an $a$-configuration.
  Let $\pi:\FFF_m\dashrightarrow\FFF_{m'}$ be a $p$-link with inverse point $s'$.
  Let $D':=\pi_*(D)$. The following statements hold: \begin{enumerate}[label=(\alph*)]
    \item\label{item:PropertiesaSectionAfterpLink--SelfIntersection} $D'^2=D^2+a^2-2a$,
    \item\label{item:PropertiesaSectionAfterpLink--mult} $m_{s'}(D')=a-1$,
    \item\label{item:PropertiesaSectionAfterpLink--point} there exists a (unique) cofibered point $p'\in D'$ with $s'$. Furthermore, $(D',\bullet,\bullet,p')_{m'}$ is an $a$-configuration.
  \end{enumerate}
  In particular, if $a\geq 2$ then $(D',\bullet,s',p')_{m'}$ is an $a$-configuration.
\end{lemma}

\begin{proof}
  Let $\rho:X\to \FFF_m$ be the blow up centered at $p$ so we can assume that $\sigma=\pi\circ\rho:X\to \FFF_{m'}$ is the contraction of the strict transform of $f$ onto the point $s'\in\FFF_{m'}$.
  Note that $p\in D$ is a smooth point because $I_p(D,f)=1$.
  So the self-intersection of the strict-transform of $D$ is $\tilde D^2=D^2-1$.
  Having $I_p(D,f)=1$ also implies \[\tilde D\cdot\tilde f=\sum_{q\in\FFF_m\setminus\{p\}}I_q(D,f)=a-1.\]
  Therefore, by contracting $\tilde f$ to the point $s'$, we find $m_{s'}(D')=a-1$, and \[D'^2=\tilde D^2+(a-1)^2=D^2+a^2-2a.\] Hence, we have shown~\ref{item:PropertiesaSectionAfterpLink--SelfIntersection} and~\ref{item:PropertiesaSectionAfterpLink--mult}.
  To prove~\ref{item:PropertiesaSectionAfterpLink--point}, recall that $D$ and $f$ intersect transversally at $p$, hence $\tilde D$ and $\tilde f$ do not meet on $E$, the strict transform of $\rho$.
  Moreover, $\tilde D$ intersects $E$ transversally at a point $\hat p\in X$ (not lying on the strict transform $\tilde f$). Therefore, $\sigma_*(D)=D'$ and $\sigma_*(E)=f'$ intersect transversally at $p'=\sigma(\hat p)$, which is a point lying on the same fiber $f'$ as $s'$.
  Hence, $p'$ is a transversal point of $D'$ and so $(D',\bullet,\bullet,p')_{m'}$ is an $a$-configuration.
  Moreover, if $a\geq 2$ we have that $s'\in D'$ and so $(D',\bullet,s',p')_{m'}$ is an $a$-configuration.
  For the uniqueness of $p'$ recall that by Lemma~\ref{lemma:aSectionIsaSection}, $D'$ is an $a$-section, and hence by~\ref{item:PropertiesaSectionAfterpLink--mult}, we have \[a=D'\cdot f'=\sum_{q\in\FFF_{m'}}I_q(D',f')\geq I_{p'}(D',f')+m_{s'}(D')=a\] and so $D'\cap f'=\{p',s'\}$.
  This concludes the proof.
\end{proof}

Therefore, it makes sense to define the direct image of an $a$-configuration.

\begin{definition}\label{definition:DirectImageOfAnAConfigurationAfterPLink}
  Let $a\geq 1$ and let $m\geq 0$.
  Let $\mathcal{C}=(C,\bullet,\bullet,p)_m$ be an $a$-configuration and let $\pi:\FFF_{m}\dashrightarrow\FFF_{m'}$ be a $p$-link with inverse point $s'$. We define the direct image $\pi_*(\mathcal{C})$ to be the $a$-configuration $(\pi_*(C),\bullet, s',p')_{m'}$ as obtained in Lemma~\ref{lemma:PropertiesaSectionAfterpLink} (respectively $(\pi_*(C),\bullet,\bullet,p')_{m'}$ if $a=1$).
  In the same way we define the direct image of an $a$-configuration $(C,S,\bullet,p)_m$ under $\pi$ to be the $a$-configuration $(\pi_*(C),\pi_*(S),s',p')_{m'}$.
\end{definition}

This leads us to a chain of $p$-links.

\begin{definition}\label{definition:transversalChainOfnLinks}
  Let $m_0=m\geq0$, $a\geq1$ and $n\geq 1$ be three integers.
  Let $\mathcal{C}=\mathcal{C}_0$ be an $a$-configuration $(C_0=C,\bullet,\bullet,p_0=p)_{m}$.
  For $i=1,\ldots,n$, let $\pi_i:\FFF_{m_{i-1}}\dashrightarrow\FFF_{m_i}$ be a $p_{i-1}$-link with inverse point $s_i$.
  Let $\mathcal{C}_i=(C_i,\bullet,s_i,p_i)_{m_i}$ be the direct image of $\mathcal{C}_{i-1}$.
  We call the composition \[\pi=\pi_n\circ\cdots\circ\pi_1:\FFF_{m_0}\dashrightarrow\FFF_{m_n}\] a \textit{transversal $\mathcal{C}$-chain of $n$ links} (because we perform a series of links obtained from blowing up the transversal point of $C$).
  We say that $s_n\subset\FFF_{m_n}$ is the \textit{inverse point} of $\pi$ and that the \textit{direct image of $\mathcal{C}$ under $\pi$} is $\mathcal{C}_n$.
\end{definition}

\begin{remark}
  A transversal $\mathcal{C}$-chain of $n$ links is unique up to isomorphism at the target, since each of the $\pi_i$ is unique up to isomorphism.
  The points $p_i$ and $s_i$ being cofibered, $\pi$ restricted to $\FFF_{m_0}\setminus f_0$ is an isomorphism to $\FFF_{m_n}\setminus f_n$, where $f_0$ (respectively $f_n$) is the fiber containing $p_0$ (respectively $p_n$).
\end{remark}

\begin{lemma}\label{lemma:3SectionAfterChainOfnLinks}
  Let $m\geq0$ and $n\geq 1$ be two integers.
  Let $\mathcal{C}=(C,\bullet,\bullet, p)_m$ be a $3$-configuration of type $k\geq -1$.
  Let $\pi:\FFF_m\dashrightarrow\FFF_{m'}$ be a transversal $\mathcal{C}$-chain of $n$ links with inverse point $s'$.
  Then the $3$-configuration $\mathcal{C}'=\pi_*(\mathcal{C})$ is of type $K=2n+k\geq 1$.
\end{lemma}

\begin{proof}
  By induction, it is enough to prove the statement for $n=1$, which we do now.
  Recall that $\rho:X\to \FFF_m$ is the blow up centered at $p$ and $\sigma=\pi\circ\rho:X\to \FFF_{m'}$ is the contraction of the strict transform of $f$ onto the point $s'\in\FFF_{m'}$.
  The exceptional divisor of $\sigma$, denoted by $E$, equals $\tilde f$, and the strict transform of $C$ with respect to $\rho$ equals the strict transform of $C'$ with respect to $\sigma$.
  So we have $\tilde C'\cap E=\tilde C\cap \tilde f$.

  By~\ref{item:PropertiesaSectionAfterpLink--mult} of Lemma~\ref{lemma:PropertiesaSectionAfterpLink}, we know that $m_{s'}(C')=2$, so part~\ref{item:OneBlowUpOrDownOfAk--equivalence} of Lemma~\ref{lemma:OneBlowUpOrDownOfAk} is satisfied (note that $C'$ is a $3$-section and is therefore reduced) and $C'$ has a singularity of type $A_K$ at $s'$ for some $K\geq 1$.
  Hence, part~\ref{item:OneBlowUpOrDownOfAk--IToIII} can be applied onto $C'$ (with respect to $\sigma$). In all cases we find that $k=K-2$.
  \begin{itemize}
    \item[\ref{item:OneBlowUpOrDownOfAk--I}] If $K=1$, then $\tilde C\cap \tilde f$ contains two distinct points, say $\hat s, \hat t$.
    Hence, $C\cap f$ also contains the distinct points $s=\rho(\hat s)$ and $t=\rho(\hat t)$, but it also contains $p$, so $\mathcal{C}$ is of type $-1$.
    \item[\ref{item:OneBlowUpOrDownOfAk--II}] If $K=2$, then $\tilde C\cap \tilde f=\{\hat s\}$, where $\hat s\in\tilde C$ is smooth and $E=\tilde f$ and $\tilde C'=\tilde C$ are tangent at $\hat s$.
    Hence, $C\cap f=\{p,s\}$, where $s=\rho(\hat s)$ and $C$ and $f$ are tangent at $s$, concluding that $\mathcal{C}$ is of type $0$.
    \item[\ref{item:OneBlowUpOrDownOfAk--III}] If $K\geq 3$, then $\tilde C\cap \tilde f=\{\hat s\}$, where $\hat s\in \tilde C$ is a singular point of type $A_{K-2}=A_k$.
    Therefore, $C\cap f=\{p,s\}$, where $s=\rho(\hat s)$ and $C$ has a singularity of type $A_{k}$ at $s$ with $k=K-2\geq1$. Hence, $\mathcal{C}$ is of type $k=K-2$.
  \end{itemize}
  By Lemma~\ref{lemma:SituationOfA3SectionOnTheFiber} we know that $\mathcal{C}$ is of type $k\geq -1$, and so the dichotomy proves the ``if and only if''-statement.
\end{proof}

The following pictures illustrate the situation for $k=-1$, $k=0$, and $k\geq1$:
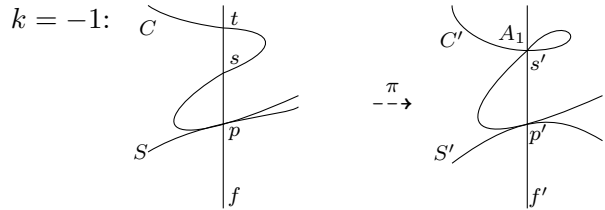
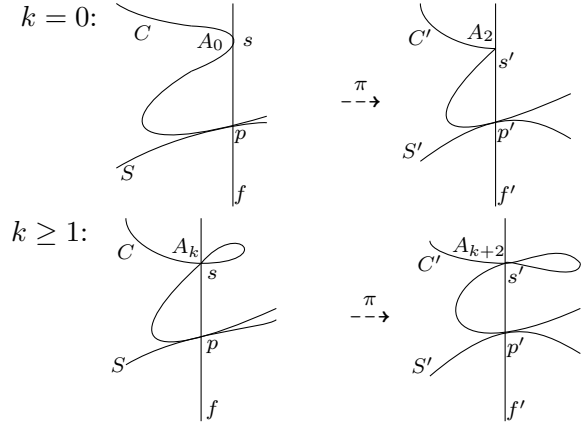
\begin{center}
  $k=-1$:
  \begin{tikzpicture}[xscale=0.5,yscale=0.3,baseline=(b.base)]
 \draw (0,-2) -- (0, 7);
 \draw (-2,7) ..controls (-1.5,6.4) and (-0.5,6.1).. (0,6);
 \draw (0,6) ..controls (1.5,5.9) and (1.5,5).. (0,4);
 \draw (0,4) ..controls (-2,2) and (-2,0).. (2,3);

 \draw (-2,0.5) ..controls (-0.5,2) and (1.7,2.1).. (2,2.5);

 \draw (-2.2,0.5) node{\scriptsize{$S$}};
 \draw (-2,6) node{\scriptsize{$C$}};
 \draw (0.3,-1.5) node{\scriptsize{$f$}};
 \draw (0.3,1.3) node{\scriptsize{$p$}};
 \draw (0.3,4.6) node{\scriptsize{$s$}};
 \draw (0.3,6.4) node{\scriptsize{$t$}};

     \draw (4.5,3) node{$\stackrel{\pi}\dashrightarrow$};
 \begin{scope}[xshift=230]
   \draw (0,-2) -- (0, 7);
   \node (b) at (0,6){};
   \draw (-2,7) ..controls (-2,6) and (-1,5).. (0,5);
   \draw (0,5) ..controls +(2,0) and +(1, 2).. (0,5);
   \draw (0,5) ..controls (-2,2) and (-2,0).. (2,3);

   \draw (-2,0) ..controls (0,2.5) and (1,2).. (2,1);

   \draw (-2.2,0.5) node{\scriptsize{$S'$}};
   \draw (-2,5.5) node{\scriptsize{$C'$}};
   \draw (0.3,-1.5) node{\scriptsize{$f'$}};
   \draw (0.3,1.3) node{\scriptsize{$p'$}};
   \draw (0.3,4.5) node{\scriptsize{$s'$}};
   \draw (-0.4,5.7) node{\scriptsize{$A_1$}};
 \end{scope}
 \end{tikzpicture}

$k=0$:
 \begin{tikzpicture}[xscale=0.5,yscale=0.3,baseline=(b.base)]
 \draw (1.1,-2) -- (1.1, 7);
 \draw (-2,7) ..controls (-1.5,6.4) and (-0.5,6.1).. (0,6);
 \draw (0,6) ..controls (1.5,5.9) and (1.5,5).. (0,4);
 \draw (0,4) ..controls (-2,2) and (-2,0).. (2,2);

 \draw (-2,-0.3) ..controls (-0.2,1.5) and (2,1.8).. (2,1.7);
 \draw (-1.7,-0.5) node{\scriptsize{$S$}};
 \draw (-1.3,5.7) node{\scriptsize{$C$}};
 \draw (1.3,-1.5) node{\scriptsize{$f$}};
 \draw (1.3,1.1) node{\scriptsize{$p$}};
 \draw (1.5,5.3) node{\scriptsize{$s$}};
 \draw (0.5,5.3) node{\scriptsize{$A_0$}};

     \draw (4.5,3) node{$\stackrel{\pi}\dashrightarrow$};
 \begin{scope}[xshift=230]
   \draw (0,-2) -- (0, 7);
   \node (b) at (0,6){};
   \draw (-2,7) ..controls (-2,6) and (-1,5).. (0,5);
   \
   \draw (0,5) ..controls (-2,2) and (-2,0).. (2,3);

   \draw (-2,0) ..controls (0,2.5) and (1,2).. (2,1);

   \draw (-2.2,0.5) node{\scriptsize{$S'$}};
   \draw (-2,5.5) node{\scriptsize{$C'$}};
   \draw (0.3,-1.5) node{\scriptsize{$f'$}};
   \draw (0.3,1.3) node{\scriptsize{$p'$}};
   \draw (0.3,4.5) node{\scriptsize{$s'$}};
   \draw (-0.4,5.7) node{\scriptsize{$A_2$}};
 \end{scope}
 \end{tikzpicture}

$k\geq1$:
 \begin{tikzpicture}[xscale=0.5,yscale=0.3,baseline=(b.base)]
 \draw (0,-2) -- (0, 7);
 \node (b) at (0,6){};
 \draw (-2,7) ..controls (-2,6) and (-1,5).. (0,5);
 \draw (0,5) ..controls +(2,0) and +(1, 2).. (0,5);
 \draw (0,5) ..controls (-2,2) and (-2,0).. (2,3);

 \draw (-2,0.5) ..controls (-0.5,2) and (1.7,2.1).. (2,2.5);

 \draw (4.5,3) node{$\stackrel{\pi}\dashrightarrow$};
 \draw (-2.2,0.5) node{\scriptsize{$S$}};
 \draw (-2,5.5) node{\scriptsize{$C$}};
 \draw (0.3,-1.5) node{\scriptsize{$f$}};
 \draw (0.3,1.3) node{\scriptsize{$p$}};
 \draw (0.3,4.5) node{\scriptsize{$s$}};
 \draw (-0.4,5.7) node{\scriptsize{$A_k$}};
 \begin{scope}[xshift=230]
   \draw (0,-2) -- (0, 7);

   \draw (-2,6) ..controls (-2,5.5) and (-1,5).. (0,5);
   \draw (0,5) ..controls (0.5,5) and (1.5,6).. (2,5);
   \draw (0,5) ..controls (0 ,5.3) and (2,4).. (2,5);
   \draw (0,5) ..controls (-2,4) and (-2,0).. (2,3);
   \draw (-2,0) ..controls (-0.1,2.7) and (0.8,2.2).. (2,1);
   \draw (-2.2,0.5) node{\scriptsize{$S'$}};
   \draw (-2,5) node{\scriptsize{$C'$}};
   \draw (0.3,-1.5) node{\scriptsize{$f'$}};
   \draw (0.3,1.3) node{\scriptsize{$p'$}};
   \draw (0.3,4.5) node{\scriptsize{$s'$}};
   \draw (-0.7,5.7) node{\scriptsize{$A_{k+2}$}};
 \end{scope}
 \end{tikzpicture}
\end{center}

\begin{lemma}\label{lemma:LocalIntersectionAfterTransversalChain}
  Let $m\geq 0$ and $a\geq1$ be two integers, and let $\mathcal{C}=(C,S,\bullet,p)_m$ be a tangent $a$-configuration.
  Let $1\leq n\leq I_p(S,C)$ be an integer.
  Let $\pi:\FFF_{m}\dashrightarrow\FFF_{m'}$ be a transversal $\mathcal{C}$-chain of $n$ links, and let $\mathcal{C}'=(C',S',\bullet,p')_{m'}$ be the direct image $\pi_*(\mathcal{C})$. The $a$-configuration $\mathcal{C}'$ is tangent or disjoint, and $I_{p'}(S',C')=I_p(S,C)-n$.
  In particular, it is tangent if and only if $I_{p'}(C',S')\geq 1$.
\end{lemma}

\begin{proof}
  By induction, it is enough to prove the statement for $n=1$ since assuming $n\leq I_p(S,C)$ asserts that in each step we have $I_{p_i}(S_i,C_i)\geq 1$ for $i=1,\ldots,n-1$, hence the direct image is again a tangent $a$-configuration. So assume $n=1$.

  Let $s'$ be the inverse point of $\pi$. Observe that $s'\notin S'$ by \ref{item:PropertiesaSectionAfterpLink--mult} of Lemma~\ref{lemma:PropertiesaSectionAfterpLink}.

  Recall that $\rho:X\to \FFF_m$ is the blow up centered at $p$ and $\sigma=\pi\circ\rho:X\to \FFF_{m'}$ is the contraction of the strict transform of the fiber $f$ containing $p$ onto the point $s'\in\FFF_{m'}$.
  Let $\tilde C$, respectively $\tilde S$, be the strict transform of $C$, respectively $S$.
  Note that $p$ is a smooth point of $C$ (because $I_p(C,f)=1$) and of $S$ (because $S$ is smooth since it is a section). Having $p\in C\cap S$ yields \[\tilde C\cdot \tilde S=C\cdot S-1=I_p(C,S)-1.\]
  Since $\mathcal{C}$ is tangent, $C\cap S=\{p\}$ and so $\tilde C\cap \tilde S\subset E$, where $E$ is the exceptional divisor with respect to the blow-up of $p$.
  Let $\hat p$ be the intersection of $E$ and $\tilde C$ (as in the proof of Lemma~\ref{lemma:PropertiesaSectionAfterpLink}), then $\sigma(\hat p)=p'$.
  Hence, $\tilde C$ and $\tilde S$ intersect at most in $\hat p$, giving \[I_{\hat p}(\tilde C,\tilde S)=\tilde C\cdot\tilde S=I_p(S,C)-1,\] and so also $I_{p'}(C',S')=I_{p}(S,C)$, as the intersection multiplicity is a local property.
\end{proof}

To summarize the lemmas of this section, we equip a $3$-configuration $\mathcal{C}=(C,S,s,p)_m$ of type $k\geq -1$ with information \[[C^2,S^2,k,I_p(C,S);m].\]
With this notation, we obtain the following lemma.

\begin{lemma}\label{lemma:InformationOfTangentAConfiguration}
  Let $m\geq0$ and $a\geq 1$ be two integers. Let $\mathcal{C}=(C,S,s,p)_m$ be a tangent $a$-configuration and let $n$ be an integer with $n\leq I_p(S,C)$.
  Let $\pi:\FFF_m\dashrightarrow\FFF_{m'}$ be a transversal $\mathcal{C}$-chain of $n$ links.
  Then $\mathcal{C}'=\pi_*(\mathcal{C})$ is equipped with \[[C^2+n(a^2-2a),S^2-n, \bullet, I_p(S,C)-n;\bullet].\]
  Moreover, if $a=3$ then $C'^2=C^2+3n$, and $\mathcal{C}'$ is of type $k+2n$ if and only if $\mathcal{C}$ is of type $k\geq -1$.
\end{lemma}

\begin{proof}
  Write $\mathcal{C}'=(C',S',s',p')_{m'}$.
  Lemma~\ref{lemma:LocalIntersectionAfterTransversalChain} gives directly that $I_{p'}(S',C')=I_{p}(S,C)-n$, and Lemma~\ref{lemma:3SectionAfterChainOfnLinks} provides the statement about the type of $\mathcal{C}'$.
  It remains to compute $C'^2$ and $S'^2$.
  We want to apply \ref{item:PropertiesaSectionAfterpLink--SelfIntersection} of Lemma~\ref{lemma:PropertiesaSectionAfterpLink} $n$ times to the $1$-configuration $(S,\bullet,\bullet,p)_m$ and to the $a$-configuration $(C,\bullet,\bullet,p)_m$.
  To do this, we need to know that in each step of the transversal $\mathcal{C}$-chain the point $p_i$ is contained in $S_i\cap C_i$. This is true because we chose $n\leq I_p(S,C)$, and so $I_{p_i}(S_i,C_i)\geq1$, hence $p_i\in S_i\cap C_i$ in each step.
  So by applying~\ref{item:PropertiesaSectionAfterpLink--SelfIntersection} of Lemma~\ref{lemma:PropertiesaSectionAfterpLink}, we get $S'^2=S^2-n$ and $C'^2=C^2+n(a^2-2a)$ as claimed.
\end{proof}

In general, we do not know on which Hirzebruch surface $\FFF_{m'}$ we arrive.
However, under assumptions as in the following lemma, we can determine $m'$.

\begin{lemma}\label{lemma:aCollectionGivesSpecialCollection}
  Let $m\geq0$ and $a\geq 1$ be two integers.
  Let $\mathcal{C}=(C,S,\bullet,p)_m$ be a tangent $a$-configuration and let $S^2\leq n\leq I_p(S,C)$ be an integer.
  Assume that $C^2=2an-a^2S^2$. Let $\pi:\FFF_{m}\dashrightarrow\FFF_{m'}$ be a transversal $\mathcal{C}$-chain of $n$ links and let $\pi_*(\mathcal{C})=(C',S',\bullet,p')_{m'}$.
  Then, $m'=n-S^2$, $S'=S_-$ and $C'\sim aS_+$.
\end{lemma}

\begin{proof}
  It remains to prove that $m'=n-S^2$, using $S^2\leq n$, and that $C'\sim aS_+$ and $S'=S_-$ under the assumption $C^2=2an-a^2S^2$.
  Applying Lemma~\ref{lemma:InformationOfTangentAConfiguration}, we find that $S'^2=S^2-n\leq 0$.
  Since $S'$ is a section, it is therefore irreducible and isomorphic to $\PPP^1$, and hence it is the $(-m')$-curve on $\FFF_{m'}$.
  Therefore, $m'=n-S^2$ and $S'=S_-\subset\FFF_{m'}$.
  With Lemma~\ref{lemma:InformationOfTangentAConfiguration} we also find that \[C'^2=C^2+n(a^2-2a)=a^2(n-S^2)=a^2m'.\]
  As $C'$ is an $a$-section, we have $C_n\sim aS_-+bf$ for some $b\geq0$. Inserting this into the value of $C'^2$ yields $b=am'$ and hence $C'\sim aS_+$.
\end{proof}

\subsection{The ingredients}\label{section:ToBe--Ingredients}

In this section we will show the existence of some tangent $3$-configurations $\mathcal{C}=(C,S,\bullet,p)_m$ satisfying the assumptions of Lemma~\ref{lemma:aCollectionGivesSpecialCollection} with $m=0$ (that is, a configuration in $\FFF_0=\PPP^1\times\PPP^1$) or $m=1$ (that is, a configuration in $\FFF_1$ obtained by the blow-up at one point of a configuration in $\PPP^2$), and we let $n=I_p(S,C)$.

These are then the ``ingredients'' that we can put into the ``recipe'' that we established in the last section:
Applying a transversal $\mathcal{C}$-chain of $n$ links $\pi:\FFF_m\dashrightarrow\FFF_{m'}$, we get a divisor $C'\sim 3S_+\subset\FFF_{m'}$, where $m'=n-S^2$, and $C'$ has a singularity of type $A_{K}$, where $K=2n+k$ and $\mathcal{C}$ is of type $k$.
Upon the divisor $C'\subset\FFF_{m'}$ we can apply Lemma~\ref{lemma:PolynomialVsDivisor} and find a polynomial $F\in\CCC[x,y]$ of bidegree $(3,3m')$ that has a singularity of type $A_{K}$ at some point.

Moreover, if the curve $C$ we started with has a special intersection property with some general fiber as in \ref{item:PolynomialVsDivisor--A} or~\ref{item:PolynomialVsDivisor--B}, then also $C'$ has the same intersection property with a general fiber, since $\pi$ sends a general fiber onto a general fiber, leading to $F$ being of bidegree $(3,3m'-1)$ (case~\ref{item:PolynomialVsDivisor--A}) or $(3,3m'-2)$ (case~\ref{item:PolynomialVsDivisor--B}).

To recapitulate, we start with a relatively ``easy'' configuration and can then find a polynomial of a certain type with a ``large'' singularity of type $A_K$. In this way, we will get a lower bound for $N(3,b)$.

First, we give a $3$-configuration $(C,S,\bullet,p)_0$ in $\FFF_0=\PPP^1\times\PPP^1$ that yields with the method described above a lower bound of $N(3,9)$.

\begin{lemma}\label{lemma:Existence-3-9--13}
  There exists an irreducible polynomial of bidegree $(3,9)$ with a singularity of type $A_{13}$.
\end{lemma}

\begin{proof}
  Consider the curves $C=V(F)$ and $S=V(G)$ in $\FFF_0=\PPP^1\times\PPP^1$, where $G= x_0y_1^2-x_1y_0^2$ and \[F=x_0^3(y_0+7\,y_1)+x_0^2\,x_1(21\,y_0+35\,y_1)+x_0\,x_1^2(35\,y_0+21\,y_1)+x_1^3(7\,y_0+y_1).\] Let $p=[1:1;1:-1]$ and let $f$ be the fiber going through $p$, hence $f$ is given by $y_0+y_1=0$.
  We prove that $\mathcal{C}=(C,S,\bullet,p)_0$ is a tangent $3$-configuration.

  First, we show that $p$ is the unique intersection point of $C$ and $S$.
  We can parametrize $S$ by $[y_0:y_1]\mapsto([y_0^2:y_1^2],[y_0:y_1])$.
  Inserting this parametrization into $F$, we find \[F(y_0^2,y_1^2,y_0,y_1)=(y_0+y_1)^7,\] and so $C$ and $S$ intersect only at $p$ with $I_p(S,C)=7$.

  Next, we show that $S$ is a section. Since $S$ is of bidegree $(1,2)$, it satisfies $S\cdot f=1$ for any fiber $f$ that is given by a linear equation in $y_0,y_1$ (and $S^2=4$).
  To see that $S$ is a section, we need to check that it does not contain any fibers.
  If $S$ would contain a fiber, then it were the fiber $y_0+y_1=0$ that contains $p$ (otherwise, $S$ and $C$ would meet in a second point). On this fiber we have $G(x_0,x_1,1,-1)= x_0+x_1$, which does not vanish everywhere. Therefore, $S$ does not contain any fiber and is hence a section.

  Similarly, we note that $F$ is of bidegree $(3,1)$ and so $C\cdot f=3$ (and $C^2=6$). We can compute the intersection of $C$ and $f$, namely \[F(x_0,x_1,y_0,-y_0)=(3x_0+x_1)(x_0+3x_1)(x_0-x_1).\] Hence, $C$ intersects $f$ in three distinct points and so $I_p(C,f)=1$, hence $p$ is a transversal point of $C$.
  We can now see that $C$ does not contain any fibers: If $C$ would contain a fiber, then it would be the fiber going through $p$ (otherwise, $C$ and $S$ would intersect in a second point). This contradicts $p$ being a transversal point of $C$. Hence, $C$ is a $3$-section.

  To sum it up: $\mathcal{C}$ is a tangent $3$-configuration equipped with information \[[6,4,-1,7;0],\] where it is of type $-1$ because $C$ and $f$ intersect at $3$ distinct points.
  Letting $n=7$, one can check that the assumptions of Lemma~\ref{lemma:aCollectionGivesSpecialCollection} are satisfied.
  Applying this lemma and Lemma~\ref{lemma:InformationOfTangentAConfiguration}, we get a $3$-configuration $\mathcal{C}'=(C',S_-,s',p')_{3}$ of type $-1+2n=13$ such that $C'\sim 3S_+$.
  By Lemma~\ref{lemma:ReducibleDivisorsNotInteresting}, $C'$ is irreducible.
  Therefore, there is an irreducible polynomial $F$ of bidegree $(3,9)$ such that $C'$ is its $(3,9)$-divisor by Lemma~\ref{lemma:PolynomialVsDivisor}, and so $F$ has an $A_{13}$-singularity.
\end{proof}

\begin{remark}\label{remark:BinomialCoefficientExample}
  Note that the coefficients of the polynomial $F$ from the above lemma is the 7th row of Pascal's triangle and so they are binomial coefficients.
  We can generalise Lemma~\ref{lemma:Existence-3-9--13} to the following statement:
  \begin{quotation}
    Let $m\geq 3$ be an odd integer. Then there exists an irreducible polynomial of bidegree $(3,3m)$ with at least an $A_{4m+1}$-singularity.
  \end{quotation}
  It can be proved by writing $m=2a-1$ for an integer $a\geq 2$ and considering the curves $C=V(F)$ and $S=V(G)$ in $\FFF_0$, where \begin{align*}
      G = & x_0y_1^a-x_1y_0^a,\\
      F = & x_0^3\sum_{i=0}^{a-1}{{4a-1}\choose{i}}y_0^{a-1-i}y_1^i + x_0^2x_1\sum_{i=a}^{2a-1}{{4a-1}\choose{i}}y_0^{2a-1-i}y_1^{i-a} +\\
       & x_0x_1^2\sum_{i=2a}^{3a-1}{{4a-1}\choose{i}}y_0^{3a-1-i}y_1^{i-2a} + x_1^3\sum_{i=3a}^{4a-1}{{4a-1}\choose{i}}y_0^{4a-1-i}y_1^{i-3a}.\\
    \end{align*}
  Note that the coefficients of $F$ correspond to the $(4m-1)$th row of Pascal's triangle, which is the nice part: By plugging in the parametrisation of $S$ into $F$ we find \[F(y_0^a,y_1^a,y_0,y_1)=(y_0+y_1)^{4a-1},\] and so $p=[1:(-1)^a;1:-1]$ is the unique intersection point of $S$ and $C$, providing tangency in case $(C,S,\bullet,p)_0$ is a $3$-configuration.
  However, to show that it is a $3$-configuration involves lengthy computations with binomial coefficients, which is the reason why we refrain from presenting the proof.
  Furthermore, this gives an asymptotical lower bound $N(3,b)\geq \frac{4}{3}b$, whereas Lemma~\ref{lemma:LowerBoundForN(3,2b)} gives $N(3,b)\geq \frac{3}{2}b$.
\end{remark}

\begin{remark}\label{remark:SpecialWeierstrassPoints}
  We touch upon a connection to Weierstrass points on $\PPP^1\times \PPP^1$ as introduced by Maugesten and Moe \cite{maugesten-moe_2018} in 2018.
  A curve $S\subset\PPP^1\times\PPP^1$ of degree $(\alpha,\beta)$ is said to be a \textit{hyperosculating curve} to some curve $C$, if \[I_p(S,C)>(\alpha+1)(\beta+1)-1,\] and if this holds $p$ is an \textit{$(\alpha,\beta)$-Weierstrass point} of $C$.
  They study the case where $\alpha$ and $\beta$ are at most one.

  In our situation, if an $a$-configuration $(C,S,\bullet,p)_0$ is tangent, then $S\sim S_-+\beta f$ is a hyperosculating curve to $C$, and $p$ is an $(1,\beta)$-Weierstrass point of $C$.
  For instance, the curve $C$ of Remark~\ref{remark:BinomialCoefficientExample}, which has degree $(3,a-1)$, contains $p$ as a smooth $(1,a)$-Weierstrass point, with hyperosculating curve $S$ of degree $(1,a)$.

  Looking ahead, the examples of tangent $3$-configurations we obtain on $\FFF_1$ (after blowing up a situation in $\PPP^2$) can be transformed to a tangent $3$-configuration in $\FFF_0$ after just one elementary link centered at $p$, and so we get a curve $C\subset\FFF_0$ of degree $(3,b)$ that has an $(1,\beta)$-Weierstrass point, where $S\sim S_-+\beta f\subset\FFF_0$.

  For example, Lemma~\ref{lemma:LowerBoundForN(3,2b)} implies the existence of a curve $C$ of degree $(3,3k)$ that has a $(1,5k)$-Weierstrass point for every $k\geq1$.
  It would be interesting to know under which circumstances it is possible to have two curves of degree $(1,a)$ and $(3,b)$ that intersect in one point only, especially if $a>b$.
\end{remark}

The following ``ingredients'' are tangent $3$-configurations in $\FFF_1$, which we obtain from configurations of curves in $\PPP^2$.
We start with two curves $S$ and $C$ in $\PPP^2$ that are very tangent at a point $p$. Then, we do a blow-up $\sigma:\FFF_1\to\PPP^2$ centered at a point $q\neq p$.
The following lemma gives conditions for $C$, $S$, $p$, and $q$ such that $(\tilde C,\tilde S,\bullet,\hat p)_1$ is a tangent $3$-configuration satisfying $\tilde C^2=6n-9\tilde S^2$, where $\hat p=\sigma^{-1}(p)$, $\tilde C$, $\tilde S$ are the strict transforms of $C$ respecitvely $S$, and $n=I_p(S,C)$.

\begin{lemma}\label{lemma:P2GivesRightSituationInF1}
  Let $p$ and $q$ be two distinct points on $\PPP^2$ and let $L$ be the line meeting both. Let $C$ and $S$ be two curves meeting $p$ that do not contain any line passing through $q$ and that satisfy the following conditions: \begin{enumerate}
    \item\label{item:P2GivesRightSituationInF1--smooth} $p$ is a smooth point of $C$,
    \item\label{item:P2GivesRightSituationInF1--mqC} $m_q(C)=\deg C-3$,
    \item\label{item:P2GivesRightSituationInF1--mqS} $m_q(S)=\deg S-1$,
    \item\label{item:P2GivesRightSituationInF1--Ip} $n:=I_p(S,C)=3\deg S+\deg C-3$.
  \end{enumerate}
  Let $\sigma:\FFF_1\to\PPP^2$ be the blow-up centered at $q\in \PPP^2$ and let $\tilde C$ denote the strict transform of $C$, and $\tilde S$ the one of $S$.
  Let $\hat p=\rho^{-1}(p)$.
  Then, $(\tilde C,\tilde S,\bullet,\hat p)_1$ is a tangent $3$-configuration with information \[[6\deg(C)-9, 2\deg(S)-1,\bullet,n;1].\] In particular, $\tilde C^2=6n-9\tilde S^2$.
\end{lemma}

\begin{proof}
  Since the strict transforms of the lines going through $q$ are the fibers in $\FFF_1$, $\tilde S$ and $\tilde C$ do not contain any fiber.
  In particular, the strict transform $\tilde L$ is a fiber.
  With \ref{item:P2GivesRightSituationInF1--mqC} we have $\tilde C\cdot \tilde L=C\cdot L-m_q(C)=\deg C -m_q(C)=3$, hence $\tilde C$ is a $3$-section.
  Similarly, we find with \ref{item:P2GivesRightSituationInF1--mqS} that $\tilde S$ is a section.
  To see that $\hat p$ is a transversal point of $\tilde C$ we insert \ref{item:P2GivesRightSituationInF1--mqC} and \ref{item:P2GivesRightSituationInF1--mqS} into
  \ref{item:P2GivesRightSituationInF1--Ip} and find \[ I_p(S,C)=3\left(m_q(S)+1\right)+m_q(C)\geq 3.\]
  Since $p$ is distinct from $q$ and $\hat p$ is a smooth point of $\tilde C$ by~\ref{item:P2GivesRightSituationInF1--smooth} (and $p$ is a smooth point of $S$ by~\ref{item:P2GivesRightSituationInF1--mqS}), we have $I_{\hat p}(\tilde C,\tilde S)=I_p(S,C)\geq 3$, and so $\tilde C$ is tangent to the section $\tilde S$ at $\hat p$.
  As a section intersects all fibers transversally, the $3$-section $\tilde C$ intersects the fiber $\tilde L$ transversally at $\hat p$.
  Hence, $\hat p$ is a transversal point of $\tilde C$ and so $(\tilde C,\tilde S,\bullet,\hat p)_1$ is a $3$-configuration.

  To prove that the $3$-configuration $(\tilde C,\tilde S,\bullet,\hat p)_1$ is tangent, we compute using \ref{item:P2GivesRightSituationInF1--mqC} and \ref{item:P2GivesRightSituationInF1--mqS} \begin{align*}
    \tilde S\cdot \tilde C&=\deg(S)\deg(C)-m_q(S)m_q(C)\\
    & = \deg S\deg C-(\deg S -1)(\deg C-3)\\
    & = 3\deg S +\deg C -3,
  \end{align*} so with \ref{item:P2GivesRightSituationInF1--Ip} we find that $\tilde S \cdot \tilde C=I_p(S,C)=I_{\hat p}(\tilde S,\tilde C)$, since a blow-up is outside the exceptional divisor an isomorphism.
  So $\tilde C$ and $\tilde S$ intersect only at the point $\hat p$.
  Therefore, it is a tangent $3$-configuration.

  Finally, with  \ref{item:P2GivesRightSituationInF1--mqC}
  we compute \[\tilde C^2=\deg C^2-m_q(C)^2=6\deg C-9\] and with \ref{item:P2GivesRightSituationInF1--mqS} and \ref{item:P2GivesRightSituationInF1--Ip} we find \begin{align*}
    6n-9\tilde S^2 & =6\left(3\deg S+\deg C-3\right)-9\left(\deg S^2-\left(\deg S-1\right)^2\right)\\
     & = 6\deg C-9,
  \end{align*} so $\tilde C^2=6n-9\tilde S^2$ holds.
\end{proof}

In the following we give specific examples of curves $S$, $C$ and $L$ in $\PPP^2$ with a point $p$ and a distinct point $q$ that will be blown up.
We will show that they satisfy the assumptions of Lemma~\ref{lemma:P2GivesRightSituationInF1} and can therefore apply Lemma~\ref{lemma:aCollectionGivesSpecialCollection} onto $\mathcal{C}=(\tilde C,\tilde S,\bullet,\hat p)_1$.
We determine of which type $k\geq -1$ the $a$-configuration $\mathcal{C}$ is, and then Lemma~\ref{lemma:3SectionAfterChainOfnLinks} gives a large singularity of type $A_K$.
If we want to achieve a case with $r=1$ or $r=2$, we will also add to the situation in $\PPP^2$ a line $T$ (going through $q$) and describe the intersection with $C$ at a point $t\in T$.
This corresponds then to a situation such as \ref{item:PolynomialVsDivisor--A} or \ref{item:PolynomialVsDivisor--B} in Lemma~\ref{lemma:PolynomialVsDivisor}, using the fact that the line $T$ is a fiber after a blow-up.

The examples with large singularites we give provide irreducible polynomials. This follows directly from Corollary~\ref{corollary:ReducibleDivisorsNotInteresting}.

\begin{lemma}
  There exists an irreducible polynomial of bidegree $(3,5)$ with a singularity of type $A_7$.
\end{lemma}

\begin{proof}
  Let us consider the following configuration in $\PPP^2$:
  \begin{center}
    \begin{tabular}{>{$}l<{$}  >{$}l<{$} }
      L: x=0, & q=[0:1:-1],\\
      S: x+z=0, & p=[0:1:0],\\
      C: z(x^2+xy+y^2)+xy(x+y)=0, & s=[0:0:1],\\
      T: y+z=0, & t=[1:0:0].
    \end{tabular}\begin{tikzpicture}[xscale=0.5,yscale=0.3,baseline=(b.base)]
    \draw (0,-1) -- (0, 7);
    \draw (1,-1) -- (-4, 1);
    \draw (-4,0) -- (1, 7);
    \node (b) at (0,2){};

    \draw (0,3) ..controls (-0.5,3.2) and (-1.1, 4).. (-1.9,3);
    \draw (-1.9,3) ..controls (-2.3,2.3) and (-3, 2).. (-4,2.5);
    \draw (0,3) ..controls +(2,0) and +(1, 2).. (0,3);
    \draw (0,3) ..controls +(-1,-3.2) and +(-1.7,-0.3).. (1,-0.5);

    \draw (-2.2,-0.5) node{\scriptsize{$S$}};
    \draw (-3.5,2.8) node{\scriptsize{$C$}};
    \draw (-2.2,3.2) node{\scriptsize{$t$}};
    \draw (-1.3,4.5) node{\scriptsize{$T$}};
    \draw (0.2,5.2) node{\scriptsize{$q$}};
    \draw (0.2,4) node{\scriptsize{$s$}};
    \draw (0.2,1.5) node{\scriptsize{$L$}};
    \draw (0.2,-1.2) node{\scriptsize{$p$}};

    \end{tikzpicture}
  \end{center}
  First, note that $L$ is the line meeting $p$ and $q$ and that $S$ and $C$ both contain $p$. Neither of them contains a line through $q$, since $q$ does not lie on $C$ nor on $S$.
  Let us now check the conditions \ref{item:P2GivesRightSituationInF1--smooth} to \ref{item:P2GivesRightSituationInF1--Ip} from Lemma~\ref{lemma:P2GivesRightSituationInF1}, which then gives us a $3$-configuration in $\FFF_1$.

  For \ref{item:P2GivesRightSituationInF1--smooth} we remark that $p\in C$ is a smooth point.
  For \ref{item:P2GivesRightSituationInF1--mqC}, \ref{item:P2GivesRightSituationInF1--mqS} and \ref{item:P2GivesRightSituationInF1--Ip}
  note that $q$ does not lie on $S\cup C$, so $m_q(S)=m_q(C)=0$. Since $C$ has degree 3 and $S$ has degree 1, \ref{item:P2GivesRightSituationInF1--mqC} and \ref{item:P2GivesRightSituationInF1--mqS} follow.
  Part \ref{item:P2GivesRightSituationInF1--Ip} holds because by inserting $ z = - x$ into $C$ we see that $C$ and $S$ intersect only at $p$, so $I_p(S,C) = \deg C\cdot\deg S = 3$.
  So Lemma~\ref{lemma:P2GivesRightSituationInF1} gives us a tangent $3$-configuration $\mathcal{C}=(\tilde C,\tilde S,\hat s,\hat p)_1$ equipped with information $[9,1,1,3;1]$, since $C$ has an $A_1$-singularity at $s$, which lies on $L$.

  We can thus apply Lemma~\ref{lemma:aCollectionGivesSpecialCollection} on a transversal $\mathcal{C}$-chain of $3$ links and get a disjoint $3$-configuration $\mathcal{C}'=(C',S_-,s',p')_{m'}$, where $m'=3-1=2$ and $C'\sim 3S_+\subset\FFF_2$.
  Lemma~\ref{lemma:InformationOfTangentAConfiguration} says that $\mathcal{C}'$ is of type $7$, since $2n+1=7$.
  Moreover, the line $T$ (containing $q$) intersects $C$ only at $t$. Therefore, $C'$ intersects a fiber $T'$ at only one point $t'$, giving $I_{t'}(T',C')=3$.
  So we are in case~\ref{item:PolynomialVsDivisor--A} of Lemma~\ref{lemma:PolynomialVsDivisor}.
  Finally, \ref{item:PolynomialVsDivisor--DP} of Lemma~\ref{lemma:PolynomialVsDivisor} asserts that there exists a polynomial of bidegree $(3,3m'-1)=(3,5)$ with a singularity of type $A_7$.
  This polynomial is irreducible by Corollary~\ref{corollary:ReducibleDivisorsNotInteresting}.
\end{proof}

\begin{lemma}
  There exists an irreducible polynomial of bidegree $(3,7)$ with a singularity of type $A_{10}$.
\end{lemma}

\begin{proof}
  Let us consider the following configuration in $\PPP^2$:
  \begin{center}
    \begin{tabular}{>{$}l<{$}  >{$}l<{$} }
      L: y+z=0, & q=[0:1:-1],\\
      S: y=0, & p=[1:0:0],\\
      C: y^2x^2+y(x^3+3x^2z+xz^2+z^3)+z^4=0, & s=[1:1:-1],\\
      T: x=0, & t=[0:1:0].
    \end{tabular}\begin{tikzpicture}[xscale=0.5,yscale=0.3,baseline=(b.base)]
    \draw (0,-1) -- (0, 7);
    \draw (1,-1) -- (-4, 1);
    \draw (-4,0) -- (1, 7);
    \node (b) at (0,2){};

    \draw (0,3) ..controls (-0.5,3.2) and (-1.1, 4).. (-1.9,3);
    \draw (-1.9,3) ..controls (-1.1,4) and (-2, 6).. (1,5.5);

    \draw (0,3) ..controls +(-1.5,-1.2) and +(-1.7,-1.1).. (1,-0.5);

    \draw (-2.2,-0.5) node{\scriptsize{$S$}};
    \draw (-3,2.1) node{\scriptsize{$T$}};
    \draw (-2.2,3.2) node{\scriptsize{$t$}};
    \draw (-1.5,5) node{\scriptsize{$C$}};
    \draw (0.2,5.2) node{\scriptsize{$q$}};
    \draw (0.2,3) node{\scriptsize{$s$}};
    \draw (0.2,1.5) node{\scriptsize{$L$}};
    \draw (0.2,-1.2) node{\scriptsize{$p$}};
    \end{tikzpicture}
  \end{center}
  We want to check that all assumptions of Lemma~\ref{lemma:P2GivesRightSituationInF1} are satisfied.
  First, note that $L$ is the line going through $p$ and $q$. Inserting the parametrisation of $S$ into $C$, we see that $S$ and $C$ intersect only at $p$, and $C$ is smooth at $p$, giving~\ref{item:P2GivesRightSituationInF1--smooth} and~\ref{item:P2GivesRightSituationInF1--mqC}.
  So we have \[n=I_p(S,C)=\deg C\cdot\deg S=4=3\deg S+\deg C-3,\] which is \ref{item:P2GivesRightSituationInF1--Ip}.
  Clearly, $S$ does not meet $q$ (implying~\ref{item:P2GivesRightSituationInF1--mqS}) so it does not contain any line going through $q$. If $C$ would contain a line going through $q$, then it had to be $L$ (otherwise, $C$ and $S$ would intersect also in a point distinct from $p$).
  Inserting the parametrisation of $L$ into $C$ yields a non-zero polynomial. Hence, $C$ does not contain a line going through $q$.

  Therefore, Lemma~\ref{lemma:P2GivesRightSituationInF1} gives us a tangent $3$-configuration $\mathcal{C}=(\tilde C,\tilde S,\hat s,\hat p)_1$ satisfying $\tilde C^2=6n-9\tilde S^2$ and equipped with information $[15,1,\bullet,4;1]$.
  Now, let us see that $\mathcal{C}$ is of type $2$:
  Using the change of coordinates $[x:y:z]\mapsto[x+y:y:z-y]$, which sends $[0:1:0]$ onto $s$, one sees that $C$ with changed coordinates has a cusp at $[0:1:0]$, hence $s$ is an $A_2$-singularity of $C$.

  Applying Lemma~\ref{lemma:aCollectionGivesSpecialCollection} on a transversal $\mathcal{C}$-chain of $n=4$ links, we get a disjoint $3$-configuration $\mathcal{C}'=(C',S_-,s',p')_{m'}$, where $m'=4-1=3$ and $C'\sim 3S_+$.
  By Lemma~\ref{lemma:InformationOfTangentAConfiguration}, $\mathcal{C}'$ is of type $2n+2=10$.

  The largest power of $y$ in the polynomial of $C$ is $y^2$ with unique tangent direction $x=0$, which corresponds to $T$.
  Hence, $m_t(C)=2$ and $T$ is the unique tangent direction to $C$ at $t$, and so there is a fiber $T'\subset\FFF_{m'}$ containing a point $t$ that is the only tangent direction to $C'$ at a point $t'$, which is case~\ref{item:PolynomialVsDivisor--B} of Lemma~\ref{lemma:PolynomialVsDivisor}.

  Using all the results we have collected, we can apply Lemma~\ref{lemma:PolynomialVsDivisor} onto $C'$, which asserts the existence of a polynomial of bidegree $(3,3m'-2)=(3,7)$ with a singularity of type $A_{10}$.
  Moreover, this polynomial is irreducible by Corollary~\ref{corollary:ReducibleDivisorsNotInteresting}.
\end{proof}

\begin{lemma}
  There exists an irreducible polynomial of bidegree $(3,8)$ with a singularity of type $A_{12}$.
\end{lemma}

\begin{proof}
  Let \[a=-\frac{3}{8}(i\sqrt{3}-1),~b=\frac{1}{8}(3i\sqrt{3}-1),~c=\frac{1}{2}(-3+i\sqrt{3}),~d=-\frac{1}{2}(3+i\sqrt3)\] and consider the following configuration in $\PPP^2$:
  \begin{center}
    \begin{tabular}{>{$}l<{$}  >{$}l<{$} }
      L: y=0, & q=[0:0:1],\\
      S: xz-y^2=0, & p=[1:0:0],\\
      C: z^3+(xz-y^2)(ax+by+cz)=0, & \\
      T: x+d y=0, & t=[-d:1:1].
    \end{tabular}\begin{tikzpicture}[xscale=0.5,yscale=0.3,baseline=(b.base)]
    \draw (0,-1) -- (0, 7);
    \draw (0.5,-1) -- (-3, 6);
    \node (b) at (0,2){};

    \draw (0,0) ..controls ++(170:1) and ++(170: 1).. (0,1.5)
    .. controls ++(-10:1) and ++(-10:1) .. (0,0);
    \draw (1,1.3) ..controls (-0.9,1.3) and (-1,3) .. (-0.5,3.5)
    .. controls (0.3,4.3) and (0.4,6) .. (-2,5)
    .. controls (-2.7,4.5) and (-3,6) ..(-2.5,6.5);

    \draw (-1.2,0.5) node{\scriptsize{$S$}};
    \draw (-0.3,-0.5) node{\scriptsize{$q$}};
    \draw (0.3,2.5) node{\scriptsize{$L$}};
    \draw (0.3,4.5) node{\scriptsize{$s$}};
    \draw (-1.1,5.7) node{\scriptsize{$C$}};
    \draw (-3,5.3) node{\scriptsize{$t$}};
    \draw (-2,3) node{\scriptsize{$T$}};
    \end{tikzpicture}
  \end{center}
  We want to check that the assumptions of Lemma~\ref{lemma:P2GivesRightSituationInF1} are satisfied.
  Note that $L$ is the line going through $p$ and $q$, and $p\in C$ is a smooth point, giving~\ref{item:P2GivesRightSituationInF1--smooth}.
  Clearly, $S$ is an irreducible conic, hence it does not contain any line.
  Since $C$ does not contain $q$ it does not contain any line going through $q$.
  So we also have~\ref{item:P2GivesRightSituationInF1--mqC} and~\ref{item:P2GivesRightSituationInF1--mqS}.

  Inserting the parametrisation of $S$ into $C$, we see that $S$ and $C$ intersect only at $p$ and compute $n=I_p(S,C)=\deg S\cdot \deg C=6$ and $3+3m_q(S)+m_q(C)=3+3=6$, so we have \ref{item:P2GivesRightSituationInF1--Ip}.

  So we apply Lemma~\ref{lemma:P2GivesRightSituationInF1} and get a tangent $3$-configuration $(\tilde C,\tilde S,\hat s,\hat p)_1$ satisfying $\tilde C^2=6n-9\tilde S^2$ and equipped with $[9,3,\bullet,6;1]$.
  Hence, we can apply Lemma~\ref{lemma:aCollectionGivesSpecialCollection} to a transversal $\mathcal{C}$-chain of $6$ links and get a $3$-configuration $\mathcal{C}'=(C',S',s',p')_{m'}$ where $m'=n-\tilde S^2=6-3=3$ and $C'\sim 3S_+\subset\FFF_3$.

  Note that we did not give a point ``$s$'' in the listing of the curves and points of the configuration, but we draw an ``$s$'' in the picture such that $L$ is the tangent to $C$ at this point. Such a point does exist because \[F(x,0,z)=\frac{1-i\sqrt {3}}{24}\, z \left( z\left(i\sqrt {3}+3\right)-3\,x \right) ^{2}.\]%
  Hence, $\mathcal{C}$ is of type $0$, and by Lemma~\ref{lemma:InformationOfTangentAConfiguration}, $\mathcal{C}'$ is of type $12$, because $2n=12$.

  By inserting $x=-d y$ into $C$, we find \[F(-d y,y,z)=(-y+z)^3,\] so $t$ is the only intersection point of $T$ and $C$, so we have $I_t(T,C)=3$ and get therefore a fiber $T'\subset\FFF_{m'}$ containing a point $t'$ with $I_{t'}(T',C')=3$. So we are in case~\ref{item:PolynomialVsDivisor--A} of Lemma~\ref{lemma:PolynomialVsDivisor}.

  Applying Lemma~\ref{lemma:PolynomialVsDivisor} gives the existence of a polynomial of bidegree $(3,3m'-1)=(3,8)$ with a singularity of type $A_{12}$.
  Finally, this polynomial is irreducible because of Corollary~\ref{corollary:ReducibleDivisorsNotInteresting}.
\end{proof}

\begin{lemma}
  There exists an irreducible polynomial of bidegree $(3,10)$ with a singularity of type $A_{15}$.
\end{lemma}

\begin{proof} Let \[F = \left( i-1 \right) {x}^{2}yz+
  \frac{1}{2}{x}^{3}(y-iz)
  + x\,y^2(- {x}+2\,z )
  -{y}^{2}{z}^{2} +
  \frac{1}{2}x{z}^{2}( \left( 1-3\,i \right) y+i{x})
\] and let \[G = x(iy+z)+ \left( 1-i \right) {y}^{2}- \left( 1+3\,i \right) yz-{z}^{2}.\] Consider the following configuration in $\PPP^2$:
  \begin{center}
    \begin{tabular}{>{$}l<{$}  >{$}l<{$} }
      L: y=0 , & q=[1:0:1],\\
      S: G=0, & p=[1:0:0],\\
      C: F=0,& s=[0:0:1],\\
      T: x-z=0, &t=[0:1:0].\\
    \end{tabular}
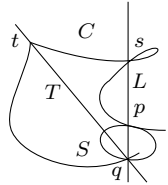
\begin{tikzpicture}[xscale=0.5,yscale=0.3,baseline=(b.base)]
    \draw (0,-1) -- (0, 7);
    \draw (0.5,-1) -- (-3, 6);
    \node (b) at (0,2){};

    \draw (0,0) ..controls ++(170:1) and ++(170: 1).. (0,1.5)
    .. controls ++(-10:1) and ++(-10:1) .. (0,0);
    \draw (1,1.3) ..controls (-0.9,1.3) and (-1,3) .. (-0.5,3.5)
    .. controls (1.4,7) and (1.6,2.7) .. (-2.6,5.2)
    .. controls (-2.6,4) and (-3.5,2) .. (-3,1)
    .. controls (-2.6,0) and (-1,-1) .. (0.3,0.3);

    \draw (-1.2,0.5) node{\scriptsize{$S$}};
    \draw (-0.3,-0.6) node{\scriptsize{$q$}};
    \draw (0.3,3.5) node{\scriptsize{$L$}};
    \draw (0.3,2.2) node{\scriptsize{$p$}};
    \draw (0.3,5.2) node{\scriptsize{$s$}};
    \draw (-1.1,5.7) node{\scriptsize{$C$}};
    \draw (-3,5.3) node{\scriptsize{$t$}};
    \draw (-2,3) node{\scriptsize{$T$}};
    \end{tikzpicture}
  \end{center}
  We want to prove that the assumptions of Lemma~\ref{lemma:P2GivesRightSituationInF1} are satisfied.
  The line going through $p$ and $q$ is $L$.
  For \ref{item:P2GivesRightSituationInF1--smooth} note that $p\in C$ is smooth, since $F$ contains the term $x^3$.
  One can also check that $q\in C$ is a smooth point, so we have $m_q(C)=1=\deg C-3$, which gives \ref{item:P2GivesRightSituationInF1--mqC}.
  The conic given by $G$ is smooth (since $G$ can be written as ${x}{\alpha(y,z)}+\beta(y,z)$ for some $\alpha,\beta\in\CCC[y,z]$), so $q\in S$ is smooth and we have $m_q(S)=1=\deg S-1$. So \ref{item:P2GivesRightSituationInF1--mqS} holds.
  This also implies that $S$ is irreducible and does thus not contain any lines. We still need to prove that $C$ does not contain any line meeting $q$ (which we will do later in the proof), and that~\ref{item:P2GivesRightSituationInF1--Ip} holds.

  To show \ref{item:P2GivesRightSituationInF1--Ip} consider the parametrisation $\varphi:\PPP^2\to\PPP^2$ of $S$, which is given by \[
    \varphi([y:z]) = [\left( -1+i \right) {y}^{2}+ \left( 1+3\,i \right) yz+{z}^
    {2} : y \left( z+iy \right) : z \left( z+iy \right)].
  \]
  Inserting this into $F$ gives \[\left( 1+i \right) y \left( iz-y \right) ^{7},\] hence $S$ and $C$ intersect at two points: at $[0:1]$ with local intersection $1$, and at $[i:1]$ with local intersection $7$.
  Note that $\varphi([i:1])=[1:0:0]=p$.
  Therefore, $I_p(C,S)=7$ and so we have $3\deg S +\deg C-3=7=I_p(S,C)$, implying \ref{item:P2GivesRightSituationInF1--Ip}.

  We show that $C$ has a node at $s=[0:0:1]$.
  Since $z^4$ and $z^3$ do not appear in $F$, the quartic $C$ has a singular point at $s$.
  It is a node, because the coefficient of $z^2$ is \[
    1/2 \left( -2\,{y}^{2}+xy(1-3i)+i{x}^{2} \right),
  \] which has discriminant $-2+\frac{i}{2}\neq 0$.
  Hence, it has two distinct roots, and $s$ is an $A_1$-singularity of $C$.

  Note that $t$ is a singular point of $C$, since $y^3$ and $y^4$ do not appear in $F$.
  Moreover, the coefficient of $y^2$ is $- \left( x-z \right) ^{2}$, which has the unique tangent direction $T$ and is hence a cusp.

  Instead of proving that $C$ does not contain any line meeting $q$, we are now ready to prove that $C$ is irreducible, which is a stronger statement.
  Since we know that $t\in C$ is a cusp, this singularity needs to come from an irreducible component $C_1$. Hence, this component needs to be of degree $3$ or $4$.
  If it is of degree $4$, we are done.
  So assume that it is of degree $3$.
  Then, $C=C_1+C_2$, where $C_2$ is irreducible and of degree $1$.
  Having $I_p(S,C_1)\leq 6$, we achieve $7=I_p(S,C)=I_p(S,C_1)+I_p(S,C_2)$ only if $I_p(S,C_2)\geq1$. Hence, $p\in C_2\,\cap\, C_1$ and so $p$ is a singular point of $C$. This is a contradiction to~\ref{item:P2GivesRightSituationInF1--smooth} and therefore, $C$ is irreducible.

  We have now proven that all assumptions of Lemma~\ref{lemma:P2GivesRightSituationInF1} are satisfied and hence $(\tilde C,\tilde S,\hat s,\hat p)_1$ is a tangent $3$-configuration that satisfies $\tilde C^2=6n-9\tilde S^2$ and is equipped with $[15,3,1,7;1]$.
  Applying Lemma~\ref{lemma:aCollectionGivesSpecialCollection} on a transversal $\mathcal{C}$-chain of $7$ links, we get that there is a disjoint $3$-configuration $\mathcal{C}'=(C',S_-,s',p')_{m'}$ where $m'=n-\tilde S^2=4$ and $C'\sim S_+$.
  By Lemma~\ref{lemma:InformationOfTangentAConfiguration}, $\mathcal{C}'$ is of type $15$, because $2n+1=15$.
  Moreover, there is a fiber $T'$ containing a point $t'$ with $m_{t'}(C')=2$ and such that $T'$ is the only tangent direction to $C'$ at $t'$.
  Therefore, $C'$ satisfies~\ref{item:PolynomialVsDivisor--B} of Lemma~\ref{lemma:PolynomialVsDivisor}.
  By applying Lemma~\ref{lemma:PolynomialVsDivisor} we get a polynomial of bidegree $(3,3m'-2)=(3,10)$ that has a singularity of type $A_{15}$. This polynomial is irreducible by Corollary~\ref{corollary:ReducibleDivisorsNotInteresting}.
\end{proof}

\begin{lemma}
  There exists an irreducible polynomial of bidegree $(3,11)$ with a singularity of type $A_{17}$.
\end{lemma}

\begin{proof}
  Let $\omega=i\sqrt{3}$ and let \begin{align*}
    F = & {z}^{3}+
    \frac{3}{8}\left(\omega+3\right)\left(y-x\right)z^2+
    \frac{9}{8}\left( {\frac {\omega}{2}}+{1} \right) {x}^{2}z+
    \frac{3}{64}\left( -{7\,\omega}-{3} \right) xyz\\
    & + \frac{3}{64}\left( {5\,\omega}-{3} \right) {y}^{2}z+
    \frac{3}{32}\left( -{\frac {5\,\omega}{2}}-{3} \right) {x}^{3}+
    \frac{9}{32}\left( {\frac {\omega}{2}}-{1} \right) {x}^{2}y.
 \end{align*}
 Consider the following configuration in $\PPP^2$: \begin{center}
   \begin{tabular}{>{$}l<{$}  >{$}l<{$} }
     L: x=0, & q=[0:0:1],\\
     S: \left( -2\,x+{y} \right) yz+ \frac{1}{4}\left( \omega+{3} \right) {x}^ {2}y-{x}^{3}=0, & p=[0:1:0].\\
     C: F=0, & \\
     T: x+y=0, &
   \end{tabular}
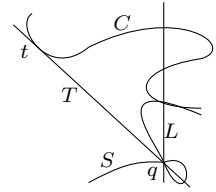
\begin{tikzpicture}[xscale=0.5,yscale=0.3,baseline=(b.base)]
   \draw (0,-1) -- (0, 7);
   \draw (0.7,-1.2) -- (-4, 6);
   \node (b) at (0,2){};

   \draw (1,2) .. controls (-1.8,4) and (-0.1,0.5) ..(0,-0.1)
   .. controls (0.6,-2.5) and (1,0.5) ..(0,-0.1)
   .. controls (-0.8,0) and (-1.2,-0.3) ..(-2,-1);

   \draw (1,2.3) ..controls (-0.9,2.3) and (-1,4) .. (1,4.5)
   .. controls (2,5) and (0.4,7) .. (-2,5)
   .. controls (-3.1,3.5) and (-4,6) ..(-3.5,6.5);

   \draw (-1.5,0) node{\scriptsize{$S$}};
   \draw (-0.3,-0.6) node{\scriptsize{$q$}};
   \draw (0.2,1.3) node{\scriptsize{$L$}};
   \draw (-1.1,6.1) node{\scriptsize{$C$}};
   \draw (-3.7,4.8) node{\scriptsize{$t$}};
   \draw (-2.5,2.9) node{\scriptsize{$T$}};
   \end{tikzpicture}
 \end{center}
 We check that the assumptions of Lemma~\ref{lemma:P2GivesRightSituationInF1} are satisfied.
 First of all, $L$ is the line meeting $p$ and $q$.
 For \ref{item:P2GivesRightSituationInF1--smooth} we note that $y^3$ does not appear in $F$, but $y^2$ does.
 So $p$ is a smooth point of $C$.

 For \ref{item:P2GivesRightSituationInF1--mqC} we have $m_q(C)=0$, since $z^3$ appears in $F$.
 Hence \ref{item:P2GivesRightSituationInF1--mqC} holds.
 We see that $m_q(S)=2$, so also \ref{item:P2GivesRightSituationInF1--mqS} holds.

 Since the polynomial defining $S$ can be written in the form $z\alpha(x,y)+\beta(x,y)$ for some $\alpha,\beta\in\CCC[x,y]$, $S$ is irreducible and does therefore contain no lines.

 For \ref{item:P2GivesRightSituationInF1--Ip} we need to know what $n=I_p(S,C)$ is.
 By plugging the parametrization of $S$ into $F$, we find (with the help of a computer algebra program) that $p$ is the only intersection point of $S$ and $C$, hence $n=I_p(S,C)=3\cdot3=9$ and we find $3\deg S+\deg C-3=9=n$.

 Now, we can prove that $C$ does not contain any line meeting $q$. If it would, then the line needs to be $L$ (otherwise, $S$ and $C$ would meet also in a point distinct from $p$). We insert $x=0$ into $C$ and see that it is not the zero polynomial. Therefore, $C$ does not contain $L$ and so does not contain any line meeting $q$.

 We have shown that the assumptions of Lemma~\ref{lemma:P2GivesRightSituationInF1} are satisfied. The lemma implies that $\mathcal{C}=(\tilde C,\tilde S,\hat s,\hat p)_1$ is a tangent $3$-configuration that satisfies $\tilde C^2=6n-9\tilde S^2$ and is equipped with $[9,5,\bullet,9;1]$.
 Note that the line $L$ intersects $C$ besides $p$ at two more points, because the discriminant of $F(0,1,z)$ divided by $z$ is $\frac{3}{32}(15-\omega)\neq 0$. Hence $\mathcal{C}$ is of type $-1$.
 So we can apply Lemma~\ref{lemma:aCollectionGivesSpecialCollection} on a transversal $\mathcal{C}$-chain of $9$ links and get a disjoint $3$-configuration $\mathcal{C}'=(C',S_-,s',p')_{m'}$, where $m'=n-\tilde S^2=4$ and $C'\sim 3S_+$.
 By Lemma~\ref{lemma:InformationOfTangentAConfiguration}, $\mathcal{C}'$ is of type $17$, because $2n-1=17$.

 Now remark that the line $T$ intersects $C$ at only one point, because \[F(-y,y,z)=-\frac{\omega}{72}(-3y+(\omega -3)z)^3.\]%
 So there is a fiber $T'$ containing a point $t'$ with $I_{t'}(T',C')=3$ and we are in case~\ref{item:PolynomialVsDivisor--A} of Lemma~\ref{lemma:PolynomialVsDivisor}.
 Therefore, Lemma~\ref{lemma:PolynomialVsDivisor} implies the existence of a polynomial of bidegree $(3,3m'-1)=(3,11)$ with a singularity of type $A_{17}$. This polynomial is irreducible by Corollary~\ref{corollary:ReducibleDivisorsNotInteresting}.
\end{proof}
\begin{lemma}
  There exists an irreducible polynomial of bidegree $(3,12)$ with a singularity of type $A_{18}$.
\end{lemma}

\begin{proof}
  Let $\omega=i\sqrt{3}$ and let \begin{align*}
    F= &{z}^{3}+ \frac{9}{2}\left( -1+\omega \right) {x}^{3}-9\,y{x}^{2}+9\,z{
    y}^{2}+ 3\left( -\omega+3 \right) xyz-\\
    & 6\,y{z}^{2}-3\,x{z}^{2}+
     \frac{3}{2}\left( -\omega+5 \right) {x}^{2}z, \\
    G= &yz\left( x+{y} \right) + \frac{1}{2}\left( -1+\frac{\omega}{3} \right) {x}^{3}-{x}^{2}y.
  \end{align*}
  Consider the following configuration in $\PPP^2$: \begin{center}
    \begin{tabular}{>{$}l<{$}  >{$}l<{$} }
      L: x=0, & q=[0:0:1],\\
      S: G=0, & p=[0:1:0],\\
      C: F=0, & s=[0:1:3].
    \end{tabular}\begin{tikzpicture}[xscale=0.5,yscale=0.3,baseline=(b.base)]
    \draw (0,-2) -- (0, 6);
    \node (b) at (0,1){};

    \draw (1,1) .. controls (-1.8,2.3) and (-0.1,-0.5) ..(0,-1.1)
    .. controls (0.6,-3.5) and (1,-0.5) ..(0,-1.1)
    .. controls (-0.8,-1) and (-1.2,-1.3) ..(-2,-2);

    \draw (1,1.3) ..controls (-0.9,1) and (-1,2.8) .. (-0.5,3.5)
    .. controls (0.3,4.3) and (0.4,6) .. (-2,5);

    \draw (-1.5,-1) node{\scriptsize{$S$}};
    \draw (-0.3,-1.6) node{\scriptsize{$q$}};
    \draw (0.2,2.3) node{\scriptsize{$L$}};
    \draw (0.2,4.5) node{\scriptsize{$s$}};
    \draw (-1.1,5.9) node{\scriptsize{$C$}};
    \end{tikzpicture}
  \end{center}
  We want to check that the assumptions of Lemma~\ref{lemma:P2GivesRightSituationInF1} are satisfied.
  The line meeting $p$ and $q$ is $L$. Since $G$ can be written as $z\alpha(x,y,)+\beta(x,y)$ for some $\alpha,\beta\in\CCC[x,y]$, $S$ is irreducible and contains therefore no lines.
  One sees that $p$ is a smooth point of $p$, so we have~\ref{item:P2GivesRightSituationInF1--smooth}.
  The multiplicities of $S$ respectively $C$ at $q$ are $m_q(S)=2$ and $m_q(C)=0$.
  So \ref{item:P2GivesRightSituationInF1--mqC} and \ref{item:P2GivesRightSituationInF1--mqS} hold.

  By plugging the parametrization of $S$ into $F$, with the help of a computer algebra program we find that $S$ and $C$ intersect at only one point, namely at $p$.
  Therefore, we have $n=I_p(S,C)=3\cdot 3=9$.
  We compute \ref{item:P2GivesRightSituationInF1--Ip}: $ 3 + 3 m_q(S) + m_q(C) = 3 + 6 = 9 = n $.

  We check now that $C$ does not contain any line meeting $q$. If it would contain such a line, then it would be $L$ (otherwise, $C$ intersects $S$ in a second point). Inserting the parametrisation of $L$ into $F$ gives the polynomial $F(0,y,z)\neq0$. Hence, $C$ does not contain $L$.

  We have proven that the assumptions of Lemma~\ref{lemma:P2GivesRightSituationInF1} are satisfied and get a $3$-configuration $\mathcal{C}=(\tilde C,\tilde S,\hat s,\hat p)_{1}$ that satisfies $\tilde C^2=6n-9\tilde S^2$ and is equipped with $[9,5,\bullet,9;1]$.
  Note that the line $L$ is tangent to $C$ at $s$, as \[F(0,1,z)=z(z-3)^2.\]%
  Hence, $\mathcal{C}$ is of type $0$.
  We apply Lemma~\ref{lemma:aCollectionGivesSpecialCollection} on a transversal $\mathcal{C}$-chain of $9$ links and get a disjoint $3$-configuration $(C',S_-,s',p')_{m'}$ where $m'=n-\tilde S^2=4$ and $C'\sim 3S_+\subset\FFF_4$, which is of type $k=18$ by Lemma~\ref{lemma:InformationOfTangentAConfiguration}.

  Therefore, by applying Lemma~\ref{lemma:PolynomialVsDivisor} there exists a polynomial of bidegree $(3,3m')=(3,12)$ with a singularity of type $A_{18}$.
  The polynomial is irreducible by Corollary~\ref{corollary:ReducibleDivisorsNotInteresting}.
\end{proof}

\bigskip

It remains to provide a lower bound for $N(3,4)$ and $N(3,6)$.
In these cases, it is not difficult to construct ``ingredients'' and apply our method.
However, we leave this as an exercise to the interested reader since a family of examples that we learned from Peter Feller gives a lower bound for $N(3,b)$ for \textit{all} even $b$.
It gives a specific polynomial of bidegree $(3,2n)$ with a singularity of type $A_{3n-1}$.
For $N(3,4)$ and $N(3,6)$ it gives the optimal bound as we will see. (For $b=8,10,12$ the bounds we have found are better.)

\begin{lemma}\label{lemma:LowerBoundForN(3,2b)}
  Let $b$ be any integer. Then $N(3,2b)\geq 3b-1$.
\end{lemma}

\begin{proof}
  By an example of Feller, the curve $y^a-(x^b-y)^2=0$ in $\AAA^2$ is of bidegree $(a,2b)$ with an $A_{ab-1}$-singularity.
  We check that $y^3-(x^b-y)^2=0$ has indeed an $A_{3b-1}$-singularity.
  The change of coordinates $y\mapsto y+x^b$ gives $(y+x^b)^3-y^2=0$. Locally, the blow-up at $(0,0)$ is given by $(x,y)\mapsto(x,xy)$, so after $b$ blow ups we get \[(x^by+x^b)^3-x^{2b}y^2=x^{2b}(x^b(y+1)^3-y^2). \]%
  With the analytic local coordinate change $x\mapsto x(y+1)^{\frac{1}{b}}$  we get $x^b-y^2$, which has an $A_{b-1}$-singularity.
  As we did $b$ blow-ups, the curve we started with has an $A_{(b-1)+2b}=A_{3b-1}$-singularity.
\end{proof}

\begin{corollary}\label{corollary:NFor3-4And3-6}
  $N(3,4)\geq 5$ and $N(3,6)\geq 8$.
\end{corollary}

\begin{proof}
  Follows directly from Lemma~\ref{lemma:LowerBoundForN(3,2b)}.
\end{proof}

\begin{remark}\label{remark:LowerBound}
  We have therefore the following lower bounds (LB) for $N(3,b)$:
  \begin{center}
    \begin{tabular}{ >{$}r<{$} || >{$}c<{$} | >{$}c<{$} | >{$}c<{$}| >{$}c<{$}| >{$}c<{$}| >{$}c<{$}| >{$}c<{$} | >{$}c<{$} | >{$}c<{$} | >{$}c<{$} | >{$}c<{$} }
      b &  3 & 4 & 5 & 6 & 7 & 8 & 9 & 10 & 11 & 12\\ \hline
      \text{LB for }N(3,b) & 3 & 5 & 7 & 8 & 10 & 12 & 13 & 15 & 17 & 18. \\
    \end{tabular}
  \end{center}
\end{remark}

\section{... Or Not To Be}\label{section:OrNotToBe}

In this chapter we find an upper bound for $N(3,b)$ for small $b$.
We use the ``recipe'' of Section~\ref{section:ToBe--Recipe} but in converse direction. First, we verify in Section~\ref{subsection:OrNotToBe--TheRecipe} that we are allowed to go backwards, and then determine in Section~\ref{subsection:OrNotToBe--NonExistence} that the configurations we get do not occur.

\subsection{The recipe}\label{subsection:OrNotToBe--TheRecipe}

We start with a polynomial $F$ of bidegree $(3,3m)$ with a large $A_k$-singularity.
By Lemma~\ref{lemma:ReducibleDivisorsNotInteresting}, its $(3,3m)$-divisor $C$ is irreducible and hence does not contain any fiber.
Having $C\sim 3S_+$, it follows that $C$ is a $3$-section, and $C\cap S_-=\emptyset$. Thus, $(C,S_-,\bullet,\bullet)_m$ is a disjoint $3$-configuration.

Recall that in Lemma \ref{lemma:3-SectionSituation} we have observed that as soon as $k\geq 3$, the $A_k$-singularity at $s$ has a cofibered point and so $(C,\bullet,s,\bullet)_m$ is a $3$-configuration.

We will use the recipe of Section~\ref{section:ToBe--Recipe} in a converse direction. To say it metaphorically: Instead of cooking ``easy'' ingredients into very singular curves, we ``de-cook'' singular curves into ``easy'' configurations (and later show that these configurations do not exist).

\begin{lemma}\label{lemma:SingularPointOfA3SectionHasATransversalPoint}
 Let $m\geq0$ and $k\geq 1$ be two integers and let $\mathcal{C}=(C,\bullet,s,\bullet)_m$ be a $3$-configuration of type $k\geq1$.
 Let $\pi:\FFF_{m}\dashrightarrow\FFF_{m'}$ be an $s$-link with inverse point $p'$.
 Then, $(\pi_*(C),\bullet,\bullet,p')$ is a $3$-configuration, which we will call the \textit{direct image of $\mathcal{C}$} and denote it by $\pi_*(\mathcal{C})$.
 Moreover, $\pi_*(\mathcal{C})$ is of type $k-2\geq-1$.
\end{lemma}

\begin{proof}
 Let $\rho:X\to\FFF_m$ be the blow-up centered at $s$, so we can assume that $\sigma=\pi\circ\rho:X\to\FFF_{m'}$ is the contraction of the strict transform $\tilde f$.
 Let $E$ denote the exceptional divisor of $\rho$. Since $C\cap f=\{p,s\}$ we have that $\tilde f\cap \tilde C=\{\hat p\}$, where $\hat p=\rho(p)$ is a transversal point of $\tilde C$.
 Moreover, the intersection of $E$ with $\tilde f$ is transversal, since $I_s(C,f)=2=m_s(C)$.
 Thus, contracting $\tilde f$ onto $p'$ makes $p'$ a transversal point of $C'=\pi_*(C)$.
 So, $\pi_*(\mathcal{C})$ is indeed a $3$-configuration.

 The fact that $\pi_*(\mathcal{C})$ is of type $k-2$ follows from Lemma~\ref{lemma:3SectionAfterChainOfnLinks} because $\pi$ is the inverse of a $p'$-link with inverse point $s$ (which is a transversal $\pi_*(\mathcal{C})$-chain of $1$ link).
\end{proof}

This leads us to the inverse of a transversal $\mathcal{C}$-chain of links, which we will call a \textit{singular} $\mathcal{C}$-chain of links (because it is obtained by blowing up a singular point), described in the following definition. However, after an $s$-link the type drops from $k$ to $k-2$. Writing $k=2n$ respectively $k=2n-1$ depending on $k$ even or odd, we can apply Lemma~\ref{lemma:SingularPointOfA3SectionHasATransversalPoint} at most $n=\ceil{\frac{k}{2}}$ times.
Then, the integer $k$ is $\leq 0$ and we cannot continue the process.
Remark that $n\leq\ceil{\frac{k}{2}}$ is equivalent to $k-2n\geq-1$.

\begin{definition}\label{definition:SingularChainOfNLinks}
  Let $k=k_0\geq 1$, $m=m_0\geq 0$ and $1\leq n\leq \ceil{\frac{k}{2}}$ be three integers.
  Let $\mathcal{C}=\mathcal{C}_0=(C=C_0,\bullet,s=s_0,p=p_0)_{m}$ be a $3$-configuration of type $k\geq1$.
  For $i=1,\ldots,n$ let $\pi_i:\FFF_{m_{i-1}}\dashrightarrow\FFF_{m_i}$ be a $s_{i-1}$-link with inverse point $p_i$ and let $\mathcal{C}_i=\pi_*(\mathcal{C}_{i-1})$ be as in Lemma~\ref{lemma:SingularPointOfA3SectionHasATransversalPoint}.
  We say that the composition $\pi=\pi_n\circ\cdots\circ\pi_1:\FFF_{m_0}\dashrightarrow\FFF_{m_n}$ is a \textit{singular $\mathcal{C}$-chain of $n$ links} and we say that $\mathcal{C}_n$ is the \textit{direct product $\pi_*(\mathcal{C})$ of $\mathcal{C}$}.
  The singular $\mathcal{C}$-chain of $n$ links is the inverse of a transversal $\pi_*(\mathcal{C})$-chain of $n$ links.
\end{definition}

\begin{lemma}\label{lemma:InformationOf3ConfigurationAfterSingularChain}
  Let $m\geq 0$, $k\geq 1$ and $1\leq n\leq \ceil{\frac{k}{2}}$ be three integers.
  Let $\mathcal{C}=(C,S,\bullet,p)_m$ be a disjoint or tangent $3$-configuration of type $k\geq1$ such that $I_p(S,C)\geq n$.
  Let $\pi:\FFF_m\dashrightarrow\FFF_{m'}$ be a singular $\mathcal{C}$-chain of $n$ links.
  Then, the $3$-configuration $\pi_*(\mathcal{C})$ is tangent and is equipped with the information \[[C^2-3n, S^2+n,k-2n, I_p(S,C)+n;\bullet].\]
\end{lemma}

\begin{proof}
  Note that it is enough to prove that $\pi_*(\mathcal{C})=(C',S',\bullet,p')_{m'}$ is tangent.
  Knowing this, we can apply Lemma~\ref{lemma:InformationOfTangentAConfiguration} onto the inverse of the singular $\mathcal{C}$-chain, which is a transversal $\mathcal{C}$-chain and find that $\mathcal{C'}$ is of type $k-2n\geq -1$.
  The other statements follow analogously.

  We now show that $\pi_*(\mathcal{C})$ is tangent.
  Since $\mathcal{C}$ is tangent or disjoint, $S$ and $C$ intersect at most in $p$. In any case, the section $S$ intersects the fiber $f$ meeting $p$ either in $p$, or in another point, say $r$.
  By blowing up $p$ and the contracting the strict transform $\tilde f$ onto $p'$, $C'$ and $S'$ intersect in $p'$ (and only in $p'$).
  Hence, $\mathcal{C}'$ is tangent.
\end{proof}

\begin{remark}\label{remark:3SectionGetsEventuallySmooth}
  Let $C\subset\FFF_m$ be a $3$-section with an $A_k$-singularity at $s$ that is smooth elsewhere.
  Assume that $s$ has a $C$-cofibered point.
  Then, Lemma~\ref{lemma:InformationOf3ConfigurationAfterSingularChain} implies the existence of a birational map $\pi:\FFF_m\dashrightarrow\FFF_{m'}$, composed by $n=\ceil{\frac{k}{2}}$ links, such that $\pi_*(C)$ is smooth.
\end{remark}

To show non-existence of a polynomial of bidegree $(3,3m-r)$ we will assume it exists and find its $(3,3m)$-divisor $C_0$ with Lemma~\ref{lemma:PolynomialVsDivisor},\ref{item:PolynomialVsDivisor--PD}.
We take $S_0$ to be the $(-m)$-curve and the following Lemma~\ref{lemma:SpecialInformationAfterSingularChain} will assert that we arrive in $\FFF_1$ and the curve we get is (almost) smooth.
Then, we will contract the $(-1)$-curve and get a situation in the projective plane $\PPP^2$.
Finally, we will show that this situation does not exist in certain cases.

The following lemma has some assumptions that are specific to the cases that we want to study.
To arrive in some $\FFF_{l}$, where $l$ is odd, we need that the parity of $m$ and the number of blow-ups, $n$, differ.
To make sure we arrive in $\FFF_1$, we have to assume that ``$C^2-3n\leq 17$'' holds.
However, we cannot have always that after $n$ links, we arrive in a smooth situation: Sometimes, we achieve only an $A_1$- or $A_2$-singularity.

\begin{lemma}\label{lemma:SpecialInformationAfterSingularChain}
  Let $m\geq0$, $k\geq 1$ and $1\leq n\leq \ceil{\frac{k}{2}}$ be integers such that $n-m$ is odd.
  Let $\mathcal{C}=(C,S_-,\bullet,p)_m$ be a $3$-configuration of type $k\geq1$ such that $C\sim3S_+$ is an irreducible $3$-section.
  Assume that $C^2-3n\leq 17$.
  Let $\varphi:\FFF_{m}\dashrightarrow\FFF_{m'}$ be a singular $\mathcal{C}$-chain of $n$ links.
  Then, $\pi_*(\mathcal{C})$ is equipped with \[[9m-3n,-m+n,k-2n,n;1].\]
\end{lemma}

\begin{proof}
  Write $\pi_*(\mathcal{C})=(C',S',\bullet,p')_{m'}$.
  Since $C\sim3S_+$ is irreducible, the $3$-configuration $\mathcal{C}$ is disjoint and hence with Lemma~\ref{lemma:InformationOf3ConfigurationAfterSingularChain} we find that $I_{p'}(S',C')=n$.
  Having $S_-^2=-m$, we find (again with Lemma~\ref{lemma:InformationOf3ConfigurationAfterSingularChain}) that $S'^2=-m+n$, which is odd by assumption. Hence, $m'$ is odd because $S'$ is a section.
  Using the same lemma, we get $C'^2=C^2-3n=9m-3n$ and by assumption $C'^2\leq 17$. Since we assumed $C$ to be irreducible, the $3$-section $C'$ is irreducible, too, and we can write $C'\sim 3S_-+bf\subset\FFF_{m'}$ for some integer $b\geq 3m'$.
  This gives $17\geq C'^2=-9m'+6b\geq9m'$ and so $m'\leq 1$. Since we already know that $m'$ is odd, $m'=1$ follows.
\end{proof}

\subsection{The non-ingredients}\label{subsection:OrNotToBe--NonExistence}

In this section, we give an upper bound for $N(3,9)$ and $N(3,12)$.
We assume that there is a curve with a larger $A_k$-singularity and find a contradiction to the configuration in $\PPP^2$ that we obtain with the ``recipe'' from Section~\ref{subsection:OrNotToBe--TheRecipe}.

\begin{lemma}\label{lemma:NonExistenceOfPolynomialsWhereMIs3}
  There is an upper bound $N(3,9)\leq 13$. In particular, there is no polynomial of bidegree $(3,9)$ with a singularity of type $A_{14}$.
\end{lemma}

\begin{proof}
  Let us assume that a polynomial of bidegree $(3,9)$ with an $A_k$-singularity and $k\geq 14$ exists.
  If the $(3,9)$-divisor $C_0$ of the polynomial is reducible, we already know by Lemma~\ref{lemma:ReducibleDivisorsNotInteresting} that it cannot have such a singularity.
  So $C_0$ is irreducible and by Lemma~\ref{lemma:GenusUpperBoundForIrreducibleDivisors} we know that $k\leq 14$. So assume that $C_0$ has an $A_k$-singularity at a point $s_0$ with $k=14$.
  We consider the $3$-configuration $\mathcal{C}_0=(C_0,S_-,s_0,p_0)_3$, which is disjoint because $C_0\sim 3S_+$ is irreducible.
  Let $\pi:\FFF_m\dashrightarrow\FFF_{m'}$ be a singular $\mathcal{C}_0$ chain of $n=6$ links. Note that $n-m=3$ is odd, that $6<7=\frac{14}{2}$ and that $C_0^2-3n=27-18=9<17$.
  Together with Lemma~\ref{lemma:InformationOf3ConfigurationAfterSingularChain}, we find that $\mathcal{C}_n=\pi_*(\mathcal{C}_0)$ is equipped with \[[9,3,2,6;1],\] where we found $A_2$ because $14-2n=2$.
  Knowing the self-intersection of the $3$-section $C_n$ and the section $S_n$ on $\FFF_1$, we find $C_n\sim 3S_-+3f$ and $S_n\sim S_-+2f$.
  As $C_n$ is irreducible, $C_n$ and $S_-$ do not intersect. In particular, $p_n$ is not contained in $S_-$. Let $\rho:\FFF_1\to\PPP^2$ be the contraction of the $(-1)$-curve $S_-$ onto a point $q$ in $\PPP^2$.
  Let $C=\rho_*(C)$, $S=\rho_*(S)$, $p=\rho(p_n)\neq q$ and $s=\rho(s_n)\neq q$.
  We have $C^2=C_n^2=9$ (so $C$ is a cubic with $m_q(C)=0$) and $S^2=S_n^2+1$ because $S_n\cdot S_-=1$ and thus $m_q(S)=1$ (so $S$ is a conic going through $q$).
  Moreover, $C$ has a singularity of type $A_2$ at $s$, and since $s_n$ and $p_n$ are cofibered, the three distinct points $s$, $p$ and $q$ are collinear.
  To recapitulate, we have a cubic $C$ with a cusp at $s\in\PPP^2$ that intersects a conic $S$ at $p$ with $I_p(S,C)=6$, so they intersect only at one point $p\neq s$.
  This is not possible by Lemma~\ref{lemma:NonExistenceCubicConicConfiguration}.
\end{proof}

\begin{lemma}\label{lemma:NonExistenceCubicConicConfiguration}
  Let $C\subset\PPP^2$ be an irreducible cubic curve that has a singularity at a point $s$, and let $S$ be an irreducible conic that intersects $C$ at exactly one point $p$ at which $C$ is smooth.
  Then, the singularity of $C$ is a node.
\end{lemma}

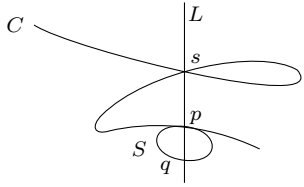
\begin{figure}[h]
  \begin{center}
    \begin{tikzpicture}[xscale=0.5,yscale=0.3,baseline=(b.base)]
    \draw (0,-1) -- (0, 7);
    \node (b) at (0,2){};

    \draw (0,0) ..controls ++(170:1) and ++(170: 1).. (0,1.5)
    .. controls ++(-10:1) and ++(-10:1) .. (0,0);
    \draw (2,0.5) ..controls (0.5,1.7) and (-1,1.7) .. (-2,1.3)
    .. controls (-2.4,1.1) and (-2.5,1.5) .. (-2.2,2)
    .. controls (-1,4) and (2,5) ..(3,4)
    .. controls (4,2) and (-3,4.9) ..(-4,6);

    \draw (-1.2,0.5) node{\scriptsize{$S$}};
    \draw (-0.5,-0.3) node{\scriptsize{$q$}};
    \draw (0.3,6.5) node{\scriptsize{$L$}};
    \draw (0.3,1.9) node{\scriptsize{$p$}};
    \draw (0.3,4.5) node{\scriptsize{$s$}};
    \draw (-4.5,6) node{\scriptsize{$C$}};
    \end{tikzpicture}%
  \end{center}
  \caption{The situation of Lemma~\ref{lemma:NonExistenceCubicConicConfiguration}}
  \label{figure:NonExistenceCubicConicConfiguration}
\end{figure}

\begin{proof}
  First of all, let us prove that there exists a point $q$, distinct from $p$, on $S\cap L$, where $L$ is the line going through $p$ and $s$.
  If $S\cap L=\{p\}$ were true, then $S$ and $L$ would be tangent at $p$, and since $C$ and $S$ are tangent at $p$, also $L$ and $C$ would be tangent.
  Hence, we would have $3=C\cdot L\geq I_p(C,L)+I_s(C,L)\geq 4$, since $s$ is a singularity of $C$.
  This is a contradiction, so there exists a point $q$, distinct from $p$, in the intersection of $L$ and $S$.

  We can fix coordinates such that $p=[0:0:1]$, $q=[1:0:0]$, and that the conic $S$ is given by the zero set of $G=xz-y^2$.
  Hence, the line $L$ is given by $y=0$.
  So we can write $s=[\lambda:0:1]$, with $\lambda\in\CCC^*$ since $s\neq p$ (as $C$ is singular at $s$ and smooth at $p$).
  With the linear change of coordinates $[x:y:z]\mapsto[x:\sqrt{\lambda}y:\lambda z]$, the conic and the points $p$ and $q$ are preserved, but $s$ is mapped onto $[1:0:1]$.
  Hence, we can assume that $\lambda=1$.
  To help with the proof, remark that we have shown that we can assume the following coordinates, and that we want to show that Figure~\ref{figure:NonExistenceCubicConicConfiguration} is accurate.
  \begin{center}
    \begin{tabular}{>{$}l<{$}  >{$}l<{$} }
      L: y=0, & q=[1:0:0],\\
      S: xz-y^2=0, & p=[0:0:1],\\
       & s=[1:0:1].\\
    \end{tabular}
  \end{center}
  Let $\psi: k[x,y,z]_3\to k[u,v]_6$ be the linear map $F\mapsto F(u^2,uv,v^2)$.
  We have \[\psi^{-1}(u^6)=\{x^3+(xz-y^2)(ax+by+cz) \sd a,b,c\in k\}.\]
  So our cubic $C$ is given by a polynomial $F$ that is of this form and we have $F(x,0,1)=x(x^2+ax+c)$, which has a double root in $x-1$ since $F$ is singular at $s$.
  So $x^2+ax+c=(x-1)^2$ and thus $a=-2$ and $c=1$.
  Our $F$ is therefore given by \[F = x^3+(xz-y^2)(-2 x+by+z).\]
  It is singular, hence all its derivatives must be zero at $s$, including in $y$-direction: \[\frac{\partial F}{\partial y}(q)=b,\] hence $b=0$.

  The coordinate change $x\mapsto x+ z$ maps $[0:0:1]$ onto $s$, we see that $[0:0:1]$ is not a cusp of $F(x+ z,y,z)$: \[F(x+ z,y,z)=z\,(x^2+ y^2)+x^3-2\, xy^2.\]
\end{proof}

\begin{lemma}\label{lemma:NonExistenceWhereMIs4}
  There is an upper bound $N(3,12)\leq 18$.
  In particular, there is no polynomial of bidegree $(3,12)$ with a singularity of type $A_{19}$ or $A_{20}$.
\end{lemma}

\begin{proof}
  Let us assume that a polynomial of bidegree $(3,12)$ with an $A_k$-singularity and $k\geq 19$ exists.
  If the $(3,12)$-divisor $C_0$ of the polynomial is reducible, we already know that such a singularity cannot exist by Lemma~\ref{lemma:ReducibleDivisorsNotInteresting}.
  So we assume that $C_0$ is irreducible.
  By Lemma~\ref{lemma:GenusUpperBoundForIrreducibleDivisors}, we know that $k\leq 20$.
  We assume that $C_0$ has an $A_k$-singularity at a point $s_0$, where $k\in\{19,20\}$.
  So we consider the $3$-configuration $\mathcal{C}_0=(C_0,S_-,s_0,p_0)_4$, which is disjoint because $C_0\sim3S_+$ is irreducible.
  Let $\pi:\FFF_{4}\dashrightarrow\FFF_{m'}$ be a singular $\mathcal{C}_0$-chain of $n=9$ links.
  Note that $n-m=5$ is odd, that $9<10=\ceil{\frac{k}{2}}$, and that $C_0^2-3n=36-27=9<17$.
  Together with Lemma~\ref{lemma:SpecialInformationAfterSingularChain}, we find that $\mathcal{C}_n=\pi_*(\mathcal{C}_0)$ is equipped with \[[9,5,K,9;1],\] where $K=1$ if $k=19$, and $K=2$ if $k=20$.
  Knowing the self-intersection of the $3$-section $C_n$ and the section $S_n$ in $\FFF_1$, we find $C_n\sim 3S_-+3f$ and $S_n\sim S_-+3f$.
  As $C_n$ is irreducible, it does not intersect the $(-1)$-curve $S_-$. In particular, $p_n\notin S_-$ and $s_n\notin S_-$.

  Let $\rho:\FFF_{1}\to\PPP^2$ be the contraction of $S_-$ onto a point $q\in\PPP^2$.
  Let $C=\rho_*(C_n)$, $S=\rho_*(S_n)$, $p=\rho(p_n)\neq q$, and $s=\rho(s_n)\neq q$.
  We have $C^2=C_n^2=9$, so $C$ is a cubic not going through $q$.
  Since $S_n\cdot S_-=2$, we have $m_q(S)=2$ and so $S^2=S_n^2+4=9$. Hence, $S$ is a cubic with a singular point at $q$.
  They intersect only at $p$, because $I_p(S,C)=I_{p_n}(S_n,C_n)=9$.
  Since $s_n$ and $p_n$ are cofibered, the distinct points $s$, $p$, and $q$ are collinear.
  Recall that $C$ has a singular point at $s$ (of type $A_1$ or $A_2$, depending on $k$).
  To summarize, we have two singular cubics that intersect at exactly one point, and this point is collinear with the two singular points. \begin{itemize}
    \item Lemma~\ref{lemma:NonExistenceOfCuspidalCubicWithSingularCubic} contradicts this situation if one of the singularities is a cusp.
    \item If one of the singularities is a node, then Lemma~\ref{lemma:NonExistenceNodalCubicWithCollinearSingularCubicInP2} contradicts our situation.
  \end{itemize}
  The lemma is proved.
\end{proof}

\begin{lemma}\label{lemma:NonExistenceOfCuspidalCubicWithSingularCubic}
  Let $C\subset\PPP^2$ be an irreducible cubic curve with a cusp.
  Let $D\subset\PPP^2$ be another irreducible cubic that intersects $C$ in exactly one point that is not the cusp. Then, $D$ is smooth.
\end{lemma}

\begin{proof}
  If $C$ has a cusp, with a linear transformation we can assume that the cubic curve $C$ is given by the zero set of $F=x^2z-y^3$, which has a cusp at $q=[0:0:1]$.
  It can be parametrized by $[u:v]\mapsto [u^3:u^2v:v^3]$.
  We consider the linear map \begin{align*}
    \varphi:\CCC[x,y,z]_3&\to \CCC[u,v]_9\\
    G&\mapsto G(u^3,u^2v,v^3)
  \end{align*} and notice that both source and target are of dimension 10.
  The equation of the curve being $F$, we have $\ker\varphi=\CCC\cdot F$ and so $\Ima\varphi$ has dimension $9$.
  One can check that $y^3$ and $x^2z$ are both sent to $u^6v^3$ and that $uv^8\notin \Ima(\varphi)$, hence \[\Ima(\varphi)=\{a_0u^9+\cdots+a_7u^2v^7+a_9v^9\}.\]
  As the parametrization of $F$ sends $[0:1]$ to the cusp $q=[0:0:1]$, any point on $C$ outside of $q$ is of the form $[1:-\alpha]\mapsto [1:-\alpha:-\alpha^3]$ for some $\alpha\in\CCC$.
  As the intersection point $p$ of $C$ and $D$ is not the cusp, we have $p=[1:-\alpha:-\alpha^3]$ for some $\alpha\in\CCC$.
  We want to prove that the following picture is accurate:
  \begin{center}
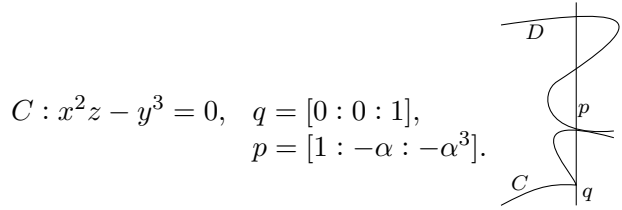
    \begin{tabular}{>{$}l<{$}  >{$}l<{$} }
      C: x^2z-y^3=0, & q=[0:0:1],\\
      & p=[1:-\alpha:-\alpha^3].
    \end{tabular}\begin{tikzpicture}[xscale=0.5,yscale=0.3,baseline=(b.base)]
    \draw (0,-2) -- (0, 7);
    \node (b) at (0,1){};

    \draw (1,1) .. controls (-1.8,2.3) and (-0.1,-0.5) ..(0,-1.1)
    .. controls (-0.8,-1) and (-1.2,-1.3) ..(-2,-2);

    \draw (1,1.3) ..controls (-0.9,1) and (-1,3) .. (-0.5,3.5)
    .. controls (1.8,6) and (2,7) .. (-2,6);

    \draw (-1.5,-1) node{\scriptsize{$C$}};
    \draw (0.3,-1.5) node{\scriptsize{$q$}};
    \draw (0.2,2.2) node{\scriptsize{$p$}};
    \draw (-1.1,5.7) node{\scriptsize{$D$}};
    \end{tikzpicture}
  \end{center}
  Now, let us find a polynomial $G\in \CCC[x,y,z]_3$ that intersects $F$ only at $p$.
  Inserting the parametrisation of $F$ into $G$, we get $\varphi(G)$.
  This must have only one zero at $[1:-\alpha]$.
  So we have, up to multiplication of $G$ with a scalar, that \[G(u^3,u^2v,v^3)=(\alpha u+ v)^9\in\Ima\varphi.\]
  So $(\alpha u+v)^9$ is in the image of $\varphi$ and therefore no term with $uv^8$ may appear:
  This is only possible if $\alpha=0$.
  This implies that the unique intersection point of $F$ and $G$ is $p=[1:0:0]$.
  As $G(u^3,u^2v,v^3)=v^9$, one finds $G=z+\lambda(x^2z-y^3)$ for some $\lambda\in\CCC$.
  Since $G$ is irreducible, $\lambda$ is not zero.
  Let $\mu$ be such that $\mu^6=\frac{1}{\lambda}$.
  The coordinate change $[x:y:z]\mapsto[\mu^3x:\mu^2y:z]$ preserves $C$ and the points $p$ and $q$, and it maps $G$ onto $z^3+(x^2z-y^3)$. Therefore, we can assume that $\lambda=1$.

  One concludes the proof by checking that not all partial derivatives of $G$ can be zero at the same point.
  So $G$ is not singular.
\end{proof}

\begin{lemma}\label{lemma:NonExistenceNodalCubicWithCollinearSingularCubicInP2}
  Let $C\subset\PPP^2$ be an irreducible cubic with a node.
  Let $D\subset\PPP^2$ be an irreducible cubic that intersects $C$ in only one point that is not the node.
  Then, $D$ is smooth on the line connecting the node and the intersection point.
\end{lemma}
\begin{proof}
  As $C$ has a node, by applying a linear transformation we can assume that $C$ is given by the zero set of $F=xyz-x^3-y^3$, which has a node at $[0:0:1]$.
  This cubic can be parametrized by $[u:v]\mapsto[u^2v:uv^2:u^3+v^3]$.
  Consider the linear map that is given by evaluating this parametrization, \begin{align*}
    \varphi:\CCC[x,y,z]_3&\to \CCC[u,v]_9\\
    G&\mapsto G(u^2v,uv^2,u^3+v^3).
  \end{align*}
  As $F$ is the equation of $C$, we have $\ker\varphi=\CCC\cdot(xyz-x^3-y^3)$, so the kernel is of dimension $1$.
  As $k[x,y,z]_3$ has dimension $10$, the image is of dimension $9$.\\
  Now, we show that \[\Ima\varphi=\{g\in k[u,v]_9 \sd g(1,0)=g(0,1)\}\] holds.
  For the direction ``$\subset$'' write $g(u,v)=G(u^2v,uv^2,u^3+v^3)$ and so we have $g(1,0)=G(0,0,1)=g(0,1)$.
  The condition ``$g(1,0)=g(0,1)$'' is equivalent to the coefficients of $u^9$ and $v^9$ being equal, hence the vector space on the right is also of dimension $9$ and equality holds.

  Since $[1:0]$ and $[0:1]$ are both mapped onto $[0:0:1]$, any point on $C$ outside of the node is of the form $[\alpha:-1]\mapsto [-\alpha^2:\alpha:\alpha^3-1]$ for some $\alpha\in\CCC^*$.
  Since the intersection point of $C$ and $D$ is not the node, this point can be written as $q=[-\lambda^2: \lambda : \lambda^3-1]$ for some $\lambda\in\CCC^*$.
  Now, let us find a $G\in k[x,y,z]_3$ that meets $F$ only at $q$.
  Hence, $\varphi(G)$ must have only one zero at $[\lambda:-1]\in\PPP^1$.
  Up to multiplying $G$ with a scalar, we get   \[g(u,v):=G(u^2v,uv^2,u^3+v^3)=(u+\lambda v)^9\in \Ima\varphi\] and so $1=g(1,0)=g(0,1)=\lambda^9$. So $\lambda$ is a ninth root of 1.

  We look at \[ \begin{split} H=&z^3+z^2 9\lambda(x+\lambda^7y)\\
    &+z\left(9\lambda^2 x^2(-\lambda^6+4) + 3xy(28\lambda^3-1)+9\lambda y^2(4\lambda^6-1)\right)\\
    &+9\lambda xy\left(x\,(-4\lambda^6+14\lambda^3-1)+\lambda y(-\lambda^6+14\lambda^3-4)\right)\\
    &+84\lambda^3 y^3(\lambda^3-1)
 \end{split}\]
  and check that $\varphi(H)=(u+\lambda v)^9$.
  Therefore, we have \[G=H+a(x\,y\,z-x^3-y^3)\] for some $a\in \CCC$.

  Let $L$ be the line connecting $[0:0:1]$ (the node of $F$) and $q=[-\lambda^2: \lambda : \lambda^3-1]$ (the common point of $F$ and $G$), which is thus given by $x+\lambda y=0$.
  We want to see that $G$ is smooth on this line $L$.
  Let us assume that $D$ has a singular point on $L$.
  Hence $G$ has a double zero on $L$ at a point $s$.

  Now, we find the value of $a$.
  We note that \[G\mid_{x=-\lambda y}=P_1\,P_2,\] where \begin{align*}
    P_1&=y(\lambda^3-1)-\lambda\,z,\\
    P_2&= \left( -28\,{\lambda}^{6}+56\,{\lambda}^{3}+a-1 \right) {y}^{2}+ 8\,\lambda\,\left( -{\lambda}^{6}+1 \right) yz-{\lambda}^{8}{z}^{2}.
  \end{align*}
  The zero of $P_1$ corresponds to $q$, so $P_2$ has to be a square because $G$ has a double zero on $x=-\lambda y$ at some point $s$.
  The discriminant of $P_2$, multiplied by $\lambda$, is \[4(a-12\lambda^6+72\lambda^3-33)\] and so we find that \[a=12\,{\lambda}^{6}-72\,{\lambda}^{3}+33\] and then \[P_2=-\lambda^6\left(4\,y\,(\lambda^3-1)-z\,\lambda\right)^2.\]
  So its zero is at $s=[-\lambda^2:\lambda:4(\lambda^3-1)]$.

  Note that $\lambda^9-1=(\lambda^3-1)(\lambda^6+\lambda^3+1)$.
  Inserting for $\lambda$ the values with $\lambda^3=1$, we find that $G$ is reducible. (For instance, if $\lambda=1$ we find that $G=(3x+3y+z)^3$.)
  Therefore, $\lambda$ is not a third root of $1$ and so $\lambda^6+\lambda^3+1=0$.

  We insert $s$ into the differentials of $G$ and find, using $\lambda^6+\lambda^3+1=0$, that they are not zero, a contradiction to $s$ being a singular point of $D$: \begin{align*}
      \frac{dG}{dx}\left(-\lambda^2,\lambda,4(\lambda^3-1)\right)&=27\,\lambda(-2\lambda^6+\lambda^3+1)\\
      &= -3\cdot27\lambda^7\neq0.
    \end{align*}
\end{proof}

\begin{remark}\label{remark:UpperBound}
  We have shown that $N(3,b)$ has the following upper bound (UB), where the bound for $N(3,6)$ comes from the genus upper bound in Remark~\ref{remark:UpperBoundsReducibleIrreducible}:
  \begin{center}
    \begin{tabular}{ r || >{$}c<{$} | >{$}c<{$} | >{$}c<{$}| >{$}c<{$}| >{$}c<{$}| >{$}c<{$}| >{$}c<{$} | >{$}c<{$} | >{$}c<{$} |>{$}c<{$} | }
      $b$ &  3 & 4 & 5 & 6 & 7 & 8 & 9 & 10 & 11 & 12\\ \hline\hline
      UB for $N(3,b)$ &   &  &  & 8 &  &  & 13 &  & &18 \\
    \end{tabular}
  \end{center}
\end{remark}

The same method can also be applied to prove the non-existence of a polynomial of bidegree $(3,b)$ with an $A_k$-singularity, where type and singularity are one of the following: \begin{enumerate}
  \item\label{item:3-4--8} $(3,4)$ with $A_8$,
  \item\label{item:3-5--8} $(3,5)$ with $A_8$,
  \item $(3,7)$ with $A_{11}$ and $A_{12}$,
  \item $(3,8)$ with $A_{13}$ and $A_{14}$,
  \item $(3,10)$ with $A_{16}$, $A_{17}$, and $A_{18}$,
  \item $(3,11)$ with $A_{18}$, $A_{19}$, and $A_{20}$.
\end{enumerate}
For instance, \ref{item:3-5--8} implies \ref{item:3-4--8}.
However, to determine that the obtained configuration in $\PPP^2$ does not exist is tedious (and gets more tedious with increasing $b$), and it does not add any value, since in Section~\ref{section:KnotTheory} we present a knot theoretic theorem, which has as a consequence that the polynomials of the above list do not exist.

\section{Let's Tie the Knot}\label{section:KnotTheory}

In Section~\ref{subsection:KnotTheory--Detour} we present a result of knot theory and explain how it is used to obtain an upper bound for $N(3,b)$.
With this, we can finally finish the proof of Theorem~\ref{theorem:TheTheoremTable} in Section~\ref{subsection:KnotTheory--TheProof}, by marrying the lower bound from Section~\ref{section:ToBe} with the obtained upper bound.
Therefore, despite the Shakespearean section titles and the authors first name, this paper does not end in utter tragedy but with a happy end.

\subsection{Detour to knot theory}\label{subsection:KnotTheory--Detour}

A \textit{knot} is a smooth and oriented embedding of $S^1$ into $S^3$.
A \textit{link} is the disjoint union of finitely many knots.

For any positive $r\in\RRR$, let $S_r^3$ denote the sphere of dimension $3$ with radius $r$  embedded in $\CCC^2$ as $S_r^3=\{(x,y)\in\CCC^2\mid |x|^2+|y^2|=r^2\}$.

\begin{definition}
  Let $C\subset\CCC^2$ be a curve given by a polynomial $F$ with a singularity at $(0,0)$.
  Its \textit{link of singularity} is the transversal intersection $C\pitchfork S_r^3$ for some $r>0$ small enough.
  For $R>0$ large enough we say that $C\pitchfork S_R^3$ is the \textit{link at infinity} of $C$.
\end{definition}

\begin{definition}
  A \textit{torus link} $T_{p,q}$ for two integers $p,q>0$ is defined by the embedding $S^1\to S^1\times S^1,t\mapsto(t^p,t^q)$, where $S^1=\{x\in\CCC \mid \abs{x}=1\}$.
  It is called \textit{torus knot} if $p$ and $q$ are coprime.
\end{definition}

Note that this corresponds to the intersection of the zero set of $y^p-x^q$ in $\CCC^2$ with the $3$-sphere $S^3$ in $\CCC^2$.
So if there are local analytical coordinates such that a curve is locally given by $y^p-x^q=0$, then the curve has exactly the torus link $T_{p,q}$ as its link of singularity.
Hence, a singularity of type $A_k$ corresponds to a torus link $T_{2,k+1}$.
On the other hand, the link at infinity of the curve given by $y^b-x^a=0$ is $T_{a,b}$.

\begin{definition}
  A \textit{cobordism} between two links $K$ and $T$ is an oriented surface $C$ in $S^3 \times [0, 1]$ with boundary $K \times \{0\}\, \cup\, T \times \{1\}$ such that the induced orientation agrees with the orientation of $T$ and disagrees with the orientation of $K$.

  A cobordism is called \textit{algebraic} if it is given by the intersection of a smooth algebraic curve in $\CCC^2$ with the closure of $B_R^4\setminus B_r^4\subset\RRR^4$, which is isomorphic to $S^3\times [0,1]$, where $0<r<R$.
\end{definition}

In what follows we are interested in the existence of a cobordism between the link of singularity of some curve and its link at infinity.

Let $C\subset\AAA^2$ be a curve given by a polynomial $F$ of bidegree $(a,b)$ with a singularity $y^p-x^q=0$ at $(0,0)$, so its link of singularity $K$ is the torus link $T_{p,q}$.
We want to see that there exists a cobordism from $K$ to $T_{a,b}$.

We choose $r>0$ small enough such that $C\pitchfork S_r^3$ is the link of singularity $K$ of $C$.
Set $G=F+t+s(x^a+y^b)$ for some $s,t\in\CCC$ small with $|s|,|t|<<r$.
Then, $V(G)$ is smooth (for general $s,t$), and $V(G)\pitchfork S_r^3$ is isotopic to $C\pitchfork S_r^3$, that is $K$.
The polynomial $G$ was chosen in such a way that $x^a$ and $y^b$ appear with non-zero coefficients. Hence, the link of infinity of $G$ is $T_{a,b}$.
This means that $V(G)$ gives rise to an algebraic cobordism between $K$, that is $T_{p,q}$, and $T_{a,b}$.

Therefore, the existence of a polynomial of bidegree $(3,b)$ with a singularity of type $A_k$ implies the existence of an algebraic cobordism from $T_{3,b}$ to $T_{2,k+1}$.
(Whether also the converse implication holds is not known, however, Theorem~\ref{theorem:TheTheoremTable} provides evidence that it might, see Remark~\ref{remark:UpperBoundOfKnotAchieved}.)

In this way, the following theorem gives an upper bound for $N(3,b)$.

\begin{lemma}[Theorem 1 in \cite{feller_2016}]\label{lemma:KnotUpperBound}
  Let there be two positive torus knots $T_{2,k+1}$ and $T_{3,b}$ of braid index $2$ and $3$, respectively.
  There exists an algebraic cobordism between $T_{2,k+1}$ and $T_{3,b}$ if and only if $k+1 \leq\frac{5b-1}{3}$.
\end{lemma}

In fact, the proof in \cite{feller_2016} also works if $T_{2,k+1}$ is only a \textit{link}, since the results used in the proof hold for links, too (Lemma~6 and Proposition~22 in \cite{feller_2016}).
However, it does not work for $T_{3,3m}$, since a torus \textit{knot} $T_{p,q}$ needs $p$ and $q$ to be coprime.

\begin{remark}\label{remark:UpperBoundKnotTheory}
  By inserting $b=3m-r$ with $r\in\{1,2\}$ into Lemma~\ref{lemma:KnotUpperBound}, we find that the maximal $k$ such that there is an algebraic cobordism from $T_{2,k+1}$ to $T_{3,b}$ is the maximal $k$ with $k\leq 5m-r-1-\frac{2r+1}{3}$.
  \begin{itemize}
    \item If $r=1$ then $k=5m-3$,
    \item if $r=2$ then $k\leq5m-4-\frac{2}{3}$, hence $k=5m-5$.
  \end{itemize}

  This gives the following upper bound (UB) for $N(3,b)$:
  \begin{center}
    \begin{tabular}{ r || >{$}c<{$} | >{$}c<{$} | >{$}c<{$}| >{$}c<{$}| >{$}c<{$}| >{$}c<{$}| >{$}c<{$} | >{$}c<{$} | >{$}c<{$} |>{$}c<{$} | }
      $b$ &  3 & 4 & 5 & 6 & 7 & 8 & 9 & 10 & 11 & 12\\ \hline\hline
      UB for $N(3,b)$&   & 5 & 7 &  & 10 & 12 &  & 15 & 17 & \\
    \end{tabular}
  \end{center}
\end{remark}

Returning to Section~\ref{section:Introduction}, we observe that using $N(3,b)$ and Proposition~\ref{proposition:Orevkov}, we cannot improve Orevkov's lower bound of $\alpha\geq \frac{7}{6}$:

\begin{lemma}\label{lemma:NotAsGoodAsOrevkov}
  Every $b\geq 1$ satisfies $\frac{2\left(N(3,b)+1\right)}{3b} < \frac{7}{6}$.
\end{lemma}

\begin{proof}
  For $b=1,2$ the statement follows from Example~\ref{example:N1b} and Lemma~\ref{lemma:BoundsForThe2SectionCase}, so we assume $b\geq 3$.
  If $b$ is no multiple of $3$, we have with Lemma~\ref{lemma:KnotUpperBound} that \[\frac{2(N(3,b)+1)}{3b}\leq \frac{2(5b-1)}{9b}=\frac{10}{9}-\frac{2}{9b}<\frac{10}{9}<\frac{7}{6}.\]
  If $b$ is a multiple of $3$, we write $b=3m-3$ for some $m\geq 2$, and hence $b+1=3m-2$ is not a multiple of $3$, so Remark~\ref{remark:UpperBoundKnotTheory} gives \[N(3,b)\leq N(3,b+1)\leq 5m-5.\]
  We find that $\frac{2(N(3,b)+1)}{3b}\leq\frac{2(5m-4)}{3(3m-3)}$, which is strictly less than $\frac{7}{6}$ for $m\geq 6$, and that corresponds to $b\geq 15$.
  Theorem~\ref{theorem:TheTheoremTable} concludes the proof for the remaining cases $b=3,6,9,12$.
\end{proof}

\subsection{Proof of Theorem~\ref{theorem:TheTheoremTable}}\label{subsection:KnotTheory--TheProof}

\begin{proof}[Proof of Theorem~\ref{theorem:TheTheoremTable}]
  Let us look at the souvenirs collected on our journey.
  In Remark~\ref{remark:LowerBound} we have found lower bounds (LB), and the upper bounds (UB) from  Remarks~\ref{remark:UpperBound} and~\ref{remark:UpperBoundKnotTheory} combined give the following values for $N(3,b)$:
  \begin{center}
    \begin{tabular}{ r || >{$}c<{$} | >{$}c<{$} | >{$}c<{$}| >{$}c<{$}| >{$}c<{$}| >{$}c<{$}| >{$}c<{$} | >{$}c<{$} | >{$}c<{$} |>{$}c<{$} | }
      $b$ &  3 & 4 & 5 & 6 & 7 & 8 & 9 & 10 & 11 & 12\\ \hline\hline
      LB for $N(3,b)$& 3 & 5 & 7 & 8 & 10 & 12 & 13 & 15 & 17 & 18\\\hline
      UB for $N(3,b)$& 3  & 5 & 7 & 8 & 10 & 12 & 13 & 15 & 17 & 18\\\hline\hline
      $N(3,b)$& 3 & 5 & 7 & 8 & 10 & 12 & 13 & 15 & 17 & 18\\
    \end{tabular}
  \end{center}
\end{proof}

\begin{remark}\label{remark:UpperBoundOfKnotAchieved}
  Theorem~\ref{theorem:TheTheoremTable} shows that in the cases studied, the upper bound obtained from Lemma~\ref{lemma:KnotUpperBound} is always achieved.
\end{remark}

\end{document}